\documentclass[times,final,11pt,compress]{elsarticle}

\usepackage{jcomp}

\usepackage{caption, subcaption}%

\usepackage{amsmath}
\usepackage{amssymb}
\usepackage{amsfonts}
\usepackage{amsthm}
\usepackage{booktabs}
\usepackage{mathrsfs}
\usepackage{epsfig}
\usepackage{enumerate}
\usepackage{enumitem}
\usepackage{extarrows}
\usepackage{flafter}
\usepackage{float} 
\usepackage{graphicx,wrapfig}
\usepackage{hyperref}
\usepackage{leftidx}
\usepackage{lscape}
\usepackage{multicol}
\usepackage{multirow}
\usepackage{pgfplots}
\usepackage{setspace}
\usepackage{tikz}
\usepackage{xcolor}
\definecolor{mygreen}{rgb}{0,0.45,0}
\usetikzlibrary{arrows,shapes,chains}
\usetikzlibrary{graphs, positioning, quotes, shapes.geometric}

\usepackage{bm}

\allowdisplaybreaks

\newtheorem{theorem}{Theorem}[section]
\newtheorem{proposition}{Proposition}[section]
\newtheorem{lemma}{Lemma}[section]

\newtheorem{remark}{Remark}

\theoremstyle{remark}
\newtheorem{example}{\bf Example}[]

\begin{document}
	
	\begin{frontmatter}
		\title{{\bf GQL-Based Physical-Constraint-Preserving High-Order Finite Difference Schemes for Special Relativistic Hydrodynamics in Arbitrary Dimensions}}
		\author[1]{Linfeng Xu}
		\ead{xulf2022@mail.sustech.edu.cn}
		\author[2]{Shengrong Ding}
		\ead{dingshr7@mail.sysu.edu.cn}
		\author[1,3]{Kailiang Wu\corref{cor1}}
		\ead{wukl@sustech.edu.cn}
		
		\address[1]{Department of Mathematics, Southern University of Science and Technology, Shenzhen, Guangdong, 518055, China}
		\address[2]{School of Science, Sun Yat-Sen University, Shenzhen, Guangdong, 518107, China}
                \address[3]{Shenzhen International Center for Mathematics, Southern University of Science and Technology, Shenzhen, Guangdong, 518055, China}

		\cortext[cor1]{Corresponding author.}

\begin{abstract}
	High-order accurate simulations of special relativistic hydrodynamics (RHD) are prone to numerical breakdown if intrinsic physical constraints (positive rest-mass density/pressure and subluminal velocity) are violated near strong discontinuities. In this work, we develop a robust and efficient physical-constraint-preserving (PCP) flux-limiting framework for high-order schemes, using finite-difference WENO as a representative example. By leveraging the geometric quasilinearization (GQL) representation, which equivalently reformulates the nonlinear RHD constraints into a family of linear inequalities, we integrate a Zalesak-type Flux-Corrected Transport (FCT) update into a scalar-style limiter that acts directly on conservative variables.	A critical innovation is the explicit, non-iterative determination of limiting parameters via a rational stereographic parameterization of the GQL normal vector. This technique transforms the required worst-case minimization over auxiliary variables into a generalized Rayleigh-quotient formulation, allowing the optimal parameters to be obtained by solving small symmetric eigenvalue problems ($2\times2$ in 1D; $(d+1)\times(d+1)$ in $d$ dimensions). Relaxed variants are further introduced to reduce computational costs in multidimensions while retaining the PCP guarantee. Extensive numerical benchmarks ranging from 1D to 3D, including ultra-relativistic Riemann problems and astrophysical jets, demonstrate that the proposed method robustly enforces physical admissibility, sharply resolves discontinuities, and maintains design-order accuracy for smooth solutions.
\end{abstract}

		
		
	\end{frontmatter}
	
	\singlespacing

	
\section{Introduction}
\setcounter{equation}{0}

Relativistic hydrodynamics (RHD) provides a fundamental description of fluid motion at velocities approaching the speed of light and plays a central role in astrophysics, high-energy physics, and cosmology. 
The governing equations feature strong nonlinearities and, in many practically relevant regimes, extremely large Lorentz factors; as a result, analytical solutions are rare and high-fidelity numerical simulation becomes indispensable for elucidating the dynamics of relativistic flows. 
Compared with the nonrelativistic Euler equations, RHD exhibits a more intricate and nonlinearly coupled relationship between conservative and primitive variables, which poses additional challenges for designing robust, accurate, and efficient numerical methods.

Early numerical investigations date back to \cite{may1966hydrodynamic,wilson1972numerical}, where finite-difference discretizations augmented by artificial viscosity were employed in Lagrangian or Eulerian coordinates. Over the past decades, a wide range of high-resolution and high-order methods have been developed for RHD, including finite-volume, finite-difference, and discontinuous Galerkin (DG) approaches; see, e.g., \cite{chen2021second, radice2012thc, WuTang2015, radice2011discontinuous, zhao2013runge}. To further enhance the resolution of discontinuities and multiscale structures, adaptive mesh refinement and adaptive moving-mesh techniques have also been incorporated into RHD simulations \cite{2006Zhang,he2012adaptive1}.

A central difficulty in RHD computations is the strict enforcement of intrinsic physical constraints: positivity of the rest-mass density and pressure, and the subluminal constraint on the fluid velocity (i.e., $|{\bm v}|<1$ with the light speed $c=1$). 
These constraints define the admissible state set of the system. 
Numerical violations (negative density/pressure or superluminal velocity) are not only nonphysical but may also cause immediate breakdown of the computation, particularly in extreme regimes characterized by large Lorentz factors, near-vacuum conditions, or strong shocks. 
In practice, such failures are sometimes handled in an ad hoc manner by restarting with more diffusive schemes and/or smaller CFL numbers until admissibility is recovered \cite{2006Zhang, hughes2002three}. 
This lack of rigor and predictability motivates the development of high-order methods that guarantee admissibility by construction, commonly referred to as physical-constraints-preserving (PCP) or, more generally, invariant-domain-preserving (IDP) schemes.

In the broader setting of hyperbolic conservation laws, the design of provably PCP/IDP schemes has advanced substantially, with two major limiter paradigms emerging. 
The first class consists of scaling-type limiters, pioneered by Zhang and Shu, and successfully applied to scalar conservation laws \cite{zhang2010}, the nonrelativistic Euler equations \cite{zhang2012positivity, zhang2010b}, and the compressible Navier--Stokes equations \cite{ZHANG2017301}. 
The second class enforces admissibility through flux correction, i.e., nonlinear blending of high-order and low-order numerical fluxes. 
Representative examples include parametrized maximum-principle/positivity-preserving flux limiters for scalar laws \cite{Xu2014, Liang2014}, convection-dominated diffusion \cite{jiang2013parametrized}, and the compressible Euler equations \cite{Xiong2016}. 
In addition, Hu, Adams, and Shu \cite{Hu2013} proposed a sufficient positivity-preserving flux limiting condition for the Euler equations via a convex decomposition of cell averages, an approach later generalized to special RHD by Wu and Tang \cite{WuTang2015}.

The flux-correction philosophy can be traced back to the Flux-Corrected Transport (FCT) method of Boris and Book \cite{boris1973flux, book1975flux, boris1976flux}. 
Zalesak extended FCT to fully multidimensional settings \cite{zalesak1979fully} and developed a structured-grid limiter for scalar conservation laws \cite{zalesak2012design}, laying foundations for later IDP/PCP developments. 
Recent work has further clarified the close connection between Zalesak-type FCT limiters and parametrized flux limiters; see, e.g., \cite{Xu2014,wu2025high}. 
Kuzmin and collaborators systematically extended FCT to multidimensional and high-order discretizations \cite{kuzmin2001positive, kuzmin2004high} and proposed monolithic convex limiting techniques for hyperbolic systems \cite{kuzmin2020monolithic}. 
Other notable FCT-type variants include the Point-Average-Moment PolynomiAl-Interpreted (PAMPA) schemes \cite{abgrall2024bound, abgrall2025bound, abgrall2025novel}. 
Another influential line is the convex limiting framework of Guermond, Popov, and Tomas \cite{guermond2017invariant, guermond2018second, guermond2019invariant}, developed for discretization-independent IDP mechanisms and applied across continuous finite elements and other discretization frameworks. 
For comprehensive discussions of FCT/IDP principles and algorithms, we refer to \cite{kuzmin2012flux, wu2025high} and the references therein.

Applying PCP/IDP ideas to RHD is nontrivial. 
Although the equations are conservative, neither the fluxes nor the wave speeds admit simple closed-form expressions purely in terms of conservative variables, since they depend on primitive variables obtained through a nonlinear recovery procedure; see Section~\ref{sec: RHD}. 
Despite these difficulties, major progress has been achieved for special RHD. 
Wu and Tang \cite{WuTang2015} derived the first explicit equivalent characterization of the admissible state set, later generalized to general equations of state (EOS) \cite{WuTang2017ApJS}, and developed high-order PCP finite-difference WENO schemes with flux-corrected limiters. 
This foundation has supported extensions to unstructured finite-volume methods \cite{chen2022physical}, discontinuous Galerkin (DG) methods \cite{QinShu2016, WuTang2017ApJS}, and invariant-region/minimum-entropy-principle approaches \cite{WuMEP2021, cui2025local}. 
Related PCP frameworks have also been established for general relativistic hydrodynamics \cite{wu2017design} and coupled with the oscillation-eliminating approach \cite{peng2025oedg} in \cite{cao2025robust}, and PCP techniques have been extended to relativistic MHD \cite{WuTangM3AS, WuShu2020NumMath} as well as nonrelativistic MHD \cite{Wu2017a, WuShu2018, WuShu2019}. 
For further background, we refer to the reviews \cite{marti2003numerical, Marti2015}, the textbook \cite{rezzolla2013relativistic}, and recent studies \cite{WuShu2019SISC, marquina2019capturing, xu2024high}.

More recently, the geometric quasilinearization (GQL) framework was introduced in \cite{wu2023geometric}, inspired by PCP analyses for (relativistic) MHD and related systems \cite{WuTangM3AS, Wu2017a, WuShu2018, WuShu2019, WuShu2020NumMath}. 
GQL converts nonlinear admissibility constraints into equivalent families of linear inequalities parameterized by auxiliary variables, offering a geometric route to PCP design. 
However, for special RHD a key practical bottleneck remains: enforcing the nonlinear constraint $q({\bf U})>0$ (encoding positive pressure and subluminal velocity) in a way that is both sharp and efficient. 
Achieving high resolution requires the flux-corrected limiting parameter $\theta\in[0,1]$ to stay as close to $1$ as possible, while guaranteeing admissibility at every interface and Runge--Kutta stage. 
Even within the GQL framework \cite{wu2023geometric}, the limiting step naturally leads to a worst-case optimization over an auxiliary variable $v_{*}$; existing strategies often rely on discrete sampling or iterative root-finding, whose cost and implementation complexity grow rapidly in multiple space dimensions and may become overly restrictive when the solution approaches the boundary of the admissible set.

To address these issues, we propose a new, efficient PCP flux-limiting framework for special RHD that avoids iterative searches over auxiliary variables in the limiting step. 
Building on the GQL characterization and the classical Zalesak-type FCT mechanism, our main contributions are:
\begin{itemize}
	\item \textit{GQL--FCT scalar-style PCP limiting.} We integrate the GQL family of linear inequalities into an FCT update to construct a scalar-style limiter that acts directly on conservative variables. 
	In contrast to classical characteristic FCT limiters for hyperbolic systems \cite{zalesak1979fully,zalesak2012design}, our GQL-based PCP limiter does \emph{not} require characteristic decomposition in the limiting step. 
	Furthermore, unlike the FCT approach using GQL with a specific convex decomposition for PAMPA schemes \cite{abgrall2024bound}, the present limiter avoids such a prior convex decomposition; our limiting approach is thus applicable to other high-order numerical frameworks, such as finite difference and finite volume schemes.
	
	\item \textit{Eigenvalue-based parameter estimation.} We resolve the computational bottleneck associated with the nonlinear constraint. 
	By introducing a rational (stereographic-type) parameterization of the GQL normal vector, we transform the worst-case minimization over GQL auxiliary variables into a generalized Rayleigh-quotient maximization. 
	Consequently, the limiter parameters are obtained from only a few small \emph{symmetric eigenvalue problems} ($2\times2$ in 1D; $(d{+}1)\times(d{+}1)$ in $d$ dimensions), yielding an explicit, non-iterative procedure. 
	For 2D and 3D, we further propose relaxed variants that substantially reduce the number of eigenvalue evaluations while retaining robustness (from 16 to 6 in 2D and from 64 to 9 in 3D).
	
	\item \textit{High-order PCP WENO schemes and numerical validation.} We couple the new limiter with standard fifth-order finite-difference WENO reconstruction and strong-stability-preserving Runge--Kutta time integration. 
	A series of demanding one- to three-dimensional benchmarks, including ultra-relativistic Riemann problems, multidimensional Riemann problems, shock--bubble interaction, and relativistic jet simulations, demonstrate robustness and high-order accuracy. 
	Comparisons with the classical Wu--Tang PCP limiter \cite{WuTang2015} indicate that the new estimator provides better resolution.
\end{itemize}

From the viewpoint of invariant-domain enforcement, convex limiting and its monolithic variants provide a general mechanism to restrict high-order (anti-diffusive) corrections so that the updated degrees of freedom remain in a prescribed convex admissible set; see, e.g., the convex limiting framework of Guermond et al.~\cite{guermond2018second, guermond2019invariant} and the monolithic convex limiting approach of Kuzmin~\cite{kuzmin2020monolithic}. 
In contrast, we exploit the specific geometry of the special-RHD admissible set in conservative variables: besides $D>0$, the nonlinear constraint $q({\bm U})=E-\sqrt{D^2+|{\bm m}|^2}>0$ is represented \emph{exactly} as the intersection of linear half-space inequalities ${\bf U}\cdot{\bm n}({\bm v}_*)>0$ for all ${\bm v}_*\in \mathbb{B}_1({\bf 0})$. 
This makes it possible to embed constraint enforcement into a scalar-style Zalesak-type FCT update by applying lower-bound limiting to the projected quantities ${\bf U}\cdot{\bm n}({\bm v}_*)$, acting directly on conservative variables and avoiding any problem-specific convex decomposition of states in the limiting step.

The remainder of this paper is organized as follows. 
Section~\ref{sec: RHD} introduces the governing equations and the GQL characterization of the admissible state set. 
Section~\ref{Sec: FD Review} outlines the base finite-difference WENO discretization. 
Section~\ref{Sec: BP analysis} details the construction of the GQL--FCT PCP flux limiter, the eigenvalue-based parameter estimation, and the proof of the PCP property. 
Numerical validations are presented in Section~\ref{Sec: Results}, followed by conclusions in Section~\ref{sec:conclusion}.


\section{Relativistic Euler system and admissible state set} \label{sec: RHD}
The $d$-dimensional special relativistic hydrodynamics system can be formulated as a system of hyperbolic conservation laws:
\begin{equation}\label{eq:RHD3D}
	\frac{\partial {\bf U}}{\partial t} + \sum_{i=1}^{d} \frac{\partial {\bf F}_i({\bf U})}{\partial x_i} = {\bf 0},
\end{equation}
where the conservative vector ${\bf U}$ and flux function ${\bf F}_i$ are defined as
\begin{align} 
	&{\bf U} = 
	\left( D, {\bm m}^\top, E \right)^\top 
	= 
	\left( \rho \gamma, \rho h \gamma^2 {\bm v}^\top, \rho h \gamma^2 - p \right)^\top, 
	\label{eq:86} \\ 
	&{\bf F}_i = 
	\left( D v_i, v_i {\bm m}^\top + p {\bm e}_i^\top, m_i \right)^\top 
	= 
	\left( \rho \gamma v_i, \rho h \gamma^2 v_i {\bm v}^\top + p {\bm e}_i^\top, \rho h \gamma^2 v_i \right)^\top.
	\label{eq:93} 
\end{align}
Here, $D,\ {\bm m}$, and $E$ correspond to the mass density, momentum density vector, and energy density, respectively, and are collectively referred to as {\em conservative variables}. The rest-mass density $\rho$, fluid velocity ${\bm v} = (v_1, v_2, \dots, v_d)^\top$, and pressure $p$ are often termed {\em primitive variables}. The Lorentz factor 
$\gamma = (1 - |{\bm v}|^2)^{-\frac12},$ with ${|\bm v|}$ denoting the magnitude of the velocity vector, and ${\bm e}_i$ stands for the $i$-th column of the $d \times d$ identity matrix.

The RHD system \eqref{eq:RHD3D} is closed by an EOS, which relates the specific enthalpy $h$ to pressure $p$ and rest-mass density $\rho$. To ensure relativistic causality (local sound speed $c_s<1$), the EOS must satisfy the inequality \cite{WuTang2017ApJS}:
\begin{equation*}
    h\ \left(\frac{1}{\rho} - \frac{\partial h}{\partial p}(p,\rho)\right) < \frac{\partial h}{\partial \rho}(p,\rho) < 0.
\end{equation*}
In this work, we adopt the ideal EOS:
\begin{equation}\label{ID-EOS}
    h = 1 + \frac{\Gamma p}{(\Gamma -1) \rho} ,
\end{equation}
where $\Gamma \in (1,2]$ is the adiabatic index. 
The admissible state set of the RHD system \eqref{eq:RHD3D} with the ideal EOS \eqref{ID-EOS} is defined as:
\begin{equation}\label{eq:3constraints}
    \mathcal{G} = \left\{ {\bf U} = (D, {\bm m}^\top, E)^\top \in \mathbb{R}^{d+2} \mid \rho({\bf U}) > 0,\ p({\bf U}) > 0,\ {\bm v}({\bf U}) \in \mathbb{B}_1({\bf 0})\subseteq\mathbb{R}^d\right\},
\end{equation}
where $\mathbb{B}_1({\bf 0})$  denotes the $d$-dimensional unit ball centered at the origin. An explicit equivalent characterization of $\mathcal{G}$ was derived in \cite{WuTang2015}:
\begin{equation} \label{1stequivG}
    \mathcal{G}^{(1)} = \left\{ {\bf U} = (D, {\bm m}^\top, E)^\top \in \mathbb{R}^{d+2} \mid D > 0,\ q({\bf U}) := E - \sqrt{D^2 + |{\bm m}|^2} > 0 \right\},
\end{equation}
which was further linearized via the GQL framework \cite{wu2023geometric} as:
\begin{equation}\label{eq:1218}
    \mathcal{G}^{(2)}= \left\{ {\bf U} = (D,{\bm m}^\top,E)^\top \in \mathbb{R}^{d+2} \mid D > 0,\  {\bf U}\cdot {\bm n}_* > 0,\ \forall {\bm v_*} \in \mathbb{B}_1({\bf 0}) \right\},
\end{equation}
where ${\bm n}_*:= \big( -\sqrt{1-|{\bm v}_*|^2}, -{\bm v}_*^\top, 1 \big)^\top$, and ${\bm v}_*$ is the auxiliary parameter introduced in the GQL framework to linearize the constraints. As demonstrated in Section \ref{Sec: BP analysis}, this linear representation $\mathcal{G}^{(2)}$ offers substantial advantages over the original nonlinear form of $\mathcal{G}$ in PCP analysis for RHD.

Different from the nonrelativistic case, there are no explicit expressions for either the fluxes ${\bf F}_i$
or the primitive quantities in terms of the conservative variables ${\bf U}$ for RHD. In the computation, we have to first recover the primitive variables from the conservative ones before evaluating the fluxes. For a given conservative vector $\bm U=(D,\bm m,E)^\top\in\mathcal G$, we recover primitive variables $(\rho,\bm v,p)^\top$ for the ideal EOS \eqref{ID-EOS} by solving a nonlinear algebraic equation \cite{WuTang2015}
\begin{equation} \label{RHDCon2Pri}
\Phi_{\bf U}(p) := \frac{p}{\Gamma-1}-E+\frac{\|{\bm m}\|^2}{E+p}+D\sqrt{1-\frac{\|{\bm m}\|^2}{(E+p)^2}} = 0 \quad \text{for}\ \ p>0.
\end{equation}
Once $p$ is determined from \eqref{RHDCon2Pri}, the velocity vector ${\bm v}$ and rest-mass density $\rho$ are computed via:
\[
\bm v=\frac{\bm m}{E+p},\qquad \rho=D\sqrt{1-\|{\bm v}\|^2}.
\]
To solve \eqref{RHDCon2Pri} robustly, we suggest the provably convergent iterative algorithms for RHD primitive-variable recovery in \cite{chen2022physical, cai2024provably}.


\section{Review of finite difference WENO methods} \label{Sec: FD Review}
\setcounter{equation}{0}
The proposed GQL-based limiting framework is applicable to general high-order finite-difference and finite-volume schemes. {Since WENO methods are particularly popular, our numerical experiments use a high-order finite-difference classical WENO-JS scheme \cite{jiang1996efficient} with $\epsilon = 10^{-6}$ as a representative example to demonstrate the effectiveness of the framework; however, the same framework can be applied directly to other high-order discretizations as well, such as  WENO-Z \cite{borges2008improved} and monotonicity-preserving \cite{suresh1997accurate} schemes.}
For completeness, we briefly recall the standard finite-difference WENO discretization for one-dimensional systems of conservation laws.
Consider the $d=1$ case of \eqref{eq:RHD3D},
\begin{equation}\label{eq:RHD1D}
    {\bf U}_t + {\bf F}({\bf U})_x = {\bf 0},
\end{equation}
on a uniform grid $x_i=i\Delta x$.
A conservative flux-differencing semi-discretization takes the form
\begin{equation}\label{eq:semiWENO}
    \frac{d{\bf U}_i}{dt} = -\frac{1}{\Delta x}\Bigl(\widehat{\bf F}_{i+1/2}-\widehat{\bf F}_{i-1/2}\Bigr),
\end{equation}
where $\widehat{\bf F}_{i+1/2}$ is a high-order accurate numerical flux at the interface $x_{i+1/2}=(x_i+x_{i+1})/2$.

\paragraph{Global Lax--Friedrichs flux splitting}
We employ the global Lax--Friedrichs (LF) splitting
\begin{equation}\label{eq:LFsplit}
    {\bf F}^{\pm}({\bf U})=\frac12\Bigl({\bf F}({\bf U}) \pm \alpha{\bf U} \Bigr),
\end{equation}
where $\alpha$ is the globally defined maximum wave speed, that is, the spectral radius of the Jacobian ${\bf A}({\bf U})=\partial{\bf F}/\partial{\bf U}$ over all grid points, i.e.,
$\alpha = \max_j \rho\bigl({\bf A}({\bf U}_j)\bigr)$.
Then ${\bf F}({\bf U})={\bf F}^{+}({\bf U})+{\bf F}^{-}({\bf U})$, and the positive/negative parts can be reconstructed using upwind-biased stencils.

\paragraph{Characteristic projection and WENO5 reconstruction}
Let ${\bf L}_{i+1/2}$ and ${\bf R}_{i+1/2}={\bf L}_{i+1/2}^{-1}$ be the left/right eigenvector matrices of ${\bf A}({\bf U})$ evaluated at a suitable interface state (e.g., a local average of $\{{\bf U}_j\}$).
Define the characteristic split fluxes
\[
    {\bf W}_{j}^{\pm} := {\bf L}_{i+1/2}\,{\bf F}^{\pm}({\bf U}_j).
\]
For each component of ${\bf W}_j^{+}$, apply the standard fifth-order WENO reconstruction \cite{liu1994weighted,jiang1996efficient} with a \emph{left-biased} stencil to obtain the interface value $\widehat{\bf W}^{+}_{i+1/2}$.
Similarly, reconstruct ${\bf W}_j^{-}$ with a \emph{right-biased} stencil to obtain $\widehat{\bf W}^{-}_{i+1/2}$.
Transform back to physical space,
\[
    \widehat{\bf F}^{\pm}_{i+1/2} := {\bf R}_{i+1/2}\,\widehat{\bf W}^{\pm}_{i+1/2},
\qquad
    \widehat{\bf F}_{i+1/2} := \widehat{\bf F}^{+}_{i+1/2}+\widehat{\bf F}^{-}_{i+1/2}.
\]
In multiple spatial dimensions, the above procedure is applied dimension-by-dimension to each directional flux.

\paragraph{Time integration}
The semi-discrete system \eqref{eq:semiWENO} is advanced in time by the third-order SSP Runge--Kutta method.
The PCP flux limiting procedure developed in Section~\ref{Sec: BP analysis} is applied at each Runge--Kutta stage.


\section{A novel PCP flux limiter} \label{Sec: BP analysis}
\setcounter{equation}{0}

In this section, we will discuss the PCP high-order schemes for one-, two-, and three-dimensional RHD with the ideal EOS \eqref{ID-EOS}. The PCP property is achieved through our novel flux limiter, which is motivated by Zalesak's limiter in Flux-Corrected Transport (FCT) schemes for scalar partial differential equations \cite{zalesak1979fully,zalesak2012design,kuzmin2012flux}.

Recently, Wu and Shu proposed the Geometric Quasi Linearization (GQL) framework for PCP analysis, providing new insights into dealing with nonlinear hyperbolic systems, such as RHD \eqref{eq:RHD3D}. The GQL theory reformulates the nonlinear constraints into equivalent linear ones by introducing additional free parameters. As observed and verified in \eqref{eq:1218}, we can judge whether the solution lies in $\mathcal{G}$ by the sign of the {\em scalar} ${\bf U} \cdot {\bm n}_*$, which can be integrated with Zalesak's limiter in scalar equations for RHD, without transforming the system into its local characteristic space or seeking a technical convex decomposition for a sufficient PCP condition.

\subsection{One-dimensional case}
We consider the following finite difference FCT scheme for the case $d = 1$ in \eqref{eq:RHD3D}. 
{In what follows, we denote:
\begin{itemize}
    \item $\widehat{\mathcal{F}}_{i+1/2}^L$: first-order low-order numerical flux (see \eqref{eq:LF_loworder} for the detailed definition).
    \item $\widehat{\mathcal{F}}_{i+1/2}^H$: fifth-order characteristic WENO flux described in Section~\ref{Sec: FD Review}.
    \item $\widehat{\mathcal{F}}_{i+1/2}^A := \widehat{\mathcal{F}}_{i+1/2}^H - \widehat{\mathcal{F}}_{i+1/2}^L$: anti-diffusive numerical flux.
\end{itemize}}
For convenience, we omit the subscript of $x_1$:
\begin{equation} \label{uLF1DRHD}
    {\bf U}_i^L := {\bf U}_i^n - \frac{\Delta t}{\Delta x}\left(\widehat{\mathcal{F}}_{i+{1}/{2}}^L-\widehat{\mathcal{F}}_{i-{1}/{2}}^L\right),
\end{equation}
\begin{equation} \label{uPCP1D}
    {\bf U}_i^{n+1} = {\bf U}_i^L - \frac{\Delta t}{\Delta x}\left(\theta_{i+{1}/{2}}\widehat{\mathcal{F}}_{i+{1}/{2}}^A-\theta_{i-{1}/{2}}\widehat{\mathcal{F}}_{i-{1}/{2}}^A\right).
\end{equation}
In our implementation, the low-order flux is chosen as the first-order global LF (Rusanov) flux
\begin{equation}\label{eq:LF_loworder}
    \widehat{\mathcal{F}}_{i+{1}/{2}}^L=\frac12\Bigl({\bf F}({\bf U}_i^n)+{\bf F}({\bf U}_{i+1}^n)-\alpha\,({\bf U}_{i+1}^n-{\bf U}_i^n)\Bigr)
\end{equation}
with the same LF parameter $\alpha$ as in \eqref{eq:LFsplit}. (At each SSP-RK stage, ${\bf U}^n$ is replaced by the corresponding stage solution.) The LF scheme \eqref{uLF1DRHD} with flux \eqref{eq:LF_loworder} is known to preserve the admissible set for the RHD system under a suitable CFL condition; see, e.g., \cite{WuTang2015,Wu2017a}. Our goal is to find the explicit formula for the coefficients $\theta_{i+{1}/{2}}$. 

At each SSP-RK stage, we enforce the following $\epsilon$-admissible set in conservative variables:
\begin{equation}\label{eq:1219}
    \mathcal{G}_\epsilon^{(1)}=\left\{{\bf U}\in\mathbb{R}^3 \ \big|\ D({\bf U})\geq\epsilon,\ q({\bf U})=E-\sqrt{D^2+m^2}\geq\epsilon\right\}.
\end{equation}
We fix a small parameter $\epsilon>0$ such that ${\bf U}_i^L\in \mathcal{G}_\epsilon^{(1)}$ for all $i$.
In computation, we choose it stage-by-stage as
$\epsilon=\min\bigl\{10^{-13},\ D_i^L\bigr\}$ or $\epsilon=\min\bigl\{10^{-13},\ q({\bf U}_i^L)\bigr\}$ for each cell within the computational domain.
{
The $q$-constraint in \eqref{eq:1219} is nonlinear with respect to ${\bf U}$, which makes it difficult for the PCP analysis. We first find an equivalent linear form following the GQL framework \cite{WuTang2015,wu2023geometric}.
\begin{proposition}[Equivalent GQL representation of $\mathcal{G}_\epsilon^{(1)}$]\label{prop:GQL form}
    The $\epsilon$-admissible set \eqref{eq:1219} can be equivalently written as an intersection of half-spaces:
\begin{equation}\label{1DRHDGQLeps}
    \mathcal{G}_\epsilon^{(2)}=\left\{{\bf U}\in\mathbb{R}^3 \ \big|\ D({\bf U})\geq\epsilon,\ {\bf U}\cdot{\bm n}_*(v_*)\geq \epsilon,\ 
    {\bm n}_*(v_*):=\left(-\sqrt{1-{ v}_*^2},\ -{v}_*,\ 1\right)^\top,\ \forall\ { v}_*\in (-1,1)\right\}.
\end{equation}
\end{proposition}
\begin{proof}
	Let $\phi({\bf U}, v_*):= {\bf U}\cdot{\bm n}_*(v_*)-\epsilon = E-mv_*-D\sqrt{1-v_*^2}-\epsilon$. By the Cauchy-Schwarz inequality, we obtain
 \[
 \phi({\bf U}, v_*) \geq E - \sqrt{D^2+m^2}\sqrt{v_*^2+\left(\sqrt{1-v_*^2}\right)^2} = q({\bf U})-\epsilon.
 \]
 Since the equality holds when $v_* = \frac{m}{\sqrt{D^2+m^2}}\in(-1,1)$, we have $\min\limits_{v_*\in(-1,1)}\phi({\bf U},v_*)=q({\bf U})-\epsilon$. This completes the proof.
\end{proof}
}
For the exact GQL characterization with zero lower bound, normalization of the normal vector is immaterial because the half-space inequality is homogeneous under positive rescaling. 
{
Proposition~\ref{prop:GQL form} states that, for the same $\epsilon$, $q({\bf U}) \ge \epsilon$ is exactly equivalent to ${\bf U}\cdot{\bm n}_*(v_*)\geq \epsilon$ for all ${ v}_*\in (-1,1)$ with the specific unnormalized normal ${\bm n}_*(v_*)$ in \eqref{1DRHDGQLeps}, we thus do not rescale it; consequently the same lower bound $\epsilon$ is used consistently in all projected inequalities. 
If one chooses to normalize ${\bm n}_*(v_*)$ to unit Euclidean length $(\widehat{\bm n}_* = {\bm n}_* / \sqrt{2})$, the equivalent constraint becomes ${\bf U}\cdot\widehat{\bm n}_*\geq\epsilon/\sqrt{2}$. The equivalence is preserved only if the constraint parameter is scaled by the reciprocal of the norm. Our choice of the unnormalized normal vector is purely for notational convenience, as it avoids the introduction of arbitrary scaling factors in the subsequent flux limiter derivation. This normalization property holds generally for the GQL framework in all spatial dimensions.
} Therefore, it suffices to apply the scalar lower-bound operation in Zalesak's FCT limiter to the two linear constraints associated with ${\bm n}_D:=(1,0,0)^\top$ and ${\bm n}_*(v_*)$ in \eqref{1DRHDGQLeps}. To make the description convenient and clear, we follow Zalesak's notation \cite{zalesak2012design} and redefine the related quantities for any fixed vector ${\bm n}$.
\begin{align*}
    P^{-}_{i} &= \min\left(0, -\widehat{\mathcal{F}}^{A}_{i+{1}/{2}}\cdot{\bm n}\right) + \min\left(0,       \widehat{\mathcal{F}}^{A}_{i-{1}/{2}}\cdot{\bm n}\right).\\
    Q^{-}_{i} &= \frac{\Delta x}{\Delta t}\left(\epsilon - {\bf U}^{L}_{i}\cdot{\bm n}\right),\quad {\bf U}^{L}_{i}\ \text{is defined in }\eqref{uLF1DRHD}.\\
    R^{-}_{i} &= 
    \begin{cases}
        \min\left(1, \frac{Q^{-}_{i}}{P^{-}_{i}}\right), & \text{if}\quad P^{-}_{i} < 0,\\
        1,& \text{otherwise}.
    \end{cases}
\end{align*}

{\bf Step 1:} Enforce $D({\bf U})\geq\epsilon$.
This step is identical to the scalar FCT lower-bound limiting. We take ${\bm n} = {\bm n}_D$ (so that ${\bf U}\cdot{\bm n}_D=D({\bf U})$) and use the parameter $\epsilon=\min\bigl\{10^{-13},\ D_i^L\bigr\}$. The coefficients of the density flux limiter are computed by
\begin{equation} \label{alphaD1D}
    \theta_{i+\frac{1}{2}}^D = 
    \begin{cases}
        R_{i}^- & \text{if}\quad \widehat{\mathcal{F}}_{i+{1}/{2}}^{A}\cdot{\bm n}_D \geq 0, \\
        R_{i+1}^- & \text{if}\quad \widehat{\mathcal{F}}_{i+{1}/{2}}^{A}\cdot{\bm n}_D < 0.
    \end{cases}
\end{equation}
The density flux limiter $\theta_{i+{1}/{2}}^D$ can be exactly computed in each interface because $\widehat{\mathcal{F}}_{i+{1}/{2}}^{A,D}$, representing the first component of the anti-diffusive numerical flux in the $x$ direction, is a scalar.

{\bf Step 2:} Enforce ${\bf U}\cdot{\bm n_*}\geq \epsilon\iff q({\bf U})=E-\sqrt{D^2+m^2}\geq\epsilon$.
Similar to {Step 1}, we obtain the formulations of the flux limiter for $q({\bf U})$
\begin{equation} \label{1DRHDPressure}
    \theta_{i+\frac{1}{2}}^q = 
    \begin{cases}
        R_{i}^- & \text{if}\quad \widehat{\mathcal{F}}_{i+{1}/{2}}^{A}\cdot{\bm n}_* \geq 0, \\
        R_{i+1}^- & \text{if}\quad \widehat{\mathcal{F}}_{i+{1}/{2}}^{A}\cdot{\bm n}_* < 0.
    \end{cases}
\end{equation}
In this step, ${\bm n} = {\bm n}_*(v_*)=\left(-\sqrt{1-{ v}_*^2},\ -{v}_*,\ 1\right)^\top$ with ${v}_*\in (-1,1)$, and $\epsilon=\min\bigl\{10^{-13},\ q({\bf U}_i^L)\bigr\}$ is used. The main difference from Step~1 is the presence of the free parameter ${v}_*$. Since $\widehat{\mathcal{F}}_{i+{1}/{2}}^{A}\cdot{\bm n}_*$ in Step~2 depends on ${v}_*$, its sign cannot be determined once and for all at a fixed interface; hence the sign-based coefficient in \eqref{1DRHDPressure} cannot be used directly as a uniform limiter for all GQL half-spaces. We therefore solve the following worst-case optimization problem
\begin{align}
    (P_1): \mathcal{L}_{i} &= \min\limits_{{v}_*\in (-1,1)}\frac{\frac{\Delta x}{\Delta t}\left({\epsilon} - {\bf U}^{L}_{i}\cdot{\bm n}_*\right)}{\min\left(0, -\widehat{\mathcal{F}}^{A}_{i+{1}/{2}}\cdot{\bm n}_*\right) + \min\left(0, \widehat{\mathcal{F}}^{A}_{i-{1}/{2}}\cdot{\bm n}_*\right)} \nonumber \\
    &= \min\limits_{{v}_*\in (-1,1)}\frac{\frac{\Delta x}{\Delta t}\left({\bf U}^{L}_{i}\cdot{\bm n}_*-{\epsilon}\right)}{\max\left(0, \widehat{\mathcal{F}}^{A}_{i+{1}/{2}}\cdot{\bm n}_*\right) + \max\left(0, -\widehat{\mathcal{F}}^{A}_{i-{1}/{2}}\cdot{\bm n}_*\right)} \nonumber \\
    &= \frac{\Delta x}{\Delta t}\left(\max\limits_{{ v}_*\in (-1,1)}
    \frac{\max\left(0, \widehat{\mathcal{F}}^{A}_{i+{1}/{2}}\cdot{\bm n}_*\right) + 
                        \max\left(0, -\widehat{\mathcal{F}}^{A}_{i-{1}/{2}}\cdot{\bm n}_*\right)}
    {{\bf U}^{L}_{i}\cdot{\bm n}_*-{\epsilon}}\right)^{-1}. \tag{P1}\label{1DbeforeQuadForm}
\end{align}
As in the scalar FCT coefficient \(R_i^-\), the actual cell-wise coefficient is capped at one. Equivalently, if \(M_i\) denotes the denominator maximum in the last line of \eqref{1DbeforeQuadForm}, then \(\mathcal{L}_i=1\) when \(M_i=0\), and otherwise
\[
\mathcal{L}_i=\min\left\{1,\frac{\Delta x}{\Delta t\,M_i}\right\}.
\]
The uncapped denominator ratio is introduced below only for deriving the quadratic maximization.
\begin{equation} \label{1Doptfun}
\mathcal{L}_i({\bm n}_*) := \frac{\max\left(0, \widehat{\mathcal{F}}^{A}_{i+{1}/{2}}\cdot{\bm n}_*\right) + 
                        \max\left(0, -\widehat{\mathcal{F}}^{A}_{i-{1}/{2}}\cdot{\bm n}_*\right)}
    {{\bf U}^{L}_{i}\cdot{\bm n}_*-{\epsilon}}.
\end{equation}

\begin{table}[!tb] 
    \renewcommand{\arraystretch}{1.8}
    \centering
    \belowrulesep=0pt
    \aboverulesep=0pt
    \caption{Formulae for coefficients in \eqref{1DUFcdotn}.}
    \label{tab:1DQuadCoeff}
    \setlength{\tabcolsep}{3mm}{
        \begin{tabular}{c|c|c}
            \toprule[1.5pt]
            {} &
            {$(a,b,c)^\top$} &
            {$(d^{\pm},e^{\pm},f^{\pm})^\top$} \\  
            \midrule[1.5pt]
            {$|{u}_*|<1$} & ${\bf M}{\bf U}_i^L-{\bf S}$ & ${\bf M}\widehat{\mathcal{F}}^{A}_{i\pm1/2}$ \\
            \bottomrule[1.5pt]
        \end{tabular}
        }
\end{table}

{Since the normal vector ${\bm n}_*(v_*)$ contains the irrational term $\sqrt{1-|v_{*}|^{2}}$, the minimization problem \eqref{1DbeforeQuadForm} is algebraically cumbersome to handle directly.} To rigorously transform this into a quadratic form, we introduce a stereographic-type rational parameterization of the unit interval. Let $u_*$ be an auxiliary variable; we define the mapping:
\begin{equation}\label{1Dvstar_para} 
	v_{*} := \frac{2u_{*}}{1+u_{*}^{2}} {~~\mbox{ with } u_* \in (-1,1) \quad \implies \quad {\bm n}_*(u_*) = \left(-\frac{1-u_{*}^{2}}{1+u_{*}^{2}}, -\frac{2u_{*}}{1+u_{*}^{2}}, 1\right)^\top.}
\end{equation}
{For the open admissible range, the restriction $|u_*|<1$ is in one-to-one correspondence for $v_*\in(-1,1)$, with inverse $u_*={v_*}/(1+\sqrt{1-v_*^2})$.}
{ To simplify the notation and make our presentation clearer, we introduce the following matrix and vector:
\begin{equation*}
    \mathbf{M} = \begin{pmatrix} 1 & 0 & 1 \\ 0 & 2 & 0 \\ -1 & 0 & 1 \end{pmatrix}, \quad 
    \mathbf{S} = \begin{pmatrix} \epsilon \\ 0 \\ \epsilon \end{pmatrix}.
\end{equation*}
Then we obtain
{
\begin{equation}
    {\bf U}_{i}^L\cdot{\bm n}_*-\epsilon:=\frac{au_{*}^2-bu_{*}+c}{1+u_{*}^2}, \qquad
    \widehat{\mathcal{F}}_{i\pm{1}/{2}}^A\cdot{\bm n}_*:=\frac{d^{\pm}u_{*}^2-e^{\pm}u_{*}+f^{\pm}}{1+u_{*}^2}, \label{1DUFcdotn}
\end{equation}}
where the definitions of the quantities $a$, $b$, $c$, $d^{\pm}$, $e^{\pm}$, and $f^{\pm}$ are listed in Table \ref{tab:1DQuadCoeff}.}

{Using \eqref{1DUFcdotn},} we can solve $(P_1)$ by explicitly expressing the optimization function ${\mathcal{L}}_{i}({\bm n}_*)$ as the quadratic form with an extra constraint $|{ u}_*|<1$. 
{We first introduce some matrix notations, transforming \eqref{1Doptfun} into the matrix-vector formulation. Let $\star \in \{+,-\}$ and define the sign function 
\begin{equation*}
    \tt{sign}(\star) :=
    \begin{cases}
        1 & \text{if}\ \star = +,\\
        -1 & \text{if}\ \star = -.
    \end{cases}
\end{equation*}
Then we define
\begin{equation*}
    A^{\star}:= {\tt sign}(\star)
    \begin{pmatrix}
        d^{\star} & -\frac{e^{\star}}{2} \\
        -\frac{e^{\star}}{2} & f^{\star}
    \end{pmatrix},\qquad
    B:=
    \begin{pmatrix}
        a & -\frac{b}{2} \\
        -\frac{b}{2} & c
    \end{pmatrix},\qquad
    {\bm w}:=
    \begin{pmatrix}
        u_* \\ 1
    \end{pmatrix}.
\end{equation*}
{
Since ${\bm w}^\top B{\bm w}={\bf U}^{L}_{i}\cdot{\bm n}_*-\epsilon\geq 0$, the matrix $B$ is symmetrically positive definite except for the case ${b}^2=4ac\iff q({\bf U}_{i}^L)=\epsilon$. 
In this borderline case,} any anti-diffusive correction may immediately violate $q({\bf U})\ge\epsilon$, and we simply take $\mathcal{L}_{i}=0$ (i.e., keep the low-order update). Otherwise, we rewrite 
\begin{equation*}
    B = R^\top R,\qquad
    R = 
    \begin{pmatrix}
        \sqrt{a} & -\frac{b}{2\sqrt{a}} \\
        0 & \sqrt{c-\frac{{b}^2}{4a}}
    \end{pmatrix}.
\end{equation*}
We can then derive the explicit formulae for matrices $\widehat{A}^{\star}:=R^{-\top}A^{\star}R^{-1}$:
\begin{equation} \label{Astar}
    \widehat{A}^{\star} = {\tt sign}(\star)
    \begin{pmatrix}
        \frac{d^{\star}}{a} & \frac{bd^{\star}-ae^{\star}}{a\sqrt{4ac-{b}^2}} \\
        \frac{bd^{\star}-ae^{\star}}{a\sqrt{4ac-{b}^2}} & 
        \frac{4a^2f^{\star}-2a{b}{e^{\star}}+d^{\star}{b}^2}{a\left(4ac-{b}^2\right)}
    \end{pmatrix}.
\end{equation}
Based on the above matrix-vector notations, we continue to solve $(P_1)$. }

\begin{align*}
    {\mathcal{L}}_{i}&=\frac{\Delta x}{\Delta t}\left(\max\limits_{{\bm w}:=({u}_*,1)^\top\in\mathbb{R}^2 \atop {|{ u}_{*}|<1}}
    \frac{\max\left(0,{\bm w}^\top A^+{\bm w}\right)+\max\left(0,{\bm w}^\top A^-{\bm w}\right)}{{\bm w}^\top B{\bm w}}\right)^{-1}\\
    &=\frac{\Delta x}{\Delta t}\left(\max\limits_{{\bm z}:=R{\bm w}\ \in\ \mathbb{R}^2 \atop {|{ u}_{*}|<1}}
    \frac{\max\left(0,{\bm z}^\top \widehat{A}^+{\bm z}\right)+\max\left(0,{\bm z}^\top \widehat{A}^-{\bm z}\right)}
         {{\bm z}^\top {\bm z}}\right)^{-1}\\
    &=\frac{\Delta x}{\Delta t}\left(\max\limits_{{\bm z}\in\mathbb{R}^2,\ \Vert {\bm z} \Vert_2=1 \atop {|{ u}_{*}|<1}}
    \max\left(0,{\bm z}^\top \widehat{A}^+{\bm z}\right)+\max\left(0,{\bm z}^\top \widehat{A}^-{\bm z}\right)\right)^{-1}.
\end{align*}
The solution to the Rayleigh quotient $\max\limits_{{\bm z}\in\mathbb{R}^2,\ \Vert {\bm z} \Vert_2=1}{\bm z}^\top\widehat{A}^{\star}{\bm z}$ is the maximal eigenvalue of the matrix $\widehat{A}^{\star}$. However, this is not a standard Rayleigh quotient, and there are two main differences. On the one hand, we need an extra check for the feasibility of the associated optimizer ${\bm w}=({u}_*,1)^\top$, i.e., $|{u}_{*}|<1$ or not. On the other hand, for the case $d = 1$ in \eqref{eq:RHD3D}, the optimization function \eqref{1Doptfun} is a sum of two quotients composed of the maximum operation, which is clearly different. We adopt the following lemma to simplify the maximum operation. 
{\begin{lemma}\label{lem:RayleighQ}
Let $p_1,\dots,p_K\in\mathbb{R}$, and ${\bm s} = (s_1, s_2, \dots, s_K)^\top \in \{0,1\}^K$, where $s_k$ is the k-th element of the binary vector ${\bm s}$. Then
\[
\sum_{k=1}^K \max(0,p_k)
= \max_{S\subseteq\{1,\dots,K\}} \sum_{k\in S} p_k
= \max_{\bm s\in\{0,1\}^K} \sum_{k=1}^K s_k p_k.
\]
\end{lemma}}
\begin{proof}
Let $S_+:=\{k\in\{1,\dots,K\}\,:\,p_k>0\}$. Then $\sum_{k=1}^K\max(0,p_k)=\sum_{k\in S_+}p_k$.
For any subset $S$, we have $\sum_{k\in S}p_k\le \sum_{k\in S_+}p_k$, because adding any index with $p_k\le 0$ cannot increase the sum.
Therefore the maximum is attained by $S=S_+$, which proves the identities.
\end{proof}
We can apply Lemma \ref{lem:RayleighQ} to the Rayleigh quotient to solve ${\mathcal{L}}_{i}$ immediately by taking $K=2$.
{\begin{equation*}
    {\mathcal{L}}_{i}
    =\frac{\Delta x}{\Delta t}\left(
    \max\left\{0,\
    \max\limits_{{\bm z}\in\mathbb{R}^2,\ \Vert {\bm z} \Vert_2=1 \atop {|{ u}_{*}|<1}}{\bm z}^\top \widehat{A}^+{\bm z},\ \max\limits_{{\bm z}\in\mathbb{R}^2,\ \Vert {\bm z} \Vert_2=1 \atop {|{ u}_{*}|<1}}{\bm z}^\top \widehat{A}^-{\bm z},\ 
    \max\limits_{{\bm z}\in\mathbb{R}^2,\ \Vert {\bm z} \Vert_2=1 \atop {|{ u}_{*}|<1}}{\bm z}^\top \left(\widehat{A}^++\widehat{A}^-\right){\bm z}\right\}
    \right)^{-1}.
\end{equation*}}
{
The remaining step is to solve the Rayleigh quotient with the additional constraint $|{u}_{*}|<1$:
\begin{equation*}
    \widetilde{\mathcal{L}}_{i}:=\max\limits_{{\bm z}\in\mathbb{R}^2,\ \Vert {\bm z} \Vert_2=1 \atop {|{ u}_{*}|<1}}{\bm z}^\top \widehat{A}\ {\bm z},
\end{equation*}}
which is the same as the standard Rayleigh quotient except for a further validation of the associated optimizer ${\bm w}=({u}_*,1)^\top$. If the optimizer ${\bm w}$ associated with the maximal eigenvalue satisfies $|{u}_{*}|<1$, then the solution is consistent with the standard Rayleigh quotient. Otherwise, we compare the remaining feasible eigenvalues with the boundary values. Let $\widehat{\lambda}_1$ and $\widehat{\lambda}_2$ be the eigenvalues of the model matrix $\widehat{A}$, and $\widehat{\lambda}_1 > \widehat{\lambda}_2$. Then we have
\begin{equation*}
    \widetilde{\mathcal{L}}_{i}=
    \begin{cases}
        \max\left\{0,\ \widehat{\lambda}_1\right\} & \text{if}\ \left|{u}_{*,1}\right|<1\\
        \max\left\{0,\ \widehat{\lambda}_2,\ \frac{d}{a},\ \frac{f}{c}\right\} & \text{if}\ \left|{u}_{*,1}\right|\geq 1\ \text{and}\ \left|{u}_{*,2}\right|<1\\
        \max\left\{0,\ \frac{d}{a},\ \frac{f}{c}\right\} & \text{otherwise}
    \end{cases},
\end{equation*}
where {$u_{*,1}$ and $u_{*,2}$ are the first component of the optimizers corresponding to} $\widehat{\lambda}_1$ and $\widehat{\lambda}_2$, respectively. The definitions of $d$, $e$, and $f$ are consistent with $\widehat{A}$.
{The eigenvalues of the matrices $\widehat{A}^{\star}$ ($\star\in\{+,-\}$) and $\widehat{A}^+ + \widehat{A}^-$ are listed below, denoted by $\widehat{\lambda}_{k}^{\star}$ and $\widehat{\lambda}_{k}^{*}$ ($k = 1,2$):
\begin{align}
    \widehat{\lambda}_k^{\star} &= \frac{(-1)^{k-1}\sigma^{\star}+{\tt sign}(\star)(2af^{\star}-{b}{e^{\star}}+2cd^{\star})}{4ac-{b}^2},\label{1Deig1}\\
    \widehat{\lambda}_k^{*} &= \frac{(-1)^{k-1}\sigma^{*}+2af^{*}-{b}{e^{*}}+2cd^{*}}{4ac-{b}^2},\label{1Deig2}
\end{align}
where
\begin{align}
    d^{*}&:=d^+-d^-,\qquad e^{*}:=e^+ - e^-,\qquad f^{*}:=f^+-f^-,\nonumber\\
    \sigma^{\widehat{\star}}&:=2\sqrt{(af^{\widehat{\star}}-cd^{\widehat{\star}})^2-(af^{\widehat{\star}}+cd^{\widehat{\star}}){b}{e^{\widehat{\star}}}+ac{\left(e^{\widehat{\star}}\right)}^2+{b}^2d^{\widehat{\star}} f^{\widehat{\star}}},\quad {\widehat{\star}\in\{\star, *\}}.\label{sigma1D}
\end{align}
The non-negativity of the radicand in \eqref{sigma1D} is a fundamental prerequisite for the well-posedness of $\sigma^{\widehat{\star}}$. {If $b=0$, the radicand reduces to $(af^{\widehat{\star}}-cd^{\widehat{\star}})^2+ac\big(e^{\widehat{\star}}\big)^2\ge0$. Thus, in the quadratic-form representation below, we only need to consider $b\ne0$.} We rewrite it as a symmetric quadratic form to reveal its underlying mathematical structure:
\[
\sigma^{\widehat{\star}} = \sqrt{\mathbf{r}^\top \mathbf{Q} \mathbf{r}}, \quad \text{where} \quad \mathbf{r} = \left( \frac{d^{\widehat{\star}}}{a}, \frac{e^{\widehat{\star}}}{b}, \frac{f^{\widehat{\star}}}{c} \right)^\top,
\]
and the symmetric matrix $\mathbf{Q}$ is defined as
\[
\mathbf{Q} = 2ac
\begin{pmatrix}
    2ac & -b^2 & b^2-2ac \\
    -b^2 & 2b^2 & -b^2 \\
    b^2-2ac & -b^2 & 2ac
\end{pmatrix}.
\]
We prove that $\mathbf{Q}$ is symmetric positive semi-definite for all admissible states by verifying that all the eigenvalues of ${\bf Q}$ are non-negative. The three eigenvalues of ${\bf Q}$
\[
\lambda_{\bf Q}^{(1)} = 0,\qquad \lambda_{\bf Q}^{(2)} = 6ab^2c,\qquad \lambda_{\bf Q}^{(3)} = 2ac(4ac-b^2)
\]
are all non-negative, following the definitions of $a,b,c$ and the admissible state ${\bf U}_i^L\in\mathcal{G}_{\epsilon}^{(1)}$.
Therefore, the quadratic form $\langle \mathbf{r}, \mathbf{Q} \mathbf{r} \rangle$ is non-negative for all $\mathbf{r}$, and $\sigma^{\widehat{\star}}$ is well-defined.
}

{
The key components in computing the $q$-limiting factor are the eigenvalues of matrices $\widehat{A}^{\star}$, defined in \eqref{Astar}. It is necessary to explicitly clarify the situation where $\widehat{A}^{\star}$ is positive definite, positive semi-definite, or indefinite. We summarize it in the following proposition to improve the readability and rigor.

\begin{proposition}[One-dimensional classification of the matrix $\widehat{A}^{\star}$]
Assume that the low-order state ${\bf U}_i^L$ produced by \eqref{uLF1DRHD} satisfies ${\bf U}_i^L\in \mathcal{G}_\epsilon^{(1)}$ for all $i$, where $\epsilon>0$. Then $\widehat{A}^{\star}$ in \eqref{Astar} is
\begin{itemize}
    \item positive definite $\iff$ {\tt sign}$(\star)(2af^{\star}+2cd^{\star}-be^{\star})>0\ $ and $\ 4d^{\star}f^{\star}-(e^{\star})^2>0$,
    \item positive semi-definite $\iff$ {\tt sign}$(\star)(2af^{\star}+2cd^{\star}-be^{\star})\geq0\ $ and $\ 4d^{\star}f^{\star}-(e^{\star})^2\geq0$,
    \item indefinite $\iff$ $4d^{\star}f^{\star}-(e^{\star})^2<0$.
\end{itemize}
\end{proposition}
\begin{proof}
    Since ${\bf U}_i^L\in \mathcal{G}_\epsilon^{(1)}$ for all $i$, we have $a>0$, $c>0$, and $4ac-b^2\geq0$. A matrix is positive definite if and only if its trace and determinant are both positive, which are consistent with {\tt sign}$(\star)(2af^{\star}+2cd^{\star}-be^{\star})>0\ $ and $\ 4d^{\star}f^{\star}-(e^{\star})^2>0$ for $\widehat{A}^{\star}$. Similar arguments can be applied to the other two cases.
\end{proof}
}

\begin{remark}[Two branches in the rational parameterization]\label{rem:branches}
The rational map \eqref{1Dvstar_para} is two-to-one: $u_*$ and $1/u_*$ generate the same $v_*\in(-1,1)$ and hence the same normal vector ${\bm n}_*(v_*)$.
{Consequently, the optimization may be restricted to either branch, $|u_*| < 1$ or $|u_*| > 1$, as they yield the same final result.}
\end{remark}

\begin{remark}[Conservative enforcement of GQL constraints]\label{rem:qest}
	For any fixed parameter $v_*\in(-1,1)$, the inequality ${\bf U}\cdot{\bm n}_*(v_*) \ge \epsilon$ defines a single linear half-space. In a scalar setting, this would admit a parameter-dependent Zalesak coefficient, denoted here as $R_i^-(v_*)$, which depends on the sign of the projected anti-diffusive flux $\widehat{\mathcal{F}}_{i+1/2}^A\cdot{\bm n}_*(v_*)$.
	However, the admissible set $G$ is the intersection of all such half-spaces. To guarantee that the updated solution ${\bf U}_i^{n+1}$ resides in $G$, we require a single limiting coefficient that enforces the constraint \emph{uniformly} for all $v_*$.
	Since the flux sign may vary with $v_*$, we compute a cell-wise constant $\mathcal{L}_i$ via the worst-case minimization problem \eqref{1DbeforeQuadForm}, ensuring that $\mathcal{L}_i \le R_i^-(v_*)$ for all $v_* \in (-1,1)$. We then define the (uniform) interface limiting factor as
	\begin{equation}\label{alphaq1D}
		\widehat{\theta}_{i+1/2}^q := \min\{\mathcal{L}_i,\ \mathcal{L}_{i+1}\}.
	\end{equation}
	This construction guarantees that ${\bf U}_i^{n+1}\cdot{\bm n}_*(v_*)\ge\epsilon$ holds for all $v_*$, which is sufficient to prove the PCP property in Theorem~\ref{thm:1DBP}.
\end{remark}

Before ending this subsection, we summarize the one-dimensional PCP property.

\begin{theorem}[One-dimensional PCP flux limiter]\label{thm:1DBP}
Assume that, at a given SSP-RK stage, the low-order state ${\bf U}_i^L$ produced by \eqref{uLF1DRHD} satisfies ${\bf U}_i^L\in \mathcal{G}_\epsilon^{(1)}$ for all $i$, where $\epsilon>0$.
Let the density limiting factors $\theta_{i+1/2}^D$ be defined by \eqref{alphaD1D}, and let the $q$-limiting factors $\widehat{\theta}_{i+1/2}^q$ be defined by \eqref{alphaq1D}.
Define the final interface limiter by
\begin{equation}\label{eq:theta_final_1d}
    \theta_{i+1/2}:=\min\bigl\{\theta_{i+1/2}^D,\ \widehat{\theta}_{i+1/2}^q\bigr\}.
\end{equation}
Then the flux-corrected update \eqref{uPCP1D} satisfies ${\bf U}_i^{n+1}\in \mathcal{G}_\epsilon^{(1)}$ for all $i$.

Moreover, if ${\bf U}_i^n\in\mathcal{G}_\epsilon^{(1)}$ and the time step satisfies the standard LF CFL condition $\alpha\,\Delta t/\Delta x\le 1/2$, then the low-order update \eqref{uLF1DRHD}--\eqref{eq:LF_loworder} yields ${\bf U}_i^L\in\mathcal{G}_\epsilon^{(1)}$, and hence the overall limited high-order scheme is PCP.
\end{theorem}

\begin{proof}
The set $\mathcal{G}_\epsilon^{(2)}$ in \eqref{1DRHDGQLeps} is an intersection of half-spaces defined by the linear functionals
\[
    \ell_D({\bf U}) := D({\bf U}) = {\bf U}\cdot{\bm n}_D,\qquad {\bm n}_D=(1,0,0)^\top,
\]
and
\[
    \ell_{v_*}({\bf U}) := {\bf U}\cdot{\bm n}_*(v_*),\qquad
    {\bm n}_*(v_*):=\bigl(-\sqrt{1-v_*^2},\,-v_*,\,1\bigr)^\top,\quad v_*\in(-1,1).
\]
Since $\ell_D$ and $\ell_{v_*}$ are linear in ${\bf U}$, taking the dot product of \eqref{uPCP1D} with a fixed ${\bm n}$ yields a scalar flux-corrected transport (FCT) update for the scalar quantity $\ell_{\bm n}({\bf U}_i)$, with anti-diffusive flux $\widehat{\mathcal{F}}_{i+1/2}^A\cdot{\bm n}$.

\medskip
\noindent\textbf{(i) Density constraint.}
With ${\bm n}={\bm n}_D$, the scalar FCT lower-bound limiting (Zalesak limiter) shows that choosing $\theta_{i+1/2}\le \theta_{i+1/2}^D$ guarantees
$D({\bf U}_i^{n+1})={\bf U}_i^{n+1}\cdot{\bm n}_D\ge \epsilon$ for all $i$.

\medskip
\noindent\textbf{(ii) GQL constraints.}
Fix any $v_*\in(-1,1)$ and take ${\bm n}={\bm n}_*(v_*)$. The same scalar FCT argument shows that if the interface factors satisfy
$\theta_{i+1/2}\le \theta_{i+1/2}(v_*)$ for all interfaces, where $\theta_{i+1/2}(v_*)$ is the Zalesak coefficient associated with $\ell_{v_*}$, then
${\bf U}_i^{n+1}\cdot{\bm n}_*(v_*)\ge \epsilon$ for all $i$.
Our construction computes cell-wise constants $\mathcal{L}_i$ such that $\mathcal{L}_i\le R_i^-(v_*)$ for all $v_*\in(-1,1)$, and then sets
$\widehat{\theta}_{i+1/2}^q=\min\{\mathcal{L}_i,\mathcal{L}_{i+1}\}$; see \eqref{alphaq1D}--\eqref{1DbeforeQuadForm}.
Therefore $\theta_{i+1/2}\le \widehat{\theta}_{i+1/2}^q\le \theta_{i+1/2}(v_*)$ for all $v_*$, implying ${\bf U}_i^{n+1}\cdot{\bm n}_*(v_*)\ge \epsilon$ for all $v_*$.

\medskip
Combining (i)--(ii) and \eqref{1DRHDGQLeps} yields ${\bf U}_i^{n+1}\in\mathcal{G}_\epsilon^{(2)}$ for all $i$.
By the equivalence \eqref{eq:1219}--\eqref{1DRHDGQLeps}, this implies ${\bf U}_i^{n+1}\in\mathcal{G}_\epsilon^{(1)}$.
\end{proof}

\subsection{Multi-dimensional case}
The discussions about the case $d = 1$ in \eqref{eq:RHD3D} can be easily generalized to $d = 2$ and $d = 3$. This is because the GQL tools conveniently transform the dimension of computational variables ${\bf U}$ and ${\widehat{\mathcal{F}}}$ into that of the introduced auxiliary ones. For uniform rectangular cells, we only need to apply the one-dimensional flux limiter in each spatial dimension separately in the finite difference framework.
To make the notation clear, we use $y$ and $z$ to represent the other two spatial dimensions, rather than $x_2$ and $x_3$ in \eqref{eq:RHD3D}. We analyze the two-dimensional case and provide a brief summary of the three-dimensional case.

The finite difference FCT scheme for the two-dimensional RHD
\[
    {\bf U}_t+{\bf F}({\bf U})_x+{\bf G}({\bf U})_y = 0,
\]
\[
    {\bf U} = 
    \begin{pmatrix}
        D \\ m_x \\ m_y \\ E
    \end{pmatrix},
    \qquad
    {\bf F}({\bf U}) = 
    \begin{pmatrix}
        D v_x \\ m_xv_x+p \\ m_yv_x \\ m_x
    \end{pmatrix},
    \qquad
    {\bf G}({\bf U}) = 
    \begin{pmatrix}
        D v_y \\ m_xv_y \\ m_yv_y+p \\ m_y
    \end{pmatrix}
\]
can be formulated as
\begin{equation}\label{uLF2DRHD}
    {\bf U}_{ij}^{L} := {\bf U}_{ij}^n
-\frac{\Delta t}{\Delta x}\left(\widehat{\mathcal{F}}_{i+\frac{1}{2},j}^L-\widehat{\mathcal{F}}_{i-\frac{1}{2},j}^L\right)
-\frac{\Delta t}{\Delta y}\left(\widehat{\mathcal{G}}_{i,j+\frac{1}{2}}^L-\widehat{\mathcal{G}}_{i,j-\frac{1}{2}}^L\right),
\end{equation}
\begin{equation}\label{2DFE}
    {\bf U}_{ij}^{n+1} = {\bf U}_{ij}^L
-\frac{\Delta t}{\Delta x}\left(\theta_{i+\frac{1}{2},j}\widehat{\mathcal{F}}_{i+\frac{1}{2},j}^A-\theta_{i-\frac{1}{2},j}\widehat{\mathcal{F}}_{i-\frac{1}{2},j}^A\right)
-\frac{\Delta t}{\Delta y}\left(\theta_{i,j+\frac{1}{2}}\widehat{\mathcal{G}}_{i,j+\frac{1}{2}}^A-\theta_{i,j-\frac{1}{2}}\widehat{\mathcal{G}}_{i,j-\frac{1}{2}}^A\right).
\end{equation}

The density limiter is similar to the one-dimensional case, but it requires computation in both the $x$ and $y$ directions. Following Zalesak's notation \cite{zalesak2012design}, for a fixed vector ${\bm n}$, we define
\begin{align*}
    P^{-}_{ij} &= \Delta y\left[\min\left(0, -\widehat{\mathcal{F}}^{A}_{i+\frac{1}{2},j}\cdot{\bm n}\right) + \min\left(0, \widehat{\mathcal{F}}^{A}_{i-\frac{1}{2},j}\cdot{\bm n}\right)\right]
                +\Delta x\left[\min\left(0, -\widehat{\mathcal{G}}^{A}_{i,j+\frac{1}{2}}\cdot{\bm n}\right) + \min\left(0, \widehat{\mathcal{G}}^{A}_{i,j-\frac{1}{2}}\cdot{\bm n}\right)\right].\\
    Q^{-}_{ij} &= \frac{\Delta x\Delta y}{\Delta t}\left(\epsilon - {\bf U}^{L}_{ij}\cdot{\bm n}\right),\quad {\bf U}^{L}_{ij}\ \text{is defined in }\eqref{uLF2DRHD}.\\
    R^{-}_{ij} &= 
    \begin{cases}
        \min\left(1, \frac{Q^{-}_{ij}}{P^{-}_{ij}}\right), & \text{if}\quad P^{-}_{ij} < 0,\\
        1,& \text{otherwise}.
    \end{cases}
\end{align*}
Taking ${\bm n} = {\bm n}_D=(1,0,0,0)^\top$ (so that ${\bf U}\cdot{\bm n}_D=D({\bf U})$) and using $\epsilon = \min\bigl\{10^{-13},D({\bf U}_{ij}^L)\bigr\}$, the coefficients of the density flux limiter are
{\begin{align*}
    \theta_{i+\frac{1}{2},j}^D = 
    \begin{cases}
        R_{ij}^- & \text{if}\ \widehat{\mathcal{F}}_{i+\frac{1}{2},j}^{A}\cdot{\bm n}_D \geq 0, \\
        R_{i+1,j}^- & \text{if}\ \widehat{\mathcal{F}}_{i+\frac{1}{2},j}^{A}\cdot{\bm n}_D  < 0.
    \end{cases},
    \qquad\qquad
    \theta_{i,j+\frac{1}{2}}^D = 
    \begin{cases}
        R_{ij}^- & \text{if}\ \widehat{\mathcal{G}}_{i,j+\frac{1}{2}}^{A}\cdot{\bm n}_D  \geq 0, \\
        R_{i,j+1}^- & \text{if}\ \widehat{\mathcal{G}}_{i,j+\frac{1}{2}}^{A}\cdot{\bm n}_D  < 0.
    \end{cases},
\end{align*}
where ${\bm n}_D:=(1,0,0,0)^\top$ here.}

The limiter for $q({\bf U})=E-\sqrt{D^2+|{\bm m}|^2}\geq\epsilon$ is a little different. The main difference lies in the dimension of the GQL auxiliary variable ${\bm v}_*\in\mathbb{B}_1({\bf 0})\subseteq\mathbb{R}^2$. The GQL representation of the admissible state set in the two-dimensional case is
\begin{equation*}
    \mathcal{G}_\epsilon^{\tt 2D}=\left\{{\bf U}\in\mathbb{R}^4 \mid D({\bf U})\geq\epsilon,\ {\bf U}\cdot{\bm n_*}\geq \epsilon,\ {\bm n_*}=\left(-\sqrt{1-|{\bm v}_*|^2},\ -{\bm v}_*^\top,\ 1\right)^\top,\ \forall\ {\bm v}_*\in\mathbb{B}_1({\bf 0})\subseteq\mathbb{R}^2\right\}.
\end{equation*}
As in the one-dimensional case, we fix a small parameter $\epsilon>0$ such that the low-order update satisfies ${\bf U}_{ij}^L\in \mathcal{G}_\epsilon^{\tt 2D}$ for all cells. In computation, at each SSP-RK stage we set $\epsilon=\min\bigl\{10^{-13},\ q({\bf U}_{ij}^L)\bigr\}$.
We generalize the stereographic-type transformation in \eqref{1Dvstar_para} to higher dimensions (a rational parameterization of the unit ball) and obtain
\begin{equation}\label{2Dvstar_para}
    {\bm v}_*:=\frac{2{\bm u}_{*}}{1+\|{\bm u}_*\|_2^2} {~~\mbox{ with } \|{\bm u}_*\|_2<1 \quad \implies \quad 
    {\bm n}_*({\bm u}_*)=\left(-\frac{1-\|{\bm u}_*\|_2^2}{1+\|{\bm u}_*\|_2^2},\ -\frac{2{\bm u}_{*}^\top}{1+\|{\bm u}_*\|_2^2},\ 1\right)^\top,}
\end{equation}
which has the same form as the one-dimensional case because we apply the same parameterization for all velocity components. Therefore, it is also applicable to the three-dimensional RHD. Next comes the key step. We need to solve the following optimization problem
\begin{equation*}
    (P_2): \mathcal{L}_{ij} = \min\limits_{{\bm v}_*\in\mathbb{B}_1({\bf 0})\subseteq\mathbb{R}^2}
    \frac{Q^{-}_{ij}}{P^{-}_{ij}}.
\end{equation*}
{
Let ${\bm w}:=({\bm u}_*^\top,1)^\top\in\mathbb{R}^3$. Define $\widehat{A}_x^{\star}:=\Delta y\widetilde{A}_x^{\star}$, $\widehat{A}_y^{\star}:=\Delta x\widetilde{A}_y^{\star}$, and the matrix 
\begin{equation} \label{2DQuadCoeff}
    \widetilde{A}_s^{\star}:=R^{-\top}A_s^{\star} R^{-1} = {\tt sign}(\star)
    \begin{pmatrix}
        \frac{d_s^{\star}}{a}\mathcal{I}_{2\times2} & \frac{d_s^{\star}{\bm b}-a{\bm e}_s^{\star}}{a\sqrt{4ac-|{\bm b}|^2}} \\
        \frac{d_s^{\star}{\bm b}^\top-a({\bm e}_s^{\star})^\top}{a\sqrt{4ac-|{\bm b}|^2}} &
        \frac{4a^2f_s^{\star}-2a{\bm b}\cdot{\bm e}_s^{\star}+d_s^{\star}|{\bm b}|^2}{a\left(4ac-|{\bm b}|^2\right)}
    \end{pmatrix},\qquad s = x,y,\quad \star\in\{+,-\}
\end{equation}
with the coefficients listed in Table \ref{tab:2DQuadCoeff}.
To simplify the notation and make our presentation clearer, we introduce the following matrix and vector:
\begin{equation*}
    \mathbf{M}_d = \begin{pmatrix} 1 & {\bf 0}_d^\top & 1 \\ {\bf 0}_d & 2\mathcal{I}_{d\times d} & {\bf 0}_d \\ -1 & {\bf 0}_d^\top & 1 \end{pmatrix}, \quad 
    \mathbf{S}_d = \begin{pmatrix} \epsilon \\ {\bf 0}_d \\ \epsilon \end{pmatrix},
\end{equation*}
where $d$ is consistent with the space dimension.
Following the idea in the one-dimensional case and applying Lemma \ref{lem:RayleighQ} for $K = 4$, we have

\begin{align}
    \widehat{\mathcal{L}}_{ij}:&=\max\limits_{{\bm w}\in\mathbb{R}^3 \atop {\|{\bm u}_*\|_2<1} }
    \frac{\Delta y\left[\max\left(0,{\bm w}^\top A_x^+{\bm w}\right)+\max\left(0,{\bm w}^\top A_x^-{\bm w}\right)\right]
         +\Delta x\left[\max\left(0,{\bm w}^\top A_y^+{\bm w}\right)+\max\left(0,{\bm w}^\top A_y^-{\bm w}\right)\right]}
         {{\bm w}^\top B{\bm w}}\nonumber \\
    &=\max\limits_{{\bm z}:=R{\bm w}\ \in\ \mathbb{R}^3 \atop {\|{\bm u}_*\|_2<1} }
    \frac{\Delta y\left[\max\left(0,{\bm z}^\top \widetilde{A}_x^+{\bm z}\right)+\max\left(0,{\bm z}^\top \widetilde{A}_x^-{\bm z}\right)\right]
         +\Delta x\left[\max\left(0,{\bm z}^\top \widetilde{A}_y^+{\bm z}\right)+\max\left(0,{\bm z}^\top \widetilde{A}_y^-{\bm z}\right)\right]}
         {{\bm z}^\top {\bm z}} \nonumber \\
    &=\max\limits_{{\bm z}\in\mathbb{R}^3,\ \Vert {\bm z} \Vert_2=1,\ {\|{\bm u}_*\|_2<1} }
    \Big\{\max\left(0,{\bm z}^\top \widehat{A}_x^+{\bm z}\right)
         +\max\left(0,{\bm z}^\top \widehat{A}_x^-{\bm z}\right)
         +\max\left(0,{\bm z}^\top \widehat{A}_y^+{\bm z}\right)
         +\max\left(0,{\bm z}^\top \widehat{A}_y^-{\bm z}\right)\Big\},\nonumber \\
    &=\max\Big\{0,\ 
                \max\limits_{{\bm z}\in\mathbb{R}^3,\ \Vert {\bm z} \Vert_2=1,\ {\|{\bm u}_*\|_2<1} }{\bm z}^\top \widehat{A}_x^+{\bm z},\ \ldots,\ 
                \max\limits_{{\bm z}\in\mathbb{R}^3,\ \Vert {\bm z} \Vert_2=1,\ {\|{\bm u}_*\|_2<1} }{\bm z}^\top \big(\widehat{A}_x^++\widehat{A}_x^-+\widehat{A}_y^++\widehat{A}_y^-\big) {\bm z}\Big\}.\label{16mat2D}
\end{align}
}

\begin{table}[!thb] 
    \renewcommand{\arraystretch}{1.8}
    \centering
    \belowrulesep=0pt
    \aboverulesep=0pt
    \caption{Formulae for coefficients in \eqref{2DQuadCoeff}.}
    \label{tab:2DQuadCoeff}
    \setlength{\tabcolsep}{3mm}{
        \begin{tabular}{c|c|c|c}
            \toprule[1.5pt]
            {} & {$\left(a,{\bm b}^\top,c\right)^\top$} & {$\left(d_x^{\pm},({\bm e}_x^{\pm})^\top,f_x^{\pm}\right)^\top$} & {$\left(d_y^{\pm},({\bm e}_y^{\pm})^\top,f_y^{\pm}\right)^\top$} \\  
            \midrule[1.5pt]
            {$\|{\bm u}_*\|_2<1$} & ${\bf M}_2{\bf U}_{ij}^L-{\bf S}_2$ & ${\bf M}_2\widehat{\mathcal{F}}_{i\pm{1}/{2},j}^{A}$ & ${\bf M}_2\widehat{\mathcal{G}}_{i,j\pm{1}/{2}}^{A}$ \\
            \bottomrule[1.5pt]
        \end{tabular}
        }
\end{table}

\begin{remark}[On restricting to $\|{\bm u}_*\|_2<1$]\label{rem:md-branch}
The multi-dimensional mapping \eqref{2Dvstar_para} (and its 3D analogue) is also two-to-one: ${\bm u}_*$ and ${\bm u}_*/\|{\bm u}_*\|_2^2$ generate the same ${\bm v}_*\in\mathbb{B}_1(0)$ and hence the same normal vector ${\bm n}_*( {\bm v}_*)$.
Therefore the cases $\|{\bm u}_*\|_2<1$ and $\|{\bm u}_*\|_2>1$ are equivalent.
In practice we restrict the optimization to $\|{\bm u}_*\|_2<1$; the other branch can be treated analogously.
\end{remark}

From \eqref{16mat2D}, we see that the solution to $\widehat{\mathcal{L}}_{ij}$ is the maximum of $2^4=16$ maximal eigenvalues of associated matrices. Let ${\bm s}:=(s_1,s_2,s_3,s_4)\in\{0,1\}^4$ be an index vector, and $\widehat{A}_{\bm s}:=s_1\widehat{A}_x^++s_2\widehat{A}_x^-+s_3\widehat{A}_y^++s_4\widehat{A}_y^-$ contains all 16 possible combinations that make a difference in solving \eqref{16mat2D}. The basic model of $\widehat{A}_{\bm s}$ can be expressed as
\begin{equation}\label{basicmodel}
    \widehat{A}_{\bm s}=
    \begin{pmatrix}
        \frac{d}{a}\mathcal{I}_{2\times2} & \frac{d{\bm b}-a{\bm e}}{a\sqrt{4ac-|{\bm b}|^2}} \\
        \frac{d{\bm b}^\top-a{\bm e}^\top}{a\sqrt{4ac-|{\bm b}|^2}} & 
        \frac{4a^2f-2a{\bm b}\cdot{\bm e}+|{\bm b}|^2d}{a(4ac-|{\bm b}|^2)}
    \end{pmatrix} = 
    \begin{pmatrix}
        \frac{d}{a} & 0 & \frac{b_1d-ae_1}{a\sqrt{4ac-|{\bm b}|^2}} \\
        0 & \frac{d}{a} & \frac{b_2d-ae_2}{a\sqrt{4ac-|{\bm b}|^2}} \\
        \frac{b_1d-ae_1}{a\sqrt{4ac-|{\bm b}|^2}} & 
        \frac{b_2d-ae_2}{a\sqrt{4ac-|{\bm b}|^2}} &
        \frac{4a^2f-2a{\bm b}\cdot{\bm e}+|{\bm b}|^2d}{a(4ac-|{\bm b}|^2)}
    \end{pmatrix},
\end{equation}
where ${\bm b}=(b_1,b_2)^\top,\ {\bm e}=(e_1,e_2)^\top.$
Three eigenvalues of $\widehat{A}_{\bm s}$ are
{\begin{equation}\label{basicmodel:evalue}
    \widehat{\lambda}_{\bm s}^{(k)} = \frac{(-1)^{k-1}\sigma_{\bm s}+2af-{\bm b}\cdot{\bm e}+2cd}{4ac-|{\bm b}|^2}\ \ (k=1,2),\quad
    \widehat{\lambda}_{\bm s}^{(3)} = \frac{d}{a},
\end{equation}}
where
\begin{equation} \label{sigma2D}
    \sigma_{\bm s}:=\sqrt{4\Big((af-cd)^2-(af+cd)({\bm b}\cdot{\bm e})+ac|{\bm e}|^2+|{\bm b}|^2df\Big)-(b_1e_2-b_2e_1)^2}.
\end{equation}
{
The non-negativity of the radicand in \eqref{sigma2D} follows directly from the symmetric block structure of \eqref{basicmodel}. Let
\[
\Delta:=4ac-|{\bm b}|^2>0,\qquad
{\bm h}:=\frac{d{\bm b}-a{\bm e}}{a\sqrt{\Delta}},\qquad
 g:=\frac{4a^2f-2a{\bm b}\cdot{\bm e}+|{\bm b}|^2d}{a\Delta}.
\]
Then \eqref{basicmodel} has the eigenvalue $d/a$ in the subspace orthogonal to ${\bm h}$ (with the obvious multiplicity-two degeneracy if ${\bm h}={\bf 0}$), and its two remaining eigenvalues are those of the real symmetric $2\times2$ matrix
\[
\begin{pmatrix} d/a & \|{\bm h}\|_2 \\ \|{\bm h}\|_2 & g \end{pmatrix}.
\]
Consequently,
\[
\sigma_{\bm s}^2=\frac{\Delta^2}{4}\left(\left(g-\frac{d}{a}\right)^2+4\|{\bm h}\|_2^2\right)\ge0,
\]
which is algebraically equivalent to the radicand in \eqref{sigma2D}. Hence $\sigma_{\bm s}$ is real and the eigenvalue formula \eqref{basicmodel:evalue} is well-defined for all admissible states.
}
Since $\sigma_{\bm s}\geq0$, the maximal eigenvalue of $\widehat{A}_{\bm s}$ is
\begin{equation*}
    \widehat{\lambda}_{\bm s}:=\max\left\{\widehat{\lambda}_{\bm s}^{(1)},\ \widehat{\lambda}_{\bm s}^{(3)}\right\}.
\end{equation*}
We are also interested in the optimizer ${\bm w} = ({\bm u}_*^\top, 1)^\top\in\mathbb{R}^3$ when dealing with the additional constraint $\|{\bm u}_*\|_2<1$. It can be found by multiplying the eigenvector ${\bm z}$ associated with the maximal eigenvalue of $\widehat{A}_{\bm s}$ since we apply the substitution ${\bm z}=R{\bm w}$ to express $\widehat{\mathcal{L}}_{ij}$ as a Rayleigh quotient.
\begin{align}
    {\bm w}_{\bm s}^{(1)} &= \frac{1}{\sqrt{c-\frac{|{\bm b}|^2}{4a}}}\left(\frac{{\bm b}^\top}{2a}+\frac{(d{\bm b}^\top-a{\bm e}^\top)\widetilde{\sigma}_{\bm s}^{(1)}}{2a\|a{\bm e}-d{\bm b}\|_2^2},\ 1\right)^\top,\label{basicmodel:evector1}\\
    {\bm w}_{\bm s}^{(2)} &= \frac{1}{\sqrt{c-\frac{|{\bm b}|^2}{4a}}}\left(\frac{{\bm b}^\top}{2a}-\frac{(d{\bm b}^\top-a{\bm e}^\top)\widetilde{\sigma}_{\bm s}^{(2)}}{2a\|a{\bm e}-d{\bm b}\|_2^2},\ 1\right)^\top,\label{basicmodel:evector2}\\
    {\bm w}_{\bm s}^{(3)} &= \frac{1}{\sqrt{a}}\left(-\frac{b_2d-ae_2}{b_1d-ae_1},\ 1,\ 0\right)^\top,\label{basicmodel:evector3}
\end{align}
where
{\begin{equation} \label{basicemodel:evec_sig}
    \widetilde{\sigma}_{\bm s}^{(k)} = a\sigma_{\bm s}+(-1)^k\left(|{\bm b}|^2d+2a^2f-2acd-a{\bm b}\cdot{\bm e}\right),\quad k = 1,2.
\end{equation}}
The maximal eigenvalue may not be a feasible optimal solution if its corresponding eigenvector ${\bm w} = ({\bm u}_*^\top, 1)^\top$ does not satisfy $\|{\bm u}_*\|_2<1$. We list this quantity for feasibility judgment for \eqref{basicmodel:evector1}, \eqref{basicmodel:evector2}, and \eqref{basicmodel:evector3}:
\begin{align}
    \left|{\bm u}_{\bm s}^{(1)}\right|^2&=\left(\frac{b_1}{2a}+\frac{(b_1d-ae_1)\widetilde{\sigma}_{\bm s}^{(1)}}{2a\|a{\bm e}-d{\bm b}\|_2^2}\right)^2+\left(\frac{b_2}{2a}+\frac{(b_2d-ae_2)\widetilde{\sigma}_{\bm s}^{(1)}}{2a\|a{\bm e}-d{\bm b}\|_2^2}\right)^2,\label{basicmodel:generalnorm1}\\
    \left|{\bm u}_{\bm s}^{(2)}\right|^2&=\left(\frac{b_1}{2a}-\frac{(b_1d-ae_1)\widetilde{\sigma}_{\bm s}^{(2)}}{2a\|a{\bm e}-d{\bm b}\|_2^2}\right)^2+\left(\frac{b_2}{2a}-\frac{(b_2d-ae_2)\widetilde{\sigma}_{\bm s}^{(2)}}{2a\|a{\bm e}-d{\bm b}\|_2^2}\right)^2,\label{basicmodel:generalnorm2}\\
    \left|{\bm u}_{\bm s}^{(3)}\right|^2&=\left(-\frac{b_2d-ae_2}{b_1d-ae_1}\right)^2+1^2\geq1,\ \text{and the last component of ${\bm w}_{\bm s}^{(3)}$ is not 1}.\label{basicmodel:generalnorm3}
\end{align}
All the above formulations correspond to the basic model of $\widehat{A}_{\bm s}$. {For a given index vector ${\bm s}_0:=(s_1,s_2,s_3,s_4)^\top$, we can obtain $\widehat{A}_{{\bm s}_0}$ by the corresponding choices of $d$, $e_1$, $e_2$, and $f$. To simplify the notation and make our presentation clearer, we introduce the base coefficient vectors:
\begin{equation*}
    {\bm V}_d:=
    \begin{pmatrix}
        \Delta y\ d_x^+\\
        -\Delta y\ d_x^-\\
        \Delta x\ d_y^+\\
        -\Delta x\ d_y^-
    \end{pmatrix},
    \quad
    {\bm V}_{e_1}:=
    \begin{pmatrix}
        \Delta y\ e_x^{+,1}\\
        -\Delta y\ e_x^{-,1}\\
        \Delta x\ e_y^{+,1}\\
        -\Delta x\ e_y^{-,1}
    \end{pmatrix},
    \quad
    {\bm V}_{e_2}:=
    \begin{pmatrix}
        \Delta y\ e_x^{+,2}\\
        -\Delta y\ e_x^{-,2}\\
        \Delta x\ e_y^{+,2}\\
        -\Delta x\ e_y^{-,2}
    \end{pmatrix},
    \quad
    {\bm V}_f:=
    \begin{pmatrix}
        \Delta y\ f_x^+\\
        -\Delta y\ f_x^-\\
        \Delta x\ f_y^+\\
        -\Delta x\ f_y^-
    \end{pmatrix}.
\end{equation*}
Then the coefficients $d,e_1,e_2,f$ associated with the index vector ${\bm s}_0$ are given by
\begin{equation*}
    c_{\star} = {\bm s}_0\cdot{\bm V}_{c_{\star}} = \sum\limits_{i=1}^4s_i\ ({\bm V}_{c_{\star}})_i,\qquad c_{\star}\in\{d,e_1,e_2,f\}.
\end{equation*}
This compact representation unifies all 16 possible sign combinations.
}

Before the end, a special case should be verified. If the eigenvectors \eqref{basicmodel:evector1} and \eqref{basicmodel:evector2} are not well defined, i.e., $a{\bm e}-d{\bm b}={\bf 0}$, we should check whether it is consistent with the general formulation. In this special case, the basic model $\widehat{A}_{\bm s}$ degenerates into
\begin{equation}\label{basicmodel:special}
    \widehat{A}_{\bm s}=
    \begin{pmatrix}
        \frac{d}{a} & 0 & 0 \\
        0 & \frac{d}{a} & 0 \\
        0 & 0 & \frac{4a^2f-2a{\bm b}\cdot{\bm e}+|{\bm b}|^2d}{a(4ac-|{\bm b}|^2)}
    \end{pmatrix}.
\end{equation}
Three eigenvalues of $\widehat{A}_{\bm s}$ are
\begin{equation}\label{basicmodel:specialevalue}
    \widehat{\lambda}_{\bm s}^{(1)} = \frac{4af-{\bm b}\cdot{\bm e}}{4ac-|{\bm b}|^2},\quad
    \widehat{\lambda}_{\bm s}^{(2)} = \frac{d}{a},\quad
    \widehat{\lambda}_{\bm s}^{(3)} = \frac{d}{a},
\end{equation}
and the maximal eigenvalue is given by
\begin{equation*}
    \widehat{\lambda}_{\bm s} = 
    \begin{cases}
        \widehat{\lambda}_{\bm s}^{(1)} = \frac{4af-{\bm b}\cdot{\bm e}}{4ac-|{\bm b}|^2} & \text{if}\ af>cd\\
        \widehat{\lambda}_{\bm s}^{(2)} = \widehat{\lambda}_{\bm s}^{(3)} = \frac{d}{a} & \text{otherwise}
    \end{cases}.
\end{equation*}
The associated (scaled) eigenvectors are
\begin{align}
    {\bm w}_{\bm s}^{(1)} &= \frac{1}{\sqrt{c-\frac{|{\bm b}|^2}{4a}}}\left(\frac{{\bm b}^\top}{2a},\ 1\right)^\top,\label{basicmodel:specialevector1}\\
    {\bm w}_{\bm s}^{(2)} &= \frac{1}{\sqrt{a}}\left(0,\ 1,\ 0\right)^\top,\label{basicmodel:specialevector2}\\
    {\bm w}_{\bm s}^{(3)} &= \frac{1}{\sqrt{a}}\left(1,\ 0,\ 0\right)^\top.\label{basicmodel:specialevector3}
\end{align}
We also list the quantity for feasibility judgment for \eqref{basicmodel:specialevector1}, \eqref{basicmodel:specialevector2}, and \eqref{basicmodel:specialevector3}:
\begin{align}
    \left|{\bm u}_{\bm s}^{(1)}\right|^2 &= \left(\frac{b_1}{2a}\right)^2+\left(\frac{b_2}{2a}\right)^2=\frac{|{\bm b}|^2}{4a^2},\label{basicmodel:specialnorm1}\\
    \left|{\bm u}_{\bm s}^{(2)}\right|^2 &= 1 \geq 1 \ \text{and the last component of ${\bm w}_{\bm s}^{(2)}$ is not 1},\label{basicmodel:specialnorm2}\\
    \left|{\bm u}_{\bm s}^{(3)}\right|^2 &= 1 \geq 1 \ \text{and the last component of ${\bm w}_{\bm s}^{(3)}$ is not 1}.\label{basicmodel:specialnorm3}
\end{align}
However, computing 16 eigenvalues may be too expensive and complex, and there exist $2^6=64$ cases if we extend \eqref{16mat2D} to the three-dimensional case, which are impractical in application. 
One relaxation approach is to divide $\widehat{\mathcal{L}}_{ij}$ into two parts, corresponding to two one-dimensional cases that we have already solved. Note that
\begin{align*}
    \widehat{\mathcal{L}}_{ij}&=\max\limits_{{\bm z}\in\mathbb{R}^3,\ \Vert {\bm z} \Vert_2=1 \atop {\|{\bm u}_*\|_2<1} }
    \Big\{\max\left(0,{\bm z}^\top \widehat{A}_x^+{\bm z}\right)
         +\max\left(0,{\bm z}^\top \widehat{A}_x^-{\bm z}\right)
         +\max\left(0,{\bm z}^\top \widehat{A}_y^+{\bm z}\right)
         +\max\left(0,{\bm z}^\top \widehat{A}_y^-{\bm z}\right)\Big\}\\
    &\leq \max\limits_{{\bm z}\in\mathbb{R}^3,\ \Vert {\bm z} \Vert_2=1 \atop {\|{\bm u}_*\|_2<1} }
    \Big\{\max\left(0,{\bm z}^\top \widehat{A}_x^+{\bm z}\right)
         +\max\left(0,{\bm z}^\top \widehat{A}_x^-{\bm z}\right)\Big\}+
    \max\limits_{{\bm z}\in\mathbb{R}^3,\ \Vert {\bm z} \Vert_2=1 \atop {\|{\bm u}_*\|_2<1} }
    \Big\{\max\left(0,{\bm z}^\top \widehat{A}_y^+{\bm z}\right)
         +\max\left(0,{\bm z}^\top \widehat{A}_y^-{\bm z}\right)\Big\}.
\end{align*}
For convenience, we denote
\begin{equation*}
    \widehat{\mathcal{L}}_{ij}^{x}:=\max\limits_{{\bm z}\in\mathbb{R}^3,\ \Vert {\bm z} \Vert_2=1,\ {\|{\bm u}_*\|_2<1} }
    \Big\{\max\left(0,{\bm z}^\top \widehat{A}_x^+{\bm z}\right)
         +\max\left(0,{\bm z}^\top \widehat{A}_x^-{\bm z}\right)\Big\}
    =\max\left\{0,\ \widehat{\lambda}_x^+,\ \widehat{\lambda}_x^-,\ \widehat{\lambda}_x^*\right\},
\end{equation*}
\begin{equation*}
    \widehat{\mathcal{L}}_{ij}^{y}:=\max\limits_{{\bm z}\in\mathbb{R}^3,\ \Vert {\bm z} \Vert_2=1,\ {\|{\bm u}_*\|_2<1} }
    \Big\{\max\left(0,{\bm z}^\top \widehat{A}_y^+{\bm z}\right)
         +\max\left(0,{\bm z}^\top \widehat{A}_y^-{\bm z}\right)\Big\}
    =\max\left\{0,\ \widehat{\lambda}_y^+,\ \widehat{\lambda}_y^-,\ \widehat{\lambda}_y^*\right\},
\end{equation*}
where $\lambda_s^+$, $\lambda_s^-$, and $\lambda_s^*$ represent the maximal eigenvalues of $\widehat{A}_s^+$, $\widehat{A}_s^-$, and $\widehat{A}_s^++\widehat{A}_s^-$ ($s=x$ or $y$), respectively. We then obtain the final relaxation formulation
\begin{equation} \label{relax2D}
    \widehat{\mathcal{L}}_{ij}^{\mathcal{R}}:=\widehat{\mathcal{L}}_{ij}^{x}+\widehat{\mathcal{L}}_{ij}^{y}=
    \max\left\{0,\ \widehat{\lambda}_x^+,\ \widehat{\lambda}_x^-,\ \widehat{\lambda}_x^*\right\}+\max\left\{0,\ \widehat{\lambda}_y^+,\ \widehat{\lambda}_y^-,\ \widehat{\lambda}_y^*\right\},
\end{equation}
which is significantly more efficient than \eqref{16mat2D}.
{Here $\widehat{\mathcal{L}}$ denotes the maximal denominator ratio in the optimization problem, not the limiter coefficient itself. Thus, for the relaxed 2D estimator, the corresponding cell-wise $q$-limiting factor is
\[
\mathcal{L}_{ij}^{q,\mathcal R}:=
\begin{cases}
1, & \widehat{\mathcal{L}}_{ij}^{\mathcal R}=0,\\[1mm]
\min\left\{1,\dfrac{\Delta x\Delta y}{\Delta t}\left(\widehat{\mathcal{L}}_{ij}^{\mathcal R}\right)^{-1}\right\}, & \widehat{\mathcal{L}}_{ij}^{\mathcal R}>0,
\end{cases}
\]
with the exact estimator obtained by replacing $\widehat{\mathcal{L}}_{ij}^{\mathcal R}$ by the full 16-matrix maximum in \eqref{16mat2D}. The interface factors are then assigned by the usual adjacent-cell minimum, e.g., $\theta^q_{i+1/2,j}=\min\{\mathcal{L}_{ij}^{q},\mathcal{L}_{i+1,j}^{q}\}$ and analogously in the $y$ direction.}
\begin{remark}[On the relaxed multi-dimensional estimators]\label{rem:relax-md}
The relaxed estimators \eqref{relax2D} and \eqref{relax3D} are \emph{sufficient} for PCP enforcement.
Indeed, they replace the exact maximization over all sign combinations by a sum of directional upper bounds, which yields a conservative underestimate of the exact admissible factor (and hence a more restrictive limiter).
As a result, the PCP proof remains valid, while the number of eigenvalue evaluations is substantially reduced; see Section~\ref{sec:cost} {for the trade-off analysis of computational cost and numerical dissipation}.
\end{remark}

Finally, we briefly explain how to implement our flux limiter in the three-dimensional RHD, which is very similar to the two-dimensional case, and only the different parts will be listed. 
The FCT finite difference scheme for three-dimensional uniform rectangular meshes can be formulated by
\begin{equation}\label{uLF3DRHD}
    {\bf U}_{ijk}^{L} := {\bf U}_{ijk}^n
-\frac{\Delta t}{\Delta x}\left(\widehat{\mathcal{F}}_{i+\frac{1}{2},j,k}^L-\widehat{\mathcal{F}}_{i-\frac{1}{2},j,k}^L\right)
-\frac{\Delta t}{\Delta y}\left(\widehat{\mathcal{G}}_{i,j+\frac{1}{2},k}^L-\widehat{\mathcal{G}}_{i,j-\frac{1}{2},k}^L\right)
-\frac{\Delta t}{\Delta z}\left(\widehat{\mathcal{H}}_{i,j,k+\frac{1}{2}}^L-\widehat{\mathcal{H}}_{i,j,k-\frac{1}{2}}^L\right),
\end{equation}
\begin{align*}
    {\bf U}_{ijk}^{n+1} = {\bf U}_{ijk}^L
                        &-\frac{\Delta t}{\Delta x}\left(\theta_{i+\frac{1}{2},j,k}\widehat{\mathcal{F}}_{i+\frac{1}{2},j,k}^A-\theta_{i-\frac{1}{2},j,k}\widehat{\mathcal{F}}_{i-\frac{1}{2},j,k}^A\right)\\
                        &-\frac{\Delta t}{\Delta y}\left(\theta_{i,j+\frac{1}{2},k}\widehat{\mathcal{G}}_{i,j+\frac{1}{2},k}^A-\theta_{i,j-\frac{1}{2},k}\widehat{\mathcal{G}}_{i,j-\frac{1}{2},k}^A\right)
                        -\frac{\Delta t}{\Delta z}\left(\theta_{i,j,k+\frac{1}{2}}\widehat{\mathcal{H}}_{i,j,k+\frac{1}{2}}^A-\theta_{i,j,k-\frac{1}{2}}\widehat{\mathcal{H}}_{i,j,k-\frac{1}{2}}^A\right),
\end{align*}
where ${\widehat{\mathcal{F}}}$, ${\widehat{\mathcal{G}}}$, and ${\widehat{\mathcal{H}}}$ represent the numerical fluxes in $x$, $y$, and $z$ direction, respectively. 
The three-dimensional case can share the same framework as the two-dimensional one, except for the definitions of the following quantities related to the spatial dimension
\begin{align*}
    P^{-}_{ijk} &= \Delta y\Delta z\left[\min\left(0, -\widehat{\mathcal{F}}^{A}_{i+\frac{1}{2},j,k}\cdot{\bm n}_*\right) + \min\left(0, \widehat{\mathcal{F}}^{A}_{i-\frac{1}{2},j,k}\cdot{\bm n}_*\right)\right]\\
                &+\Delta x\Delta z\left[\min\left(0, -\widehat{\mathcal{G}}^{A}_{i,j+\frac{1}{2},k}\cdot{\bm n}_*\right) + \min\left(0, \widehat{\mathcal{G}}^{A}_{i,j-\frac{1}{2},k}\cdot{\bm n}_*\right)\right]\\
                &+\Delta x\Delta y\left[\min\left(0, -\widehat{\mathcal{H}}^{A}_{i,j,k+\frac{1}{2}}\cdot{\bm n}_*\right) + \min\left(0, \widehat{\mathcal{H}}^{A}_{i,j,k-\frac{1}{2}}\cdot{\bm n}_*\right)\right].\\
    Q^{-}_{ijk} &= \frac{\Delta x\Delta y\Delta z}{\Delta t}\left(\epsilon - {\bf U}^{L}_{ijk}\cdot{\bm n}_*\right),\quad {\bf U}^{L}_{ijk}\ \text{is defined in }\eqref{uLF3DRHD}.\\
    R^{-}_{ijk} &= 
    \begin{cases}
        \min\left(1, \frac{Q^{-}_{ijk}}{P^{-}_{ijk}}\right), & \text{if}\quad P^{-}_{ijk} < 0,\\
        1,& \text{otherwise}.
    \end{cases}
\end{align*}

The formulations of the model matrix $\widehat{A}_{\bm s}$ and its associated eigenpairs share the same form as \eqref{basicmodel}, \eqref{basicmodel:evalue}, \eqref{basicmodel:evector1}, \eqref{basicmodel:evector2}, and \eqref{basicmodel:evector3} but only differ in the size of ${\bm b}$ and ${\bm e}$. 
\begin{equation*}
    \widehat{A}_{\bm s}=
    \begin{pmatrix}
        \frac{d}{a}\mathcal{I}_{3\times 3} & \frac{d{\bm b}-a{\bm e}}{a\sqrt{4ac-|{\bm b}|^2}} \\
        \frac{d{\bm b}^\top-a{\bm e}^\top}{a\sqrt{4ac-|{\bm b}|^2}} & 
        \frac{4a^2f-2a{\bm b}\cdot{\bm e}+|{\bm b}|^2d}{a(4ac-|{\bm b}|^2)}
    \end{pmatrix} =
    \begin{pmatrix}
        \frac{d}{a} & 0 & 0 & \frac{b_1d-ae_1}{a\sqrt{4ac-|{\bm b}|^2}} \\
        0 & \frac{d}{a} & 0 & \frac{b_2d-ae_2}{a\sqrt{4ac-|{\bm b}|^2}} \\
        0 & 0 & \frac{d}{a} & \frac{b_3d-ae_3}{a\sqrt{4ac-|{\bm b}|^2}} \\
        \frac{b_1d-ae_1}{a\sqrt{4ac-|{\bm b}|^2}} & 
        \frac{b_2d-ae_2}{a\sqrt{4ac-|{\bm b}|^2}} &
        \frac{b_3d-ae_3}{a\sqrt{4ac-|{\bm b}|^2}} &
        \frac{4a^2f-2a{\bm b}\cdot{\bm e}+|{\bm b}|^2d}{a(4ac-|{\bm b}|^2)}
    \end{pmatrix},
\end{equation*}
where ${\bm b}=(b_1,b_2,b_3)^\top,\ {\bm e}=(e_1,e_2,e_3)^\top.$
\begin{table}[t]
    \centering
    \caption{Summary of the cost of the $q$-limiter parameter estimation per cell. ``Exact'' refers to the estimator based on the full maximization in Lemma~\ref{lem:RayleighQ}; ``relaxed'' refers to \eqref{relax2D} in 2D and \eqref{relax3D} in 3D. }
    \label{tab:cost}
    \begin{tabular}{cccc}
        \toprule
        Dimension & matrix size & \# eigenproblems (exact) & \# eigenproblems (relaxed)\\
        \midrule
        1D & $2\times2$ & $3$  & --\\
        2D & $3\times3$ & $16$ & $6$\\
        3D & $4\times4$ & $64$ & $9$\\
        \bottomrule
    \end{tabular}
\end{table}
There is one more repeated eigenvalue compared with \eqref{basicmodel:evalue}:
{\begin{equation*}
    \widehat{\lambda}_{\bm s}^{(k)} = \frac{(-1)^{k-1}\sigma_{\bm s}+2af-{\bm b}\cdot{\bm e}+2cd}{4ac-|{\bm b}|^2}\ \ (k=1,2),\quad
    \widehat{\lambda}_{\bm s}^{(3)} = \widehat{\lambda}_{\bm s}^{(4)} = \frac{d}{a},
\end{equation*}}
where
\begin{align}
    &\sigma_{\bm s}:=\sqrt{4\Big((af-cd)^2-(af+cd)({\bm b}\cdot{\bm e})+ac|{\bm e}|^2+|{\bm b}|^2df\Big)-{|{\bm b}\times{\bm e}|^2}},\label{sigma3D}\\
    &{{\bm b}\times{\bm e}=\left(b_2e_3-b_3e_2,\ b_3e_1-b_1e_3,\ b_1e_2-b_2e_1\right)^\top.\nonumber}
\end{align}
{Similarly, the radicand of $\sigma_{\bm s}$ in \eqref{sigma3D} is non-negative by the same block-structure argument. With $\Delta:=4ac-|{\bm b}|^2>0$, ${\bm h}:=(d{\bm b}-a{\bm e})/(a\sqrt{\Delta})$, and $g:=(4a^2f-2a{\bm b}\cdot{\bm e}+|{\bm b}|^2d)/(a\Delta)$, the $4\times4$ matrix above has the repeated eigenvalue $d/a$ on the two-dimensional subspace orthogonal to ${\bm h}$, while the two non-repeated eigenvalues are determined by the same real symmetric $2\times2$ block with entries $d/a$, $\|{\bm h}\|_2$, and $g$. Thus
\[
\sigma_{\bm s}^2=\frac{\Delta^2}{4}\left(\left(g-\frac{d}{a}\right)^2+4\|{\bm h}\|_2^2\right)\ge0,
\]
which is equivalent to the radicand in \eqref{sigma3D}. This guarantees the well-posedness of the 3D eigenvalue formulas.}
The associated eigenvectors are
\begin{align*}
    {\bf w}_{\bf s}^{(1)} &= \frac{1}{\sqrt{c-\frac{|{\bm b}|^2}{4a}}}\left(
    \frac{b_1}{2a}+\frac{(b_1d-ae_1)\widetilde{\sigma}_{\bf s}^{(1)}}{2a\|a{\bm e}-d{\bm b}\|_2^2},
    \frac{b_2}{2a}+\frac{(b_2d-ae_2)\widetilde{\sigma}_{\bf s}^{(1)}}{2a\|a{\bm e}-d{\bm b}\|_2^2},
    \frac{b_3}{2a}+\frac{(b_3d-ae_3)\widetilde{\sigma}_{\bf s}^{(1)}}{2a\|a{\bm e}-d{\bm b}\|_2^2},
    1\right)^\top,\\
    {\bf w}_{\bf s}^{(2)} &= \frac{1}{\sqrt{c-\frac{|{\bm b}|^2}{4a}}}\left(
    \frac{b_1}{2a}-\frac{(b_1d-ae_1)\widetilde{\sigma}_{\bf s}^{(2)}}{2a\|a{\bm e}-d{\bm b}\|_2^2},
    \frac{b_2}{2a}-\frac{(b_2d-ae_2)\widetilde{\sigma}_{\bf s}^{(2)}}{2a\|a{\bm e}-d{\bm b}\|_2^2},
    \frac{b_3}{2a}-\frac{(b_3d-ae_3)\widetilde{\sigma}_{\bf s}^{(2)}}{2a\|a{\bm e}-d{\bm b}\|_2^2},
    1\right)^\top,\\
    {\bf w}_{\bf s}^{(3)} &= \frac{1}{\sqrt{a}}\left(-\frac{b_2d-ae_2}{b_1d-ae_1},1,0,0\right)^\top,\\
    {\bf w}_{\bf s}^{(4)} &= \frac{1}{\sqrt{a}}\left(-\frac{b_3d-ae_3}{b_1d-ae_1},0,1,0\right)^\top,
\end{align*}
where $\widetilde{\sigma}_{\bf s}^{(1)}$ and $\widetilde{\sigma}_{\bf s}^{(2)}$ share the same definition in \eqref{basicemodel:evec_sig}.
The judgment for the feasibility of eigenvectors and the special case of $\widehat{A}_{\bm s}$ are very similar to the two-dimensional case, and hence we omit them here.

Since solving $2^6=64$ Rayleigh quotients is impractical, we suggest using the relaxation approach similar to \eqref{relax2D}, which is also adopted and shown in the numerical results.
\begin{equation}\label{relax3D}
    \widehat{\mathcal{L}}_{ijk}^{\mathcal{R}}:=
    \sum_{s\in\{x,y,z\}}\max\Bigl\{0,\ \widehat{\lambda}_{\bf s}^{+},\ \widehat{\lambda}_{\bf s}^{-},\ \widehat{\lambda}_{\bf s}^{*}\Bigr\},
\end{equation}
where $\widehat{\lambda}_{\bf s}^{+}$, $\widehat{\lambda}_{\bf s}^{-}$ and $\widehat{\lambda}_{\bf s}^{*}$ denote the maximal eigenvalues of $\widehat{A}_{\bf s}^{+}$, $\widehat{A}_{\bf s}^{-}$ and $\widehat{A}_{\bf s}^{+}+\widehat{A}_{\bf s}^{-}$, respectively.
The relaxed estimator \eqref{relax3D} reduces the number of eigenvalue evaluations from $2^6=64$ to $3\times 3=9$ per cell, and the corresponding 3D cell-wise factor is obtained analogously as $\mathcal{L}_{ijk}^{q,\mathcal R}=1$ when $\widehat{\mathcal{L}}_{ijk}^{\mathcal R}=0$, and otherwise $\mathcal{L}_{ijk}^{q,\mathcal R}=\min\{1,(\Delta x\Delta y\Delta z)/(\Delta t\,\widehat{\mathcal{L}}_{ijk}^{\mathcal R})\}$.

\subsection{Computational cost and implementation remarks}\label{sec:cost}
The new estimator reduces the determination of the $q$-limiting factor to a small number of symmetric eigenvalue problems of size $(d{+}1)\times(d{+}1)$.
In 1D, the maximization is reduced to three $2\times2$ eigenvalue problems, and in 2D and 3D, the exact estimators require $2^4=16$ and $2^6=64$ eigenvalue evaluations per cell, respectively.
The relaxed variants \eqref{relax2D} and \eqref{relax3D} significantly reduce this number (see Table \ref{tab:cost} for details) while remaining robust in all tests reported in Section~\ref{Sec: Results}.

In addition, the symmetric matrices arising from the GQL parameterization have a special low-rank structure, and their maximal eigenvalues (and eigenvectors when needed) can be evaluated by closed-form formulas; see \eqref{1Deig1}--\eqref{1Deig2} in 1D and \eqref{basicmodel:evalue} and \eqref{basicmodel:evector1} in 2D/3D.
This avoids iterative root finding and makes the limiter well suited for large-scale simulations.

{As noted in Remark~\ref{rem:relax-md}, the relaxed estimators are conservative lower bounds on the exact limiting factor, which introduces a small amount of additional numerical dissipation in theory. However, extensive numerical benchmarks in Section~\ref{Sec: Results} demonstrate that this dissipation is negligible or acceptable in practice. This minimal accuracy loss is more than offset by the dramatic reduction in computational cost. }

This completes the description of our PCP flux limiting framework. To facilitate reproducibility, \ref{appendix:code} provides complete pseudocode, and Section \ref{Sec: Results} lists all computational parameters, including grids, CFL numbers, final times, etc., used in the experiments.


\section{Numerical results} \label{Sec: Results}
\setcounter{equation}{0}

\begin{table}[!thb] 
    \renewcommand{\arraystretch}{1.5}
    \centering
    \belowrulesep=0pt
    \aboverulesep=0pt
    \caption{Numerical errors and convergence rates in $\rho$ at different grid resolutions.}
    \label{table:Ex5.2.0}
    \setlength{\tabcolsep}{3mm}{
        \begin{tabular}{c|cccccc}
            \toprule[1.5pt]
            \multirow{2}{*}{Flux limiter } &
            \multirow{2}{*}{$N_x\times N_y$} &
            \multicolumn{2}{c}{$L^1$ norm} &
            \multicolumn{2}{c}{$L^2$ norm}\\
            \cmidrule(r){3-4} \cmidrule(r){5-6}
            & & error & order &  error & order\\    
            \midrule[1.5pt]
            \multirow{6}{*}{Without relaxation \eqref{16mat2D}} & $8\times8$ & 1.7455e-02 & -- &  1.8592e-02 & --    \\
            & $16\times16$ & 8.6992e-04 & 4.3266 & 1.0609e-03 & 4.1313 \\
            & $32\times32$ & 2.8496e-05 & 4.9320 & 3.3626e-05 & 4.9795 \\
            & $64\times64$ & 8.6622e-07 & 5.0399 & 1.0034e-06 & 5.0667 \\
            & $128\times128$ & 2.6647e-08 & 5.0227 & 3.0419e-08 & 5.0437 \\
            & $256\times256$ & 8.3183e-10 & 5.0015 & 9.3642e-10 & 5.0217 \\
            \midrule
            \multirow{6}{*}{With relaxation \eqref{relax2D}} & $8\times8$ & 1.7455e-02 & -- &  1.8592e-02 & --    \\
            & $16\times16$ & 8.6992e-04 & 4.3266 & 1.0609e-03 & 4.1313 \\
            & $32\times32$ & 2.8496e-05 & 4.9320 & 3.3626e-05 & 4.9795 \\
            & $64\times64$ & 8.6622e-07 & 5.0399 & 1.0034e-06 & 5.0667 \\
            & $128\times128$ & 2.6647e-08 & 5.0227 & 3.0419e-08 & 5.0437 \\
            & $256\times256$ & 8.3183e-10 & 5.0015 & 9.3642e-10 & 5.0217 \\
            \bottomrule[1.5pt]
        \end{tabular}
    }
\end{table}

In this section, we present a series of numerical experiments to validate the performance of {the proposed flux-limited high-order accurate method} for the one-dimensional (1D), two-dimensional (2D), and three-dimensional (3D) RHD. 
The primary objective is to verify that the implemented flux limiter preserves the designed spatial accuracy when coupled with the classical fifth-order WENO (WENO5) reconstruction. This is examined using smooth problems where the formal convergence order can be measured. 
A key focus of the tests is to demonstrate the robustness granted by the flux limiter's ability to enforce the admissible states of the relativistic system. Several challenging benchmark problems, including strong shocks, ultra-relativistic Riemann problems, and multi-dimensional jets, are simulated. Without the limiting procedure, the standard WENO5 would break down due to the generation of nonphysical solutions. Furthermore, comparison results will be shown between our novel flux limiter and the {\tt Wu-Tang} limiter \cite{WuTang2015}. These experiments, therefore, underscore the necessity of the PCP mechanism and simultaneously assess the resolution in capturing complex wave structures. 
For time integration, we employ the third-order strong-stability-preserving Runge--Kutta method.
The CFL number is set to 0.4, and the ideal equation of state \eqref{ID-EOS} is used in all cases, with the adiabatic index set to $\Gamma=5/3$, unless otherwise specified.

\begin{figure}[!htb]
	\centering
	\begin{subfigure}[t]{.45\textwidth}
		\centering
		\includegraphics[width=1\textwidth]{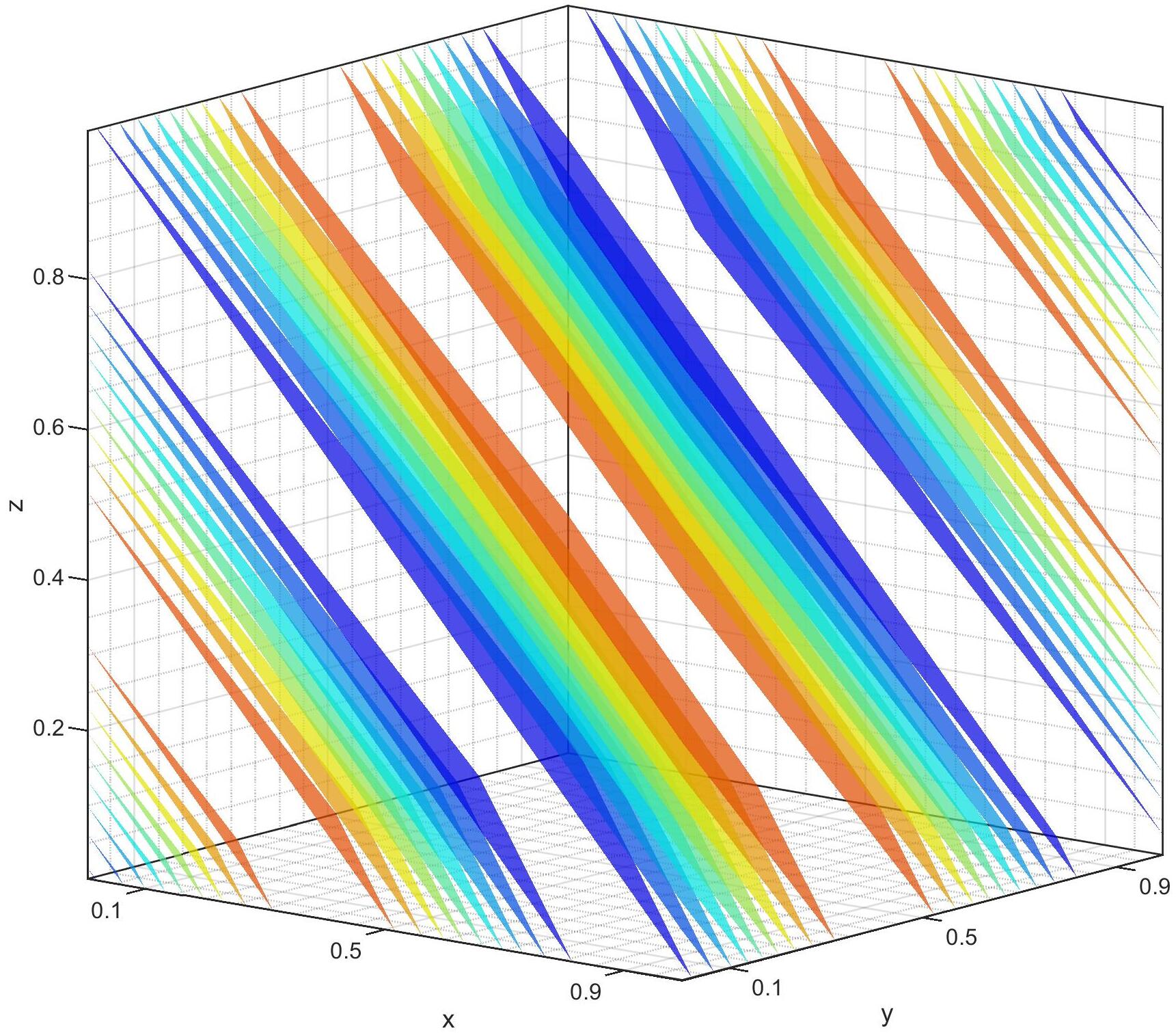}
	\end{subfigure}
    \begin{subfigure}[t]{.45\textwidth}
		\centering
		\includegraphics[width=1\textwidth]{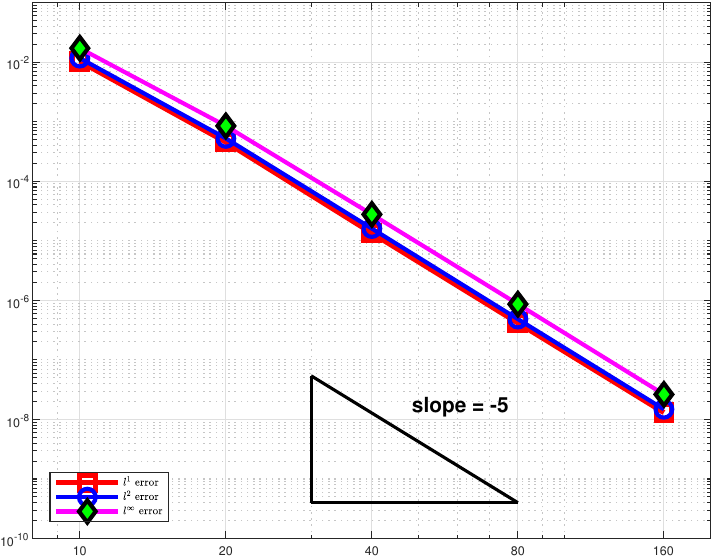}
	\end{subfigure}
	\caption{Example \ref{Ex5.2.0}: The numerical results obtained using WENO5 with the proposed PCP flux limiter. Left: 10 iso-surfaces of $\rho$ equally spaced from 0 to 2; right: $l^1$, $l^2$, and $l^{\infty}$ errors in $\rho$ at different grid resolutions.}\label{Fig:Ex6.2.0}
\end{figure}

\begin{example}[Accuracy test]\label{Ex5.2.0}
This example tests a smooth sine wave propagating periodically in the domain $[0,1]^k$, $k=2,3$, aiming to investigate the accuracy of our proposed scheme in two and three dimensions. {In the algorithm implementation, we explore two strategies for handling the flux limiter parameters: a non-relaxed version \eqref{16mat2D} and a relaxed version \eqref{relax2D}.} 
{In all smooth-solution convergence tests, we follow the standard WENO accuracy-verification strategy of Jiang and Shu \cite{jiang1996efficient}, also used in \cite{zhang2012positivity, WuTang2015}. Although the SSP-RK3 time discretization is third-order accurate, the time step is deliberately reduced according to
\[
\Delta t=C\left(\sum_{\ell=1}^k \frac{1}{\Delta x_\ell}\right)^{-p/q},\qquad p=5,\quad q=3,
\]
with $C=0.4$ and $k$ the space dimension of the smooth test. Therefore $O(\Delta t^3)=O(h^5)$ on uniform grids, so the temporal error does not mask the fifth-order spatial convergence being measured.}
For the 2D case, the initial conditions are given by 
$$
\textbf{V}(x,y,0) = \left(1+0.999\ {\rm sin}(2\pi(x+y)),\ 0.99/\sqrt{2},\ 0.99/\sqrt{2},\ 0.001\right)^\top,
$$
with the exact solution of the rest-mass density 
\begin{equation} \label{exact2D}
    \rho(x,y,t) = 1+0.999\ {\rm sin}\left(2\pi(x+y-0.99\sqrt{2}t)\right).
\end{equation}
Uniform grids are used with mesh size $\Delta x=\Delta y = \frac{1}{N}$, where $N \in \{8,16,32,64,128,256\}$.
The time step-size is chosen to match the spatial accuracy: $\Delta t = {0.4}\left({\frac{1}{\Delta x}+\frac{1}{\Delta y}}\right)^{-\frac{5}{3}}$. We report the $L^1$ and $L^2$ errors at $t = 0.1$ for the rest-mass density and corresponding orders obtained by the WENO5 scheme combined with flux limiters \eqref{16mat2D} and \eqref{relax2D} in Table \ref{table:Ex5.2.0}. The results show that {the proposed scheme} achieves fifth-order accuracy, consistent with its theoretical order. In this smooth setting the physical constraints are not active, and the computed limiting factors stay very close to one, so the limiter does not degrade the design accuracy. 
For the 3D case, the initial data are
$$
\textbf{V}(x,y,z,0) = \left(1+0.999\ {\rm sin}(2\pi(x+y+z)),\ 0.99/\sqrt{3},\ 0.99/\sqrt{3},\ 0.99/\sqrt{3},\ 0.001\right)^\top.
$$
The exact solution of the rest-mass density is
\begin{equation*}
    \rho(x,y,z,t) = 1+0.999\ {\rm sin}\left(2\pi(x+y+z-0.99\sqrt{3}t)\right).
\end{equation*}
{Due to the complexity of directly computing and comparing 64 eigenvalues in 3D, only the relaxed version of the flux limiter is employed here. }
To investigate the spatial accuracy, we set the mesh size as $\Delta x=\Delta y=\Delta z = \frac{1}{N}$ with varying  $N \in \{10,20,40,80,160\}$.
The time step is set as $\Delta t = {0.4}\left({\frac{1}{\Delta x}+\frac{1}{\Delta y}+\frac{1}{\Delta z}}\right)^{-\frac{5}{3}}$. In Fig. \ref{Fig:Ex6.2.0}, we present a 3D iso-surface plot of the rest-mass density using the numerical solution on a $160\times160\times160$ grid, and show the log-log graphs of the $L^1$, $L^2$, and $L^\infty$ errors versus $N$ at $t = 0.1$. The slopes of all error lines are close to -5, indicating that the proposed scheme exhibits fifth-order spatial convergence in 3D case, which aligns with theoretical expectations.
\end{example}

\begin{figure}[!htb]
	\centering
	\begin{subfigure}[t]{.48\textwidth}
		\centering
		\includegraphics[width=1\textwidth]{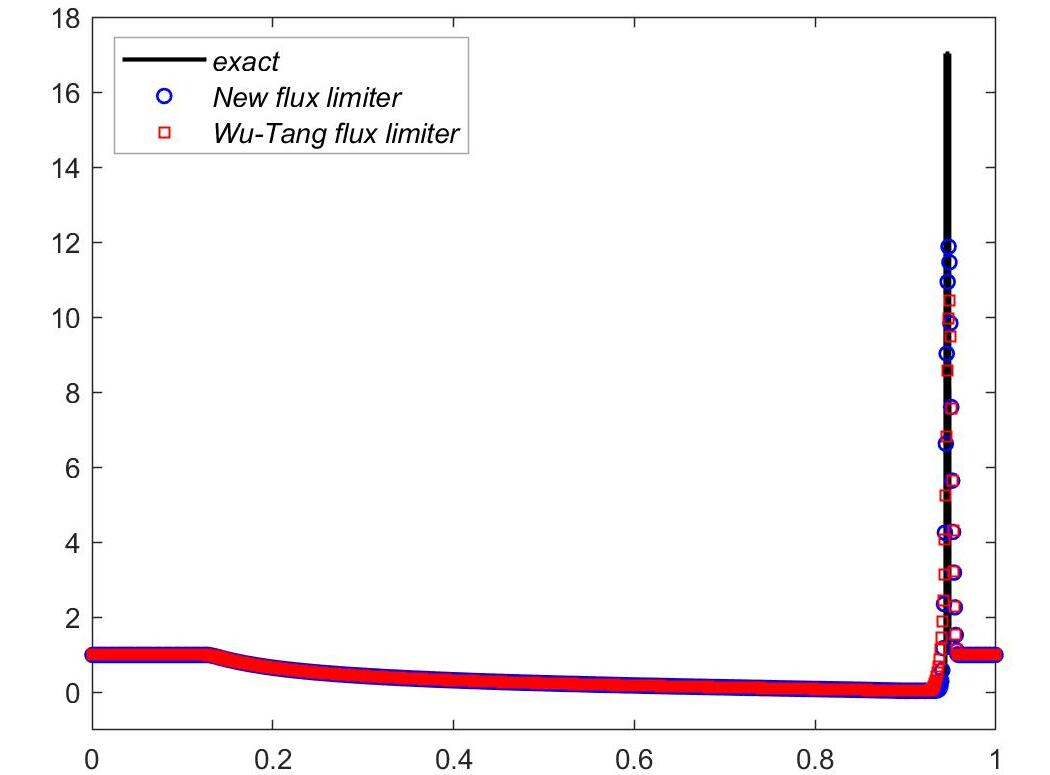}
		\caption{$\rho$}
	\end{subfigure}
    \begin{subfigure}[t]{.48\textwidth}
		\centering
		\includegraphics[width=1\textwidth]{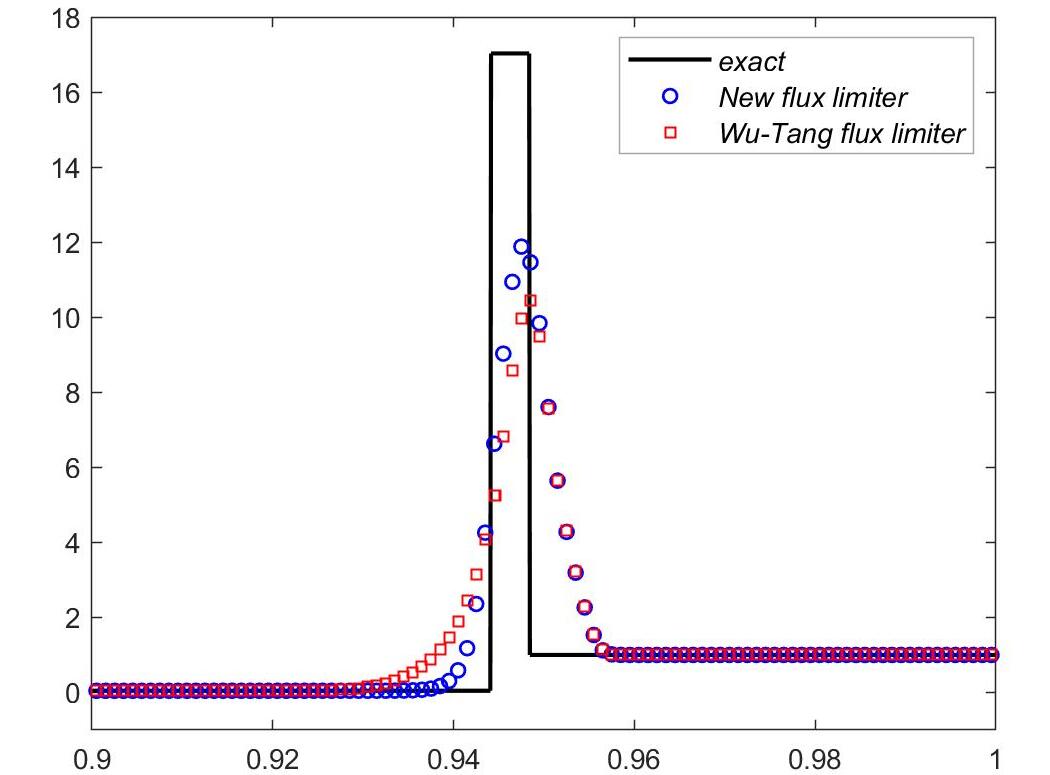}
		\caption{Close-up of $\rho$}
	\end{subfigure}
    \begin{subfigure}[t]{.48\textwidth}
		\centering
		\includegraphics[width=1\textwidth]{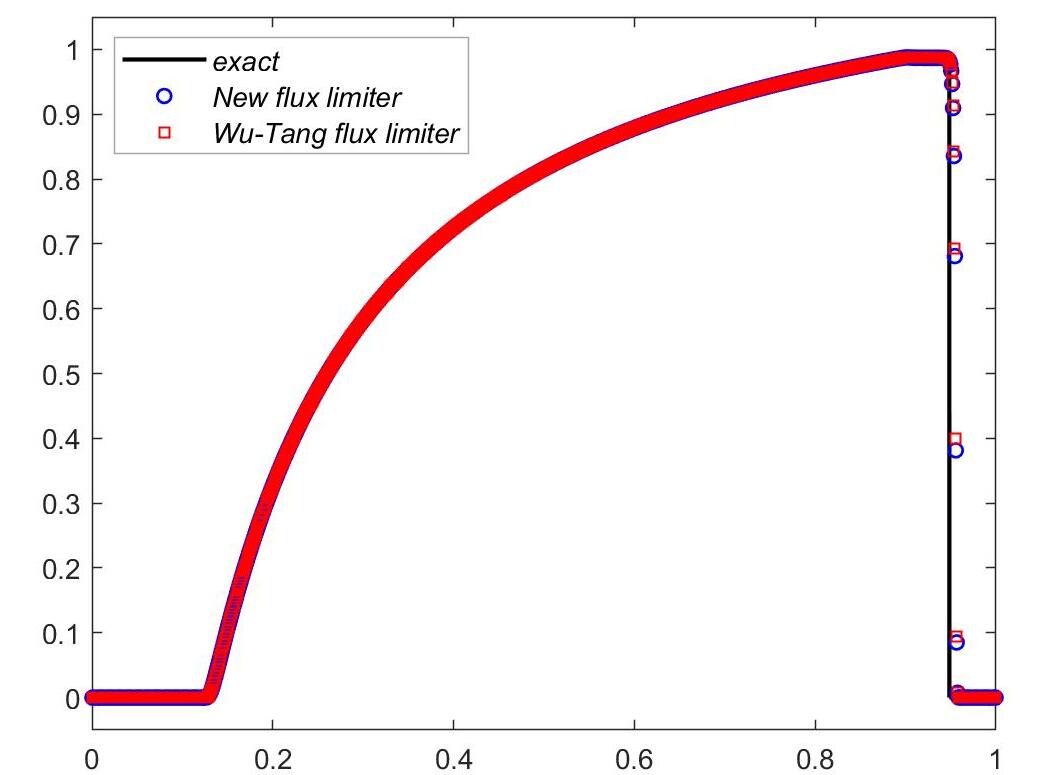}
		\caption{$v_1$}
	\end{subfigure}
    \begin{subfigure}[t]{.48\textwidth}
		\centering
		\includegraphics[width=1\textwidth]{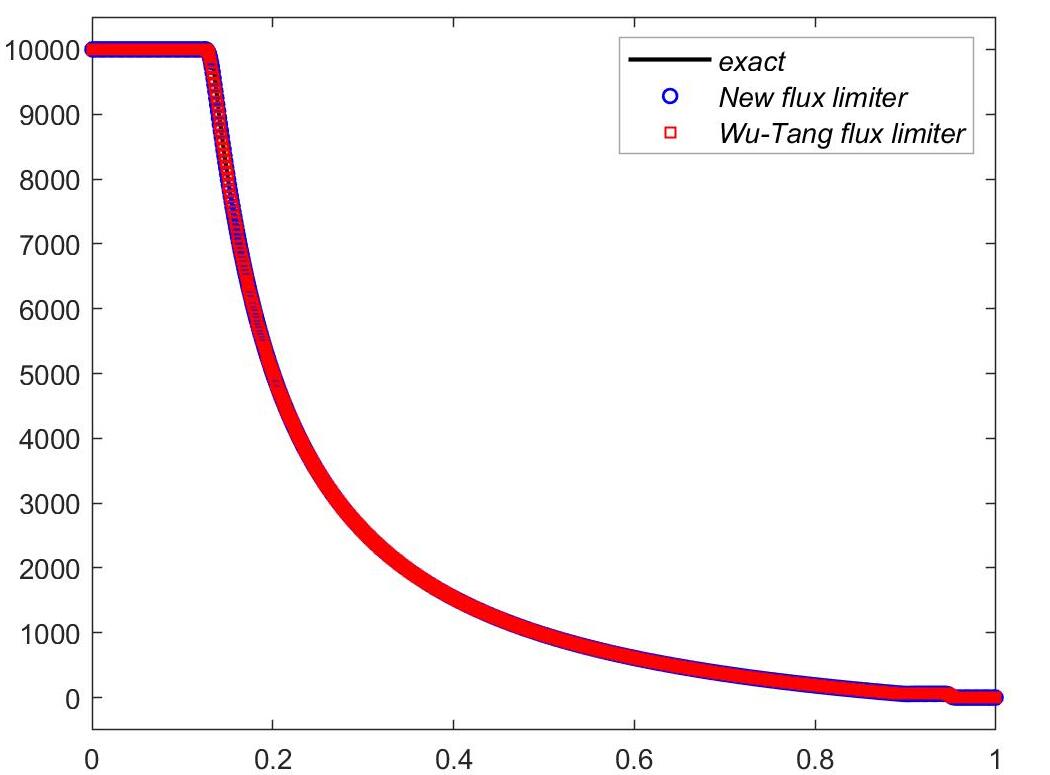}
		\caption{$p$}
	\end{subfigure}
	\caption{Example \ref{Ex:RCS}: Numerical results for $\rho,\ v_1,\ p$ obtained using WENO5 with the relaxed PCP flux limiter at $t = 0.45$. }\label{Fig:ExRCS}
\end{figure}

{
For discontinuous solutions such as Riemann problems, formal convergence order does not apply, and the error is concentrated in a narrow region around shocks and contact discontinuities. The primary goal of numerical methods for such problems is to produce sharp, non-oscillatory shock profiles while preserving physical constraints.
}

\begin{figure}[!htb]
	\centering
	\begin{subfigure}[t]{.48\textwidth}
		\centering
		\includegraphics[width=1\textwidth]{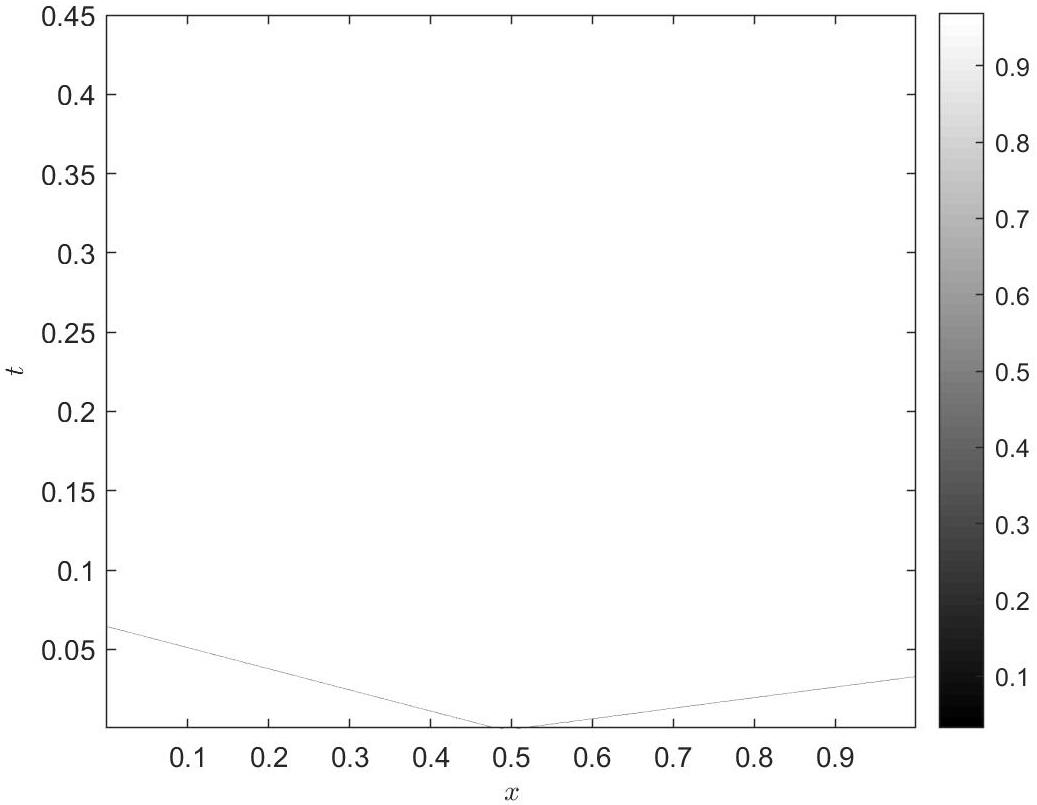}
		\caption{Example \ref{Ex:RCS}} \label{Fig:RCSLimiterMap}
	\end{subfigure}
    \begin{subfigure}[t]{.48\textwidth}
		\centering
		\includegraphics[width=1\textwidth]{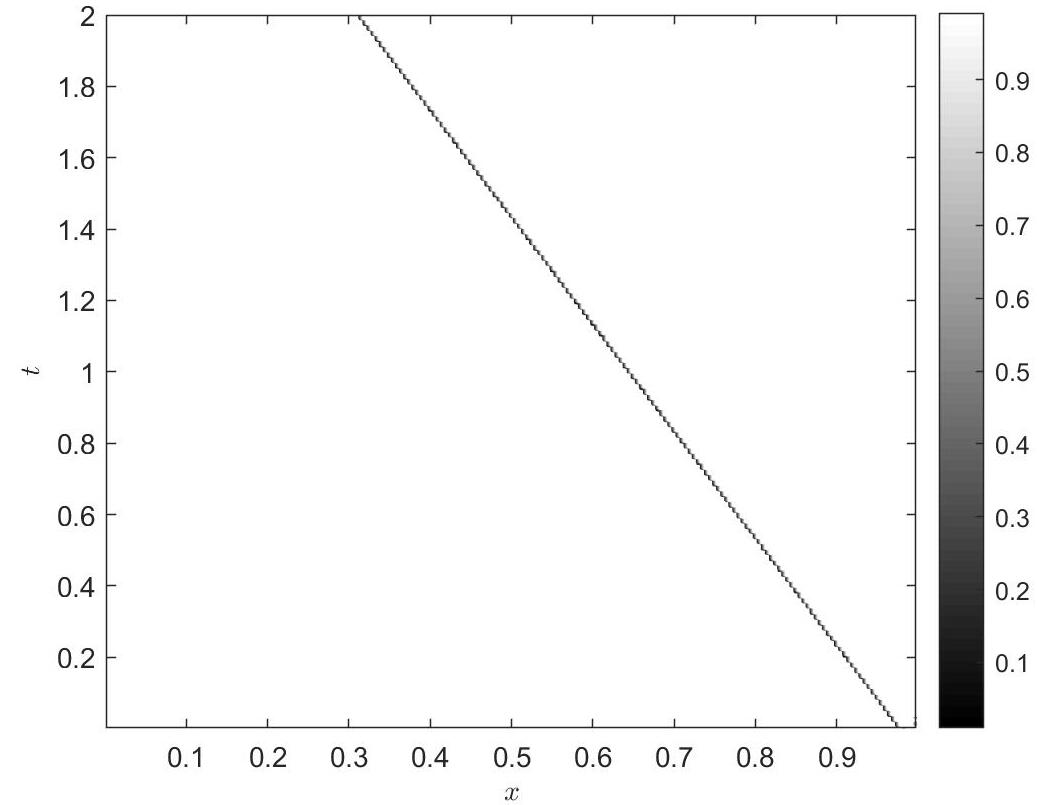}
		\caption{Example \ref{Ex:SH}} \label{Fig:SHLimiterMap}
	\end{subfigure}
	\caption{Spatiotemporal distribution of the proposed flux limiter coefficient $\theta(x,t)$ computed with WENO5. Activation regions align with shock discontinuities and track wave trajectories. The color scale represents the value of $\theta_{i+1/2}$, ranging from 0 to 1. }\label{Fig:1DLimiterActivation}
\end{figure}

\begin{example}{(1D Riemann problem)} \label{Ex:RCS}
We simulate a demanding Riemann problem within the domain $[0,1]$.
The initial conditions are as follows:
\begin{equation*}
    \textbf{V}(x,0)=
    \begin{cases}
    (1,0,10^{4})^\top,&  x<0.5,\\
    (1,0,10^{-8})^\top,&  x>0.5.
    \end{cases}
\end{equation*}
The outflow boundary conditions are applied.
The solution at final time $t=0.45$ develops a strong leftward rarefaction, a rightward contact discontinuity, and a rightward shock. The latter two waves propagate at velocities exceedingly close to the speed of light (approximately 0.986956 and 0.9963757), rendering the simulation highly challenging. A key difficulty lies in resolving the extremely narrow region between the contact discontinuity and the shock front, which has a theoretical width of about $4.2\times10^{-3}$ at $t = 0.45$.
Fig. \ref{Fig:ExRCS} compares the numerical solutions obtained using our proposed GQL-based flux limiter against the iterative limiter of Wu and Tang \cite{WuTang2015}. While both methods preserve positivity, the proposed scheme (circles) exhibits visibly sharper resolution of the contact discontinuity and the shock front compared to the {\tt Wu-Tang} limiter (squares). This improvement is attributed to the exact determination of the locally optimal limiting parameter, avoiding the over-dissipation.
{
To further quantify the performance of the proposed flux limiter and analyze its activation behavior, Fig. \ref{Fig:RCSLimiterMap} shows the spatiotemporal distribution of the flux limiter coefficient $\theta(x,t)$, which is consistent with the wave structure.
}
\end{example}

\begin{figure}[!htb]
	\centering
	\begin{subfigure}[t]{.48\textwidth}
		\centering
		\includegraphics[width=1\textwidth]{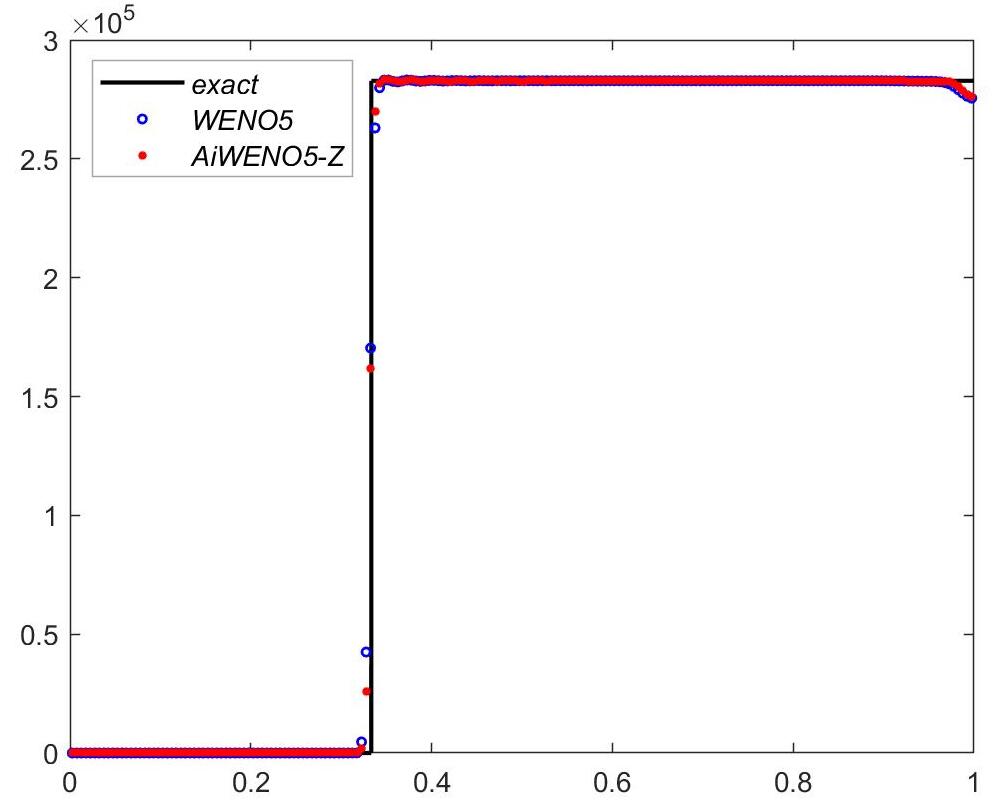}
		\caption{$\rho$}
	\end{subfigure}
    \begin{subfigure}[t]{.48\textwidth}
		\centering
		\includegraphics[width=1\textwidth]{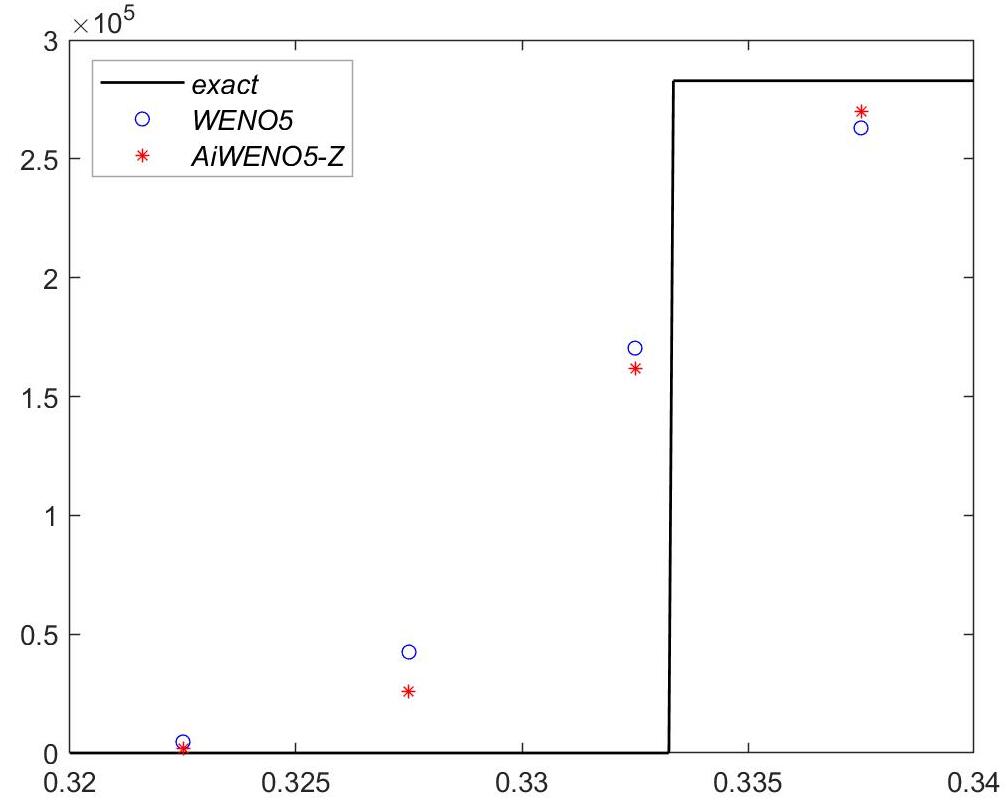}
		\caption{Close-up of $\rho$}
	\end{subfigure}
    \begin{subfigure}[t]{.48\textwidth}
		\centering
		\includegraphics[width=1\textwidth]{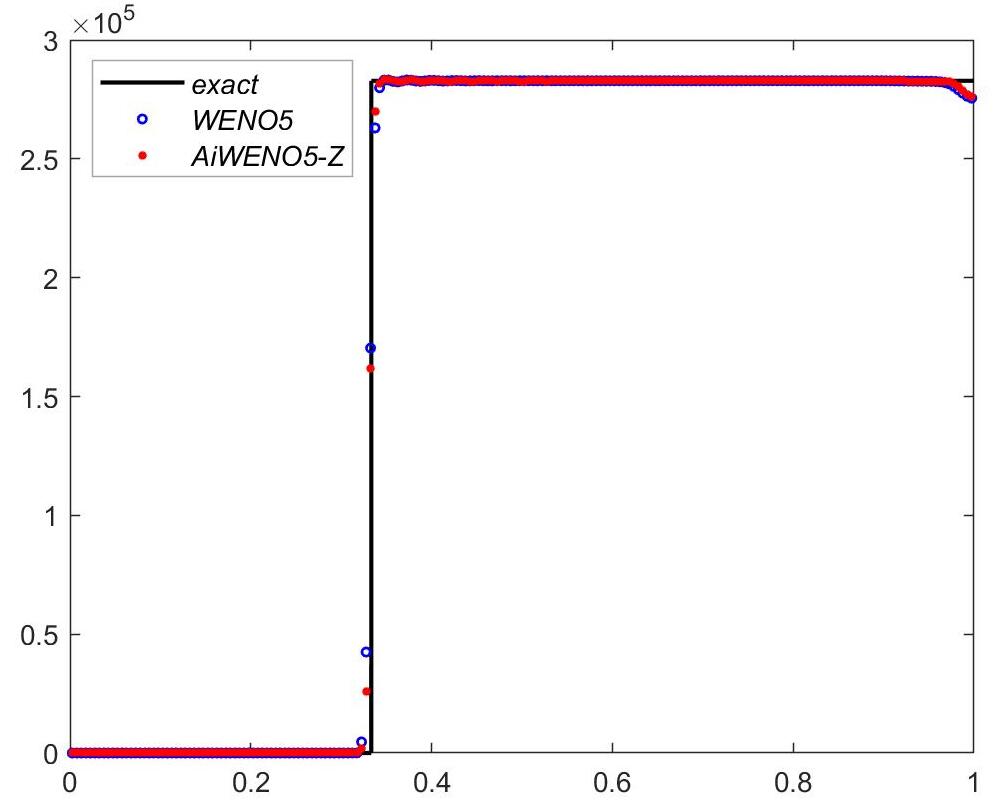}
		\caption{$v_1$}
	\end{subfigure}
    \begin{subfigure}[t]{.48\textwidth}
		\centering
		\includegraphics[width=1\textwidth]{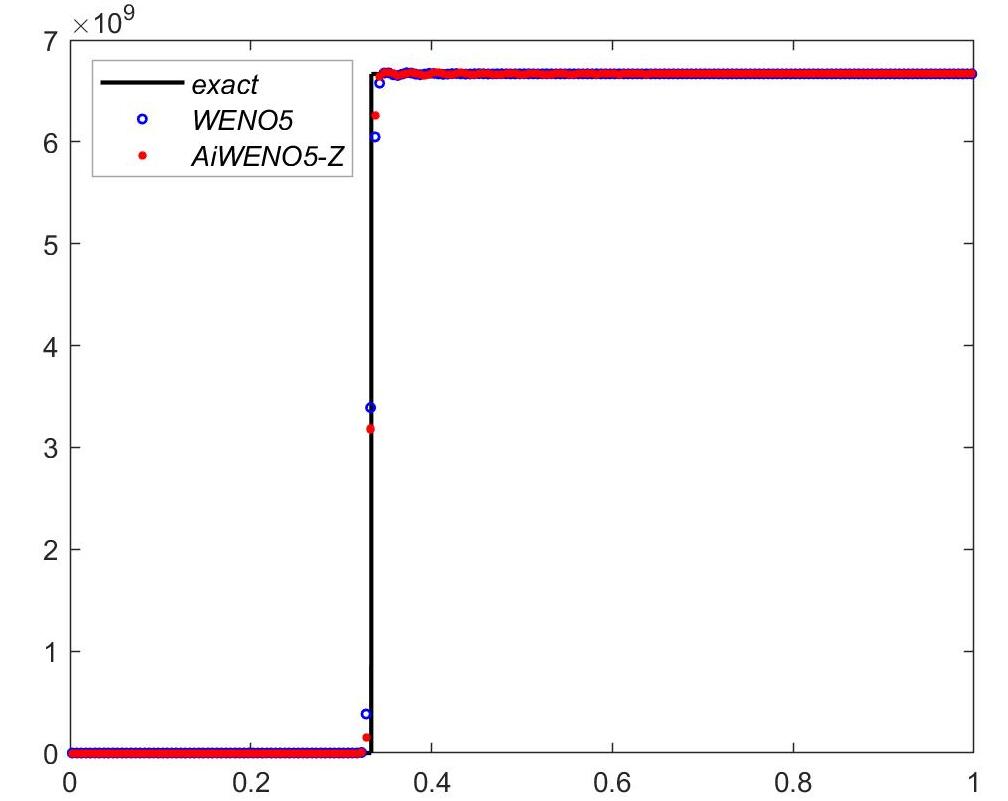}
		\caption{$p$}
	\end{subfigure}
	\caption{Example \ref{Ex:SH}: Numerical results for $\rho,\ v_1,\ p$ obtained using WENO5 or AiWENO5-Z with the proposed PCP flux limiter at $t = 2$. }\label{Fig:ExSH}
\end{figure}

\begin{example}{(Shock heating problem)}\label{Ex:SH}
We consider a shock heating test in this example. The initial data are prescribed as
\begin{equation*}
    \textbf{V}(x,0)=
    (1,1-10^{-10},{10^{-4}}/{3})^\top,\quad  0<x<1.
\end{equation*}
At $x = 1$, a reflective boundary condition is imposed, while the left boundary $x=0$ allows free inflow.
Physically, this setup describes a cold, ultra-relativistic gas stream moving to the right toward a fixed wall. Upon impact, the kinetic energy is converted into thermal energy, resulting in strong compression and the formation of a left-propagating shock. The shock speed can be derived analytically as $v_s=\frac{(\Gamma-1)W_0|v_0|}{W_0+1}$, where $\Gamma=4/3$, $v_0=1-10^{-10}$, and $W_0=(1-v_0^2)^{-1/2}\approx70710.675$. Behind the shock, the gas is at rest and attains a specific internal energy of $W_0-1$.

{To demonstrate the generality of the proposed limiter, we also incorporate our GQL-based flux limiter into the fifth-order Ai-WENO-Z scheme (AiWENO5-Z) \cite{wang2022affine}.} Fig. \ref{Fig:ExSH} presents the profiles of rest-mass density $\rho$, velocity $v_1$, and pressure $p$ at $t=2$ computed by {either WENO5 or AiWENO5-Z} with our flux limiter on a mesh of 200 uniform cells. The exact solution is provided for reference. It is noteworthy that without the flux limiter, {both the finite difference WENO5 and AiWENO5-Z scheme would produce non-physical solutions and break down during the simulation. The AiWENO5-Z solution exhibits higher resolution near the shock as expected. We also provide pointwise error plots in Fig. \ref{Fig:ExSH_error}, which implies that the error is concentrated in a narrow region around the shock, consistent with WENO's accuracy behavior.} {The tiny oscillations are visible immediately behind the reflected shock, which are also observed in \cite{WuTang2015}. Critically, this behavior is not unique to our flux limiter, but rather a general property of high-order shock-capturing methods.} Although a slight wall-heating phenomenon persists near the reflective boundary $x=1$, the overall solution exhibits excellent stability and high resolution.
{
Fig. \ref{Fig:SHLimiterMap} presents the spatiotemporal distribution of $\theta(x,t)$ for the shock heating problem. {In this problem, the flux limiter $\theta$ takes values of 0 or 1 in nearly the entire domain, and intermediate values only exist in an extremely narrow band right at the shock front.} The activation region forms a straight line that exactly coincides with the trajectory of the left-propagating reflected shock wave, demonstrating the limiter's precise shock-tracking capability.
}
\end{example}

\begin{figure}[!htb]
	\centering
	\begin{subfigure}[t]{.48\textwidth}
		\centering
		\includegraphics[width=1\textwidth]{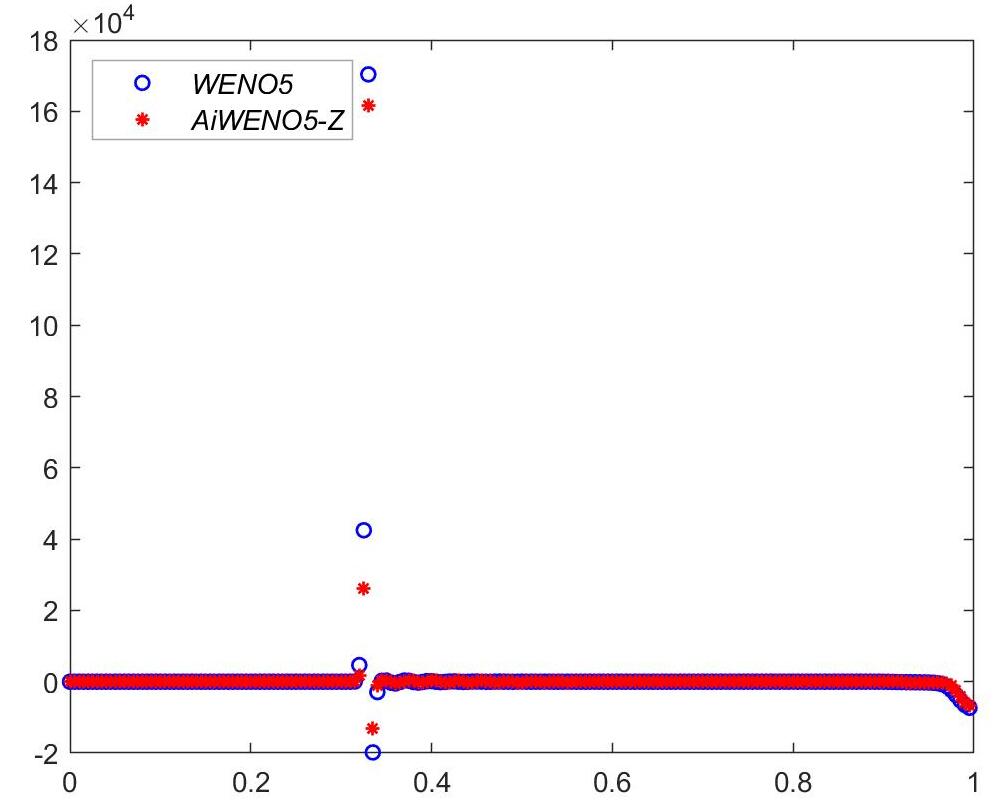}
		\caption{Error plot for $\rho$}
	\end{subfigure}
    \begin{subfigure}[t]{.48\textwidth}
		\centering
		\includegraphics[width=1\textwidth]{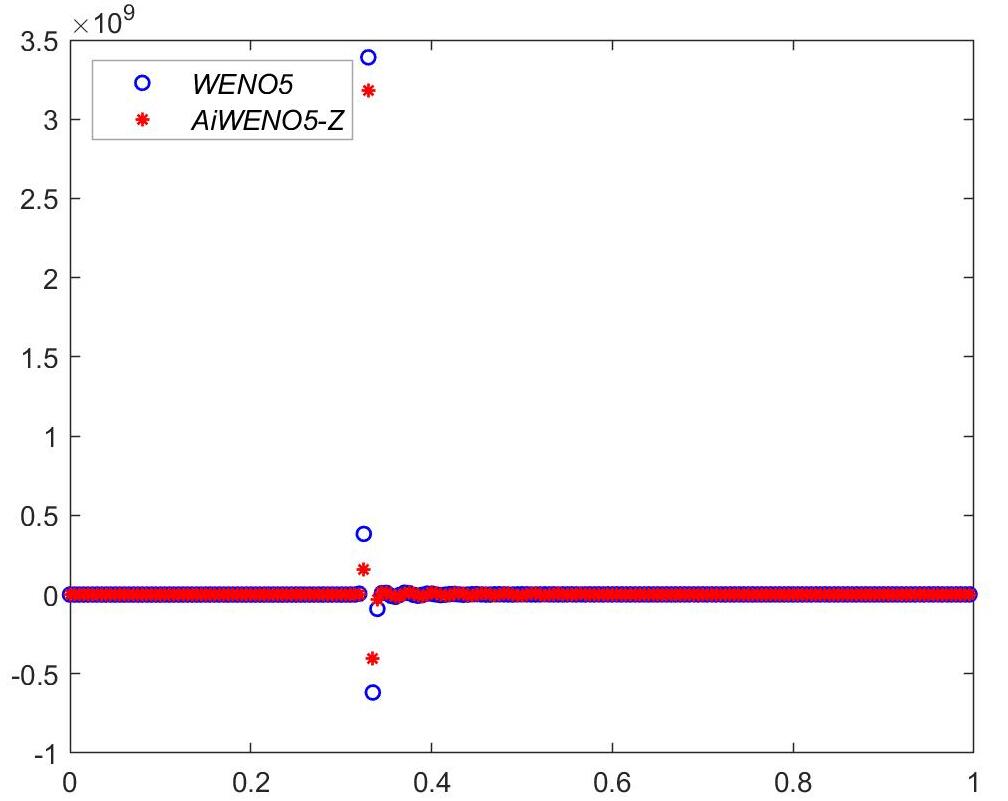}
		\caption{Error plot for $p$}
	\end{subfigure}
	\caption{Example \ref{Ex:SH}: Plots of error between numerical and exact solution $u-u_{\tt ex}$ at $t = 2$ obtained by WENO5 or AiWENO5-Z with the proposed PCP flux limiter. }\label{Fig:ExSH_error}
\end{figure}

\begin{example}{(2D Riemann problems).}\label{Ex5.2.5}
In this example, we simulate two challenging 2D Riemann problems to verify the robustness and bound-preserving property of our flux limiter. The computational domain is set to $[0,1]^2$ with outflow boundary conditions imposed on all boundaries. The maximum fluid velocity in both problems is close to the speed of light, presenting significant challenges for numerical simulation. We employ a uniform mesh with $400\times400$ cells to assess the performance of WENO5, using the PCP limiter with the estimator \eqref{16mat2D} and its relaxed variant \eqref{relax2D}.  

The initial conditions for the two 2D Riemann problems are given by 
\begin{equation*}
    \textbf{V}(x,y,0) = 
    \begin{cases}
        (0.1, 0, 0, 20)^\top, & x>0.5,\ y>0.5,\\
        (0.00414329639576, 0.9946418833556542, 0, 0.05)^\top, & x<0.5,\ y>0.5,\\
        (0.01, 0, 0, 0.05)^\top, & x<0.5,\ y<0.5,\\
        (0.00414329639576, 0, 0.9946418833556542, 0.05)^\top, & \text{otherwise},
    \end{cases}
\end{equation*} 
and
\begin{equation*}
  \textbf{V}(x,y,0) = 
    \begin{cases}
    (0.1, 0, 0, 0.01)^\top & x>0.5,\ y>0.5,\\
    (0.1, 0.99, 0, 1)^\top & x<0.5,\ y>0.5,\\
    (0.5, 0, 0, 1)^\top & x<0.5,\ y<0.5,\\
    (0.1, 0, 0.99, 1)^\top &\text{otherwise}.
    \end{cases}
\end{equation*}
Fig. \ref{Fig:2D_RP4} and Fig. \ref{Fig:2D_RP5} display the logarithmic rest-mass density $\ln\rho$ at $t = 0.4$ for the first and second Riemann problems, respectively. 
{The proposed scheme} captures the intricate structures, including the sharp jet-like spike, the curved shocks, and the finely resolved mushroom cloud with high resolution. In both simulations, the solution remains strictly within the physical bounds, demonstrating that the present method maintains high resolution while preserving stability and physical constraints. To see the advantages in the resolution of complex wave structures, we further show the results of the rest-mass density, obtained by WENO5 with both proposed PCP flux limiters and the {\tt Wu-Tang} one, over several 1D lines (see Fig. \ref{Fig:2D_RP4_1Dline}), or plot the global minimum of $\theta_{i,j}$ in each step over time (see Fig. \ref{Fig:2D_RP_FLalpha}). These plots validate the high resolution of our proposed PCP flux limiter. {The mild oscillations visible near the shock in Fig. \ref{Fig:2D_RP4_1Dline_a} are inherent to the WENO5 reconstruction on this coarse grid and remain bounded; crucially, they do not violate any physical admissibility constraint.}

{
To quantify the difference between the exact and relaxed implementations of the proposed flux limiter, Fig.~\ref{Fig:2D_RP45_cmp} shows the natural logarithm of the absolute pointwise density difference, i.e., the logarithm  $\log(|\rho_{\text{\tt ex}}-\rho_{\text{\tt re}}|)$, where $\rho_{\text{\tt ex}}$ and $\rho_{\text{\tt re}}$ denote the results obtained with the exact limiter \eqref{16mat2D} and the relaxed limiter \eqref{relax2D}, respectively. The comparison demonstrates that the differences are extremely small:}
{
\begin{itemize}
    \item For the first 2D Riemann problem, the maximum pointwise difference is approximately $3.3290 \times 10^{-5}$, and the average difference is approximately $2.5839 \times 10^{-8}$.
    \item For the second 2D Riemann problem, the maximum pointwise difference is approximately $4.4900 \times 10^{-5}$, and the average difference is approximately $3.0354 \times 10^{-8}$.
\end{itemize}
These differences are orders of magnitude smaller than the solution amplitude, confirming that the relaxed flux limiter does not introduce excessive or unnecessary numerical dissipation, while providing dramatic computational savings.
}
\end{example}

\begin{figure}[!htb]
	\centering
	\begin{subfigure}[t]{.32\linewidth}
		\centering
		\includegraphics[width=0.95\textwidth]{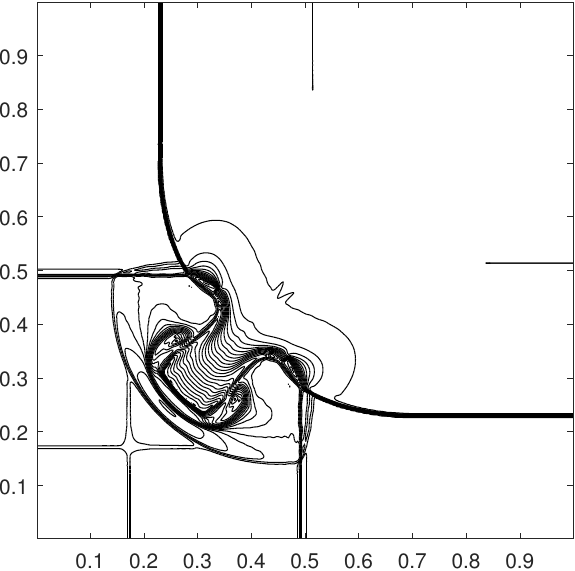}
		\caption{$\log \rho$, flux limiter \eqref{16mat2D}.}
	\end{subfigure}
	\begin{subfigure}[t]{.32\linewidth}
		\centering
		\includegraphics[width=0.95\textwidth]{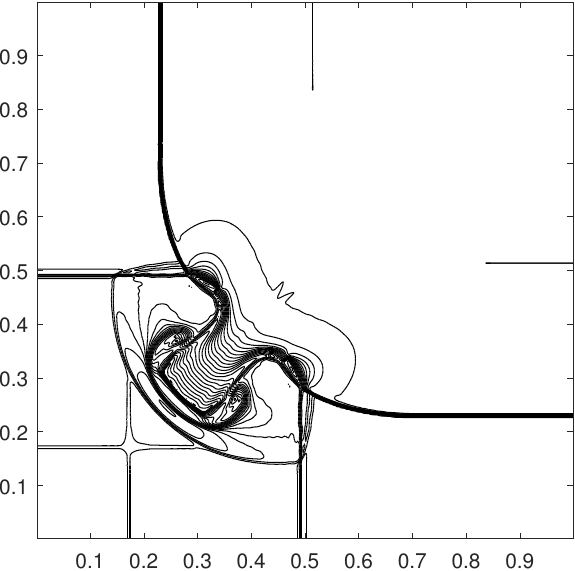}
		\caption{$\log \rho$, flux limiter \eqref{relax2D}.}
	\end{subfigure}
    \begin{subfigure}[t]{.32\linewidth}
		\centering
		\includegraphics[width=0.95\textwidth]{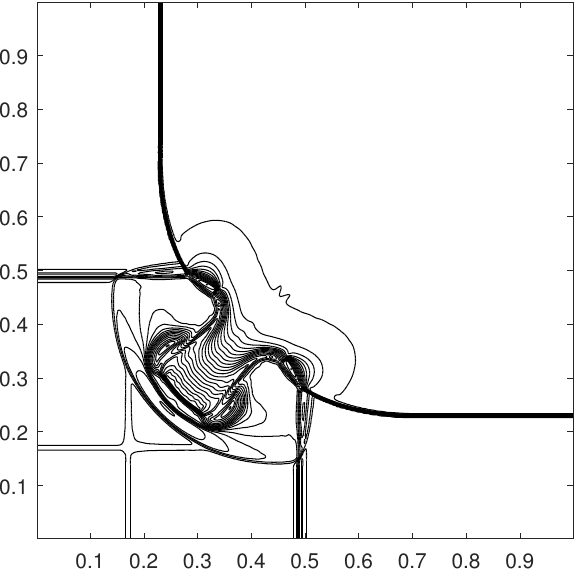}
		\caption{$\log \rho$, {\tt Wu-Tang} flux limiter.}
	\end{subfigure}
	\caption{First 2D Riemann problem in Example \ref{Ex5.2.5}: The contours of $\log\rho$ at $t = 0.4$ obtained by WENO5 {with different flux limiters.} 25 equally spaced contour lines from -8 to -2.3 are displayed.}\label{Fig:2D_RP4}
\end{figure}

\begin{figure}[!htb]
	\centering
	\begin{subfigure}[t]{.48\linewidth}
		\centering
		\includegraphics[width=0.95\textwidth]{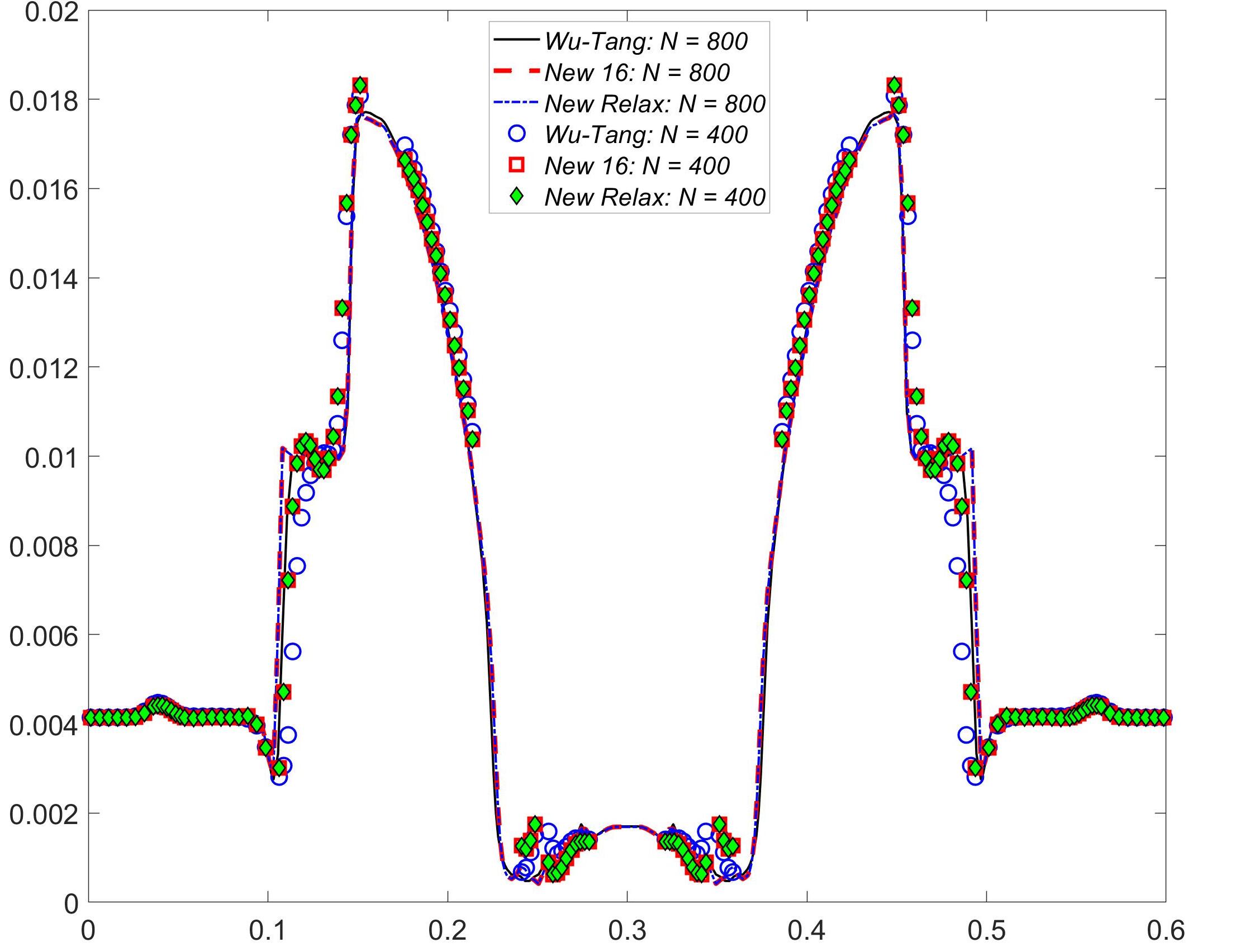}
		\caption{$\rho$ along $x+y=0.6$.}\label{Fig:2D_RP4_1Dline_a}
	\end{subfigure}
	\begin{subfigure}[t]{.48\linewidth}
		\centering
		\includegraphics[width=0.95\textwidth]{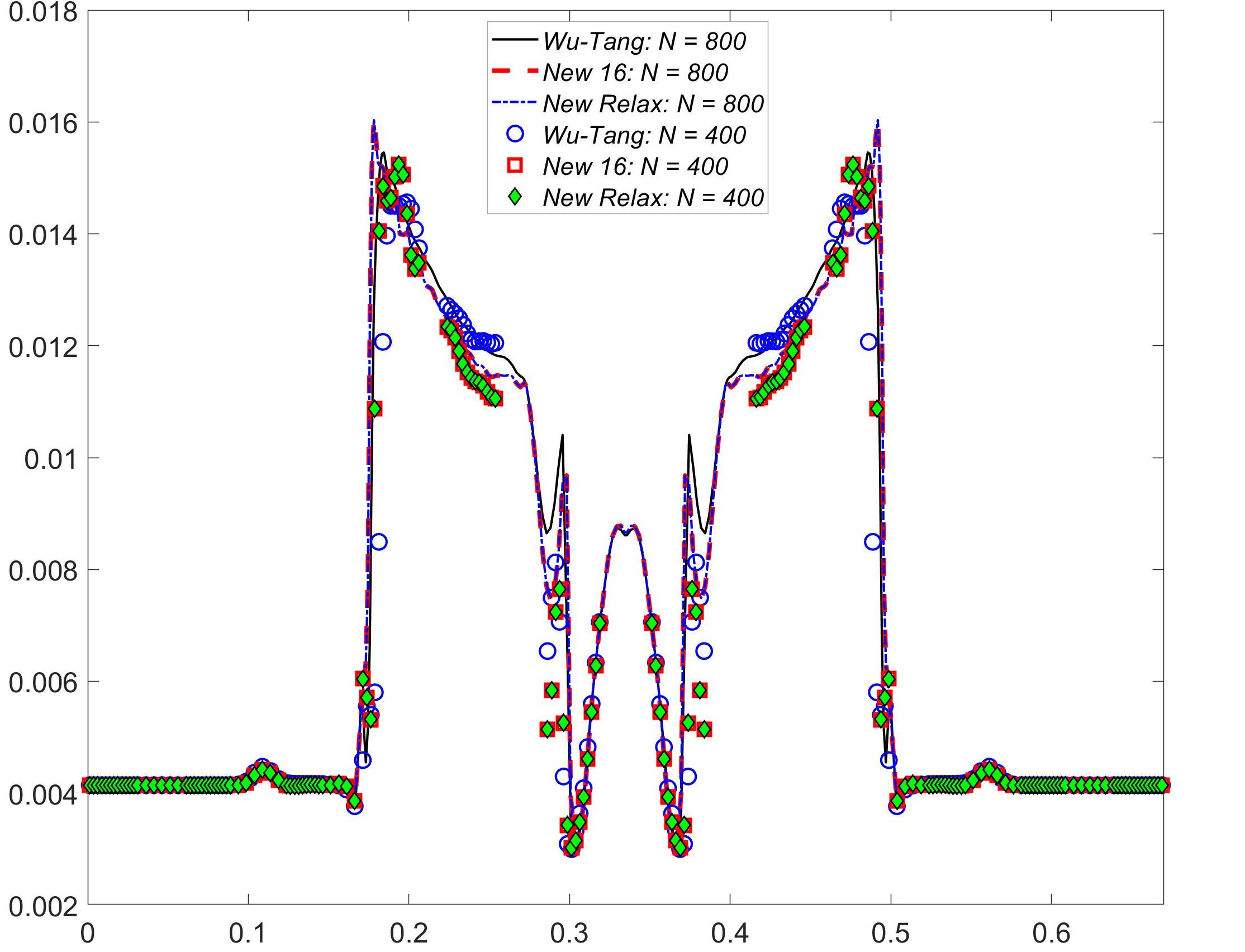}
		\caption{$\rho$ along $x+y=0.67$.}
	\end{subfigure}
    \begin{subfigure}[t]{.48\linewidth}
		\centering
		\includegraphics[width=0.95\textwidth]{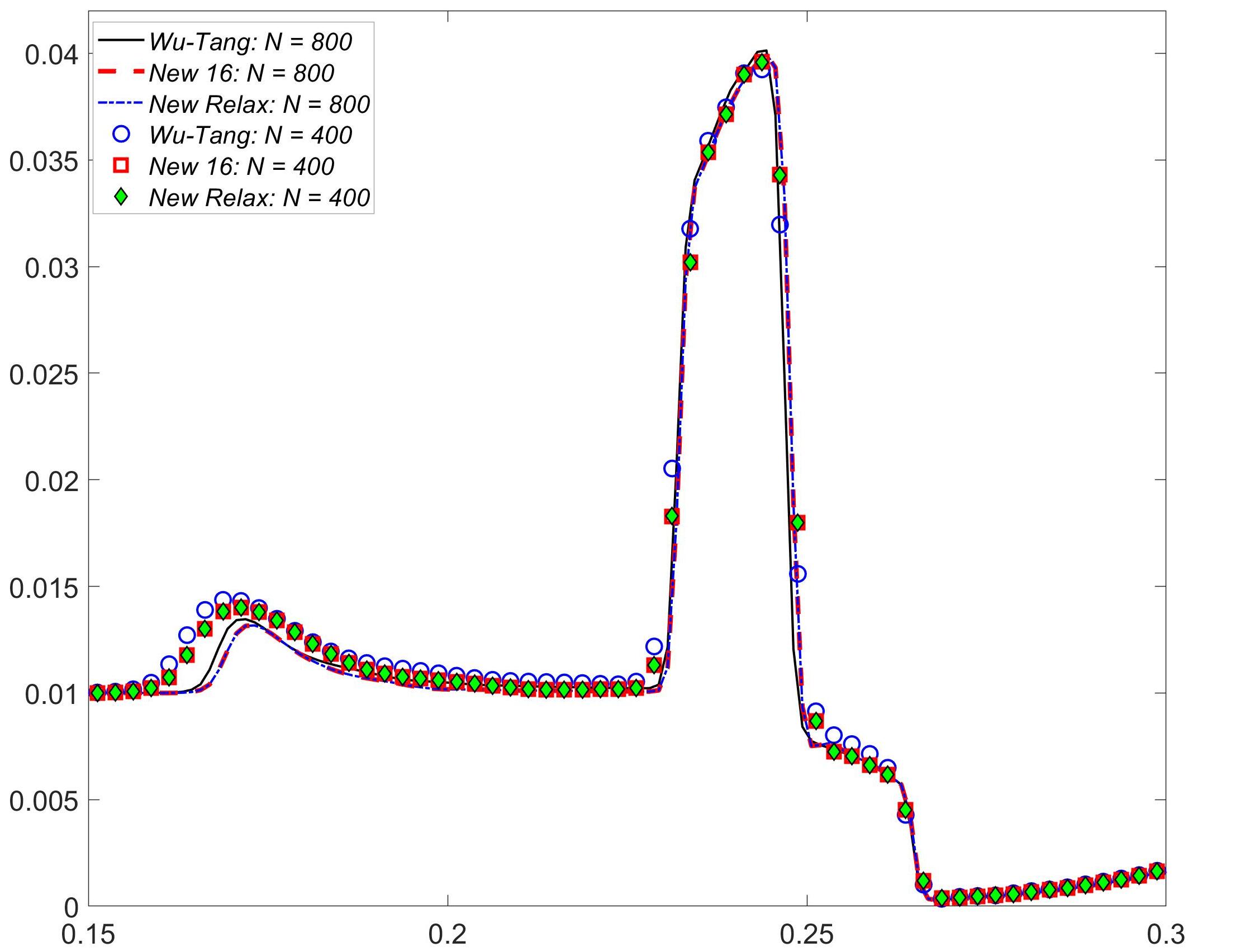}
		\caption{$\rho$ along $x=y$.}
	\end{subfigure}
    \begin{subfigure}[t]{.48\linewidth}
		\centering
		\includegraphics[width=0.95\textwidth]{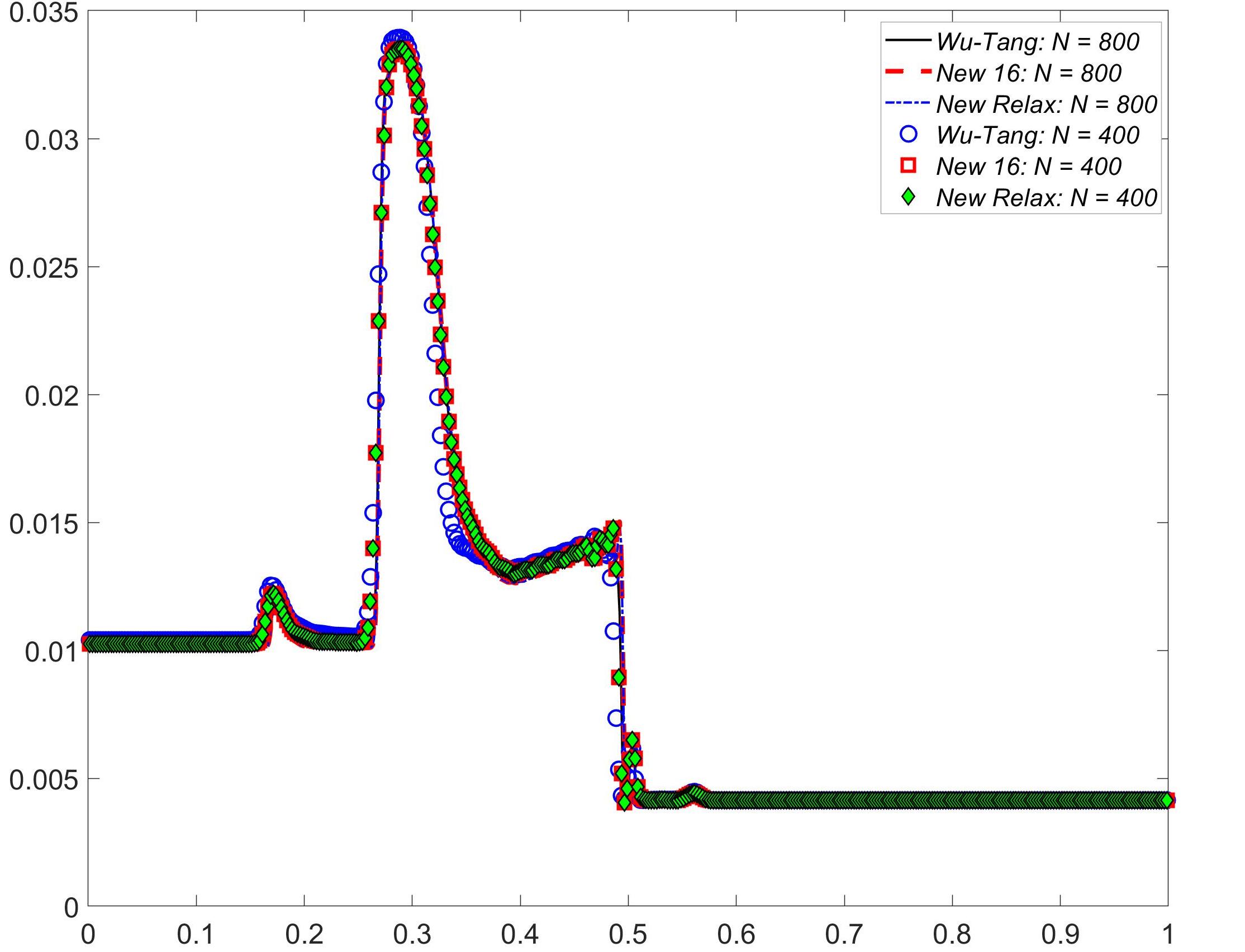}
		\caption{$\rho$ along $y=0.2$.}
	\end{subfigure}
	\caption{Same as Fig. \ref{Fig:2D_RP4} except for $\rho$ along different lines.}\label{Fig:2D_RP4_1Dline}
\end{figure}

\begin{figure}[!htb]
	\centering
	\begin{subfigure}[t]{.32\linewidth}
		\centering
		\includegraphics[width=0.95\textwidth]{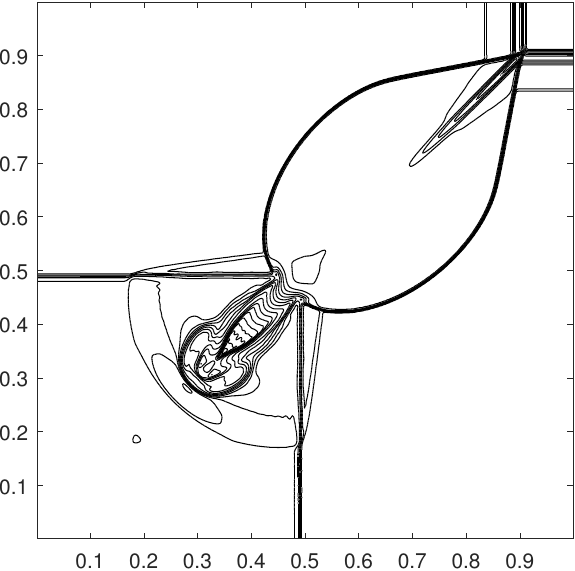}
		\caption{$\log \rho$, flux limiter \eqref{16mat2D}.}
	\end{subfigure}
	\begin{subfigure}[t]{.32\linewidth}
		\centering
		\includegraphics[width=0.95\textwidth]{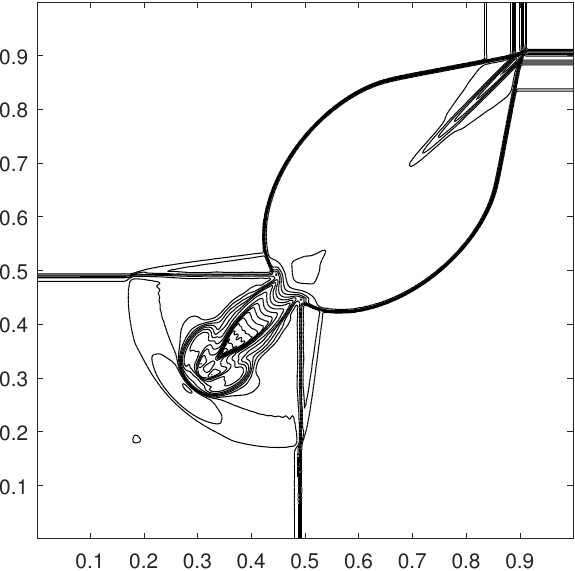}
		\caption{$\log \rho$, flux limiter \eqref{relax2D}.}
	\end{subfigure}
    \begin{subfigure}[t]{.32\linewidth}
		\centering
		\includegraphics[width=0.95\textwidth]{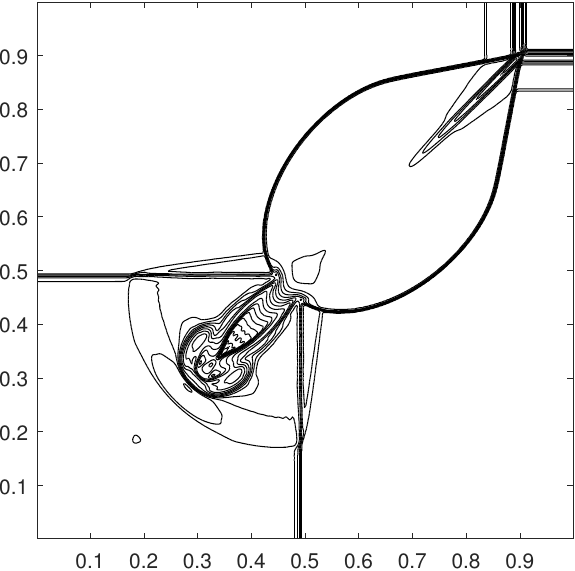}
		\caption{$\log \rho$, {\tt Wu-Tang} flux limiter.}
	\end{subfigure}
	\caption{Second 2D Riemann problem in Example \ref{Ex5.2.5}: The contours of $\log\rho$ at $t = 0.4$ obtained by WENO5 {with different flux limiters.} 25 equally spaced contour lines from -6 to 1.9 are displayed.}\label{Fig:2D_RP5}
\end{figure}

\begin{figure}[!htb]
	\centering
	\begin{subfigure}[t]{.48\linewidth}
		\centering
		\includegraphics[width=1\textwidth]{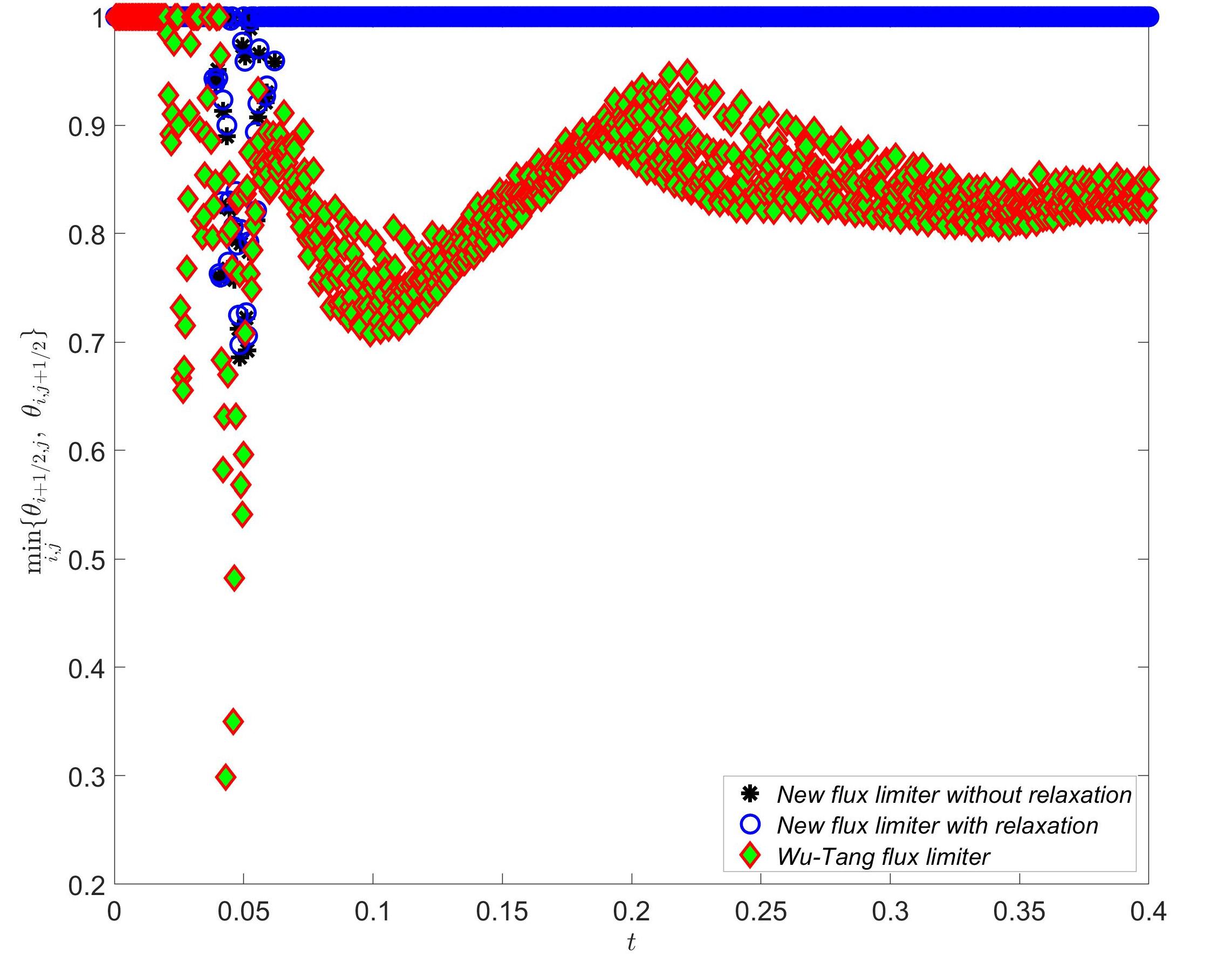}
		\caption{First 2D Riemann problem.}
	\end{subfigure}
	\begin{subfigure}[t]{.48\linewidth}
		\centering
		\includegraphics[width=1\textwidth]{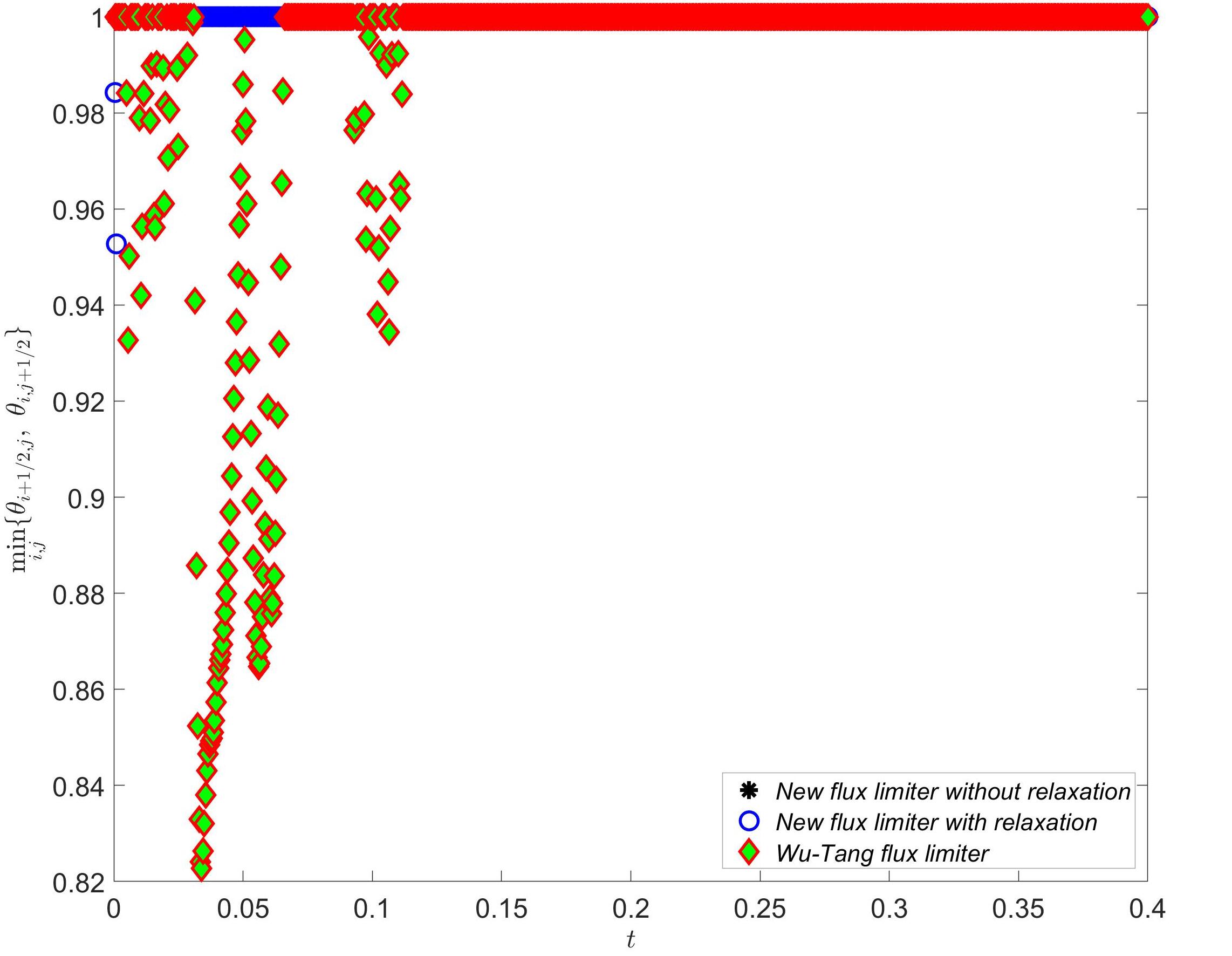}
		\caption{Second 2D Riemann problem.}
	\end{subfigure}
	\caption{Example \ref{Ex5.2.5}: Comparison of $\min\limits_{i,j}\left\{\theta_{i+1/2,j},\ \theta_{i,j+1/2}\right\}$ between different flux limiters over time. }\label{Fig:2D_RP_FLalpha}
\end{figure}

\begin{figure}[!htb]
	\centering
	\begin{subfigure}[t]{.48\linewidth}
		\centering
		\includegraphics[width=0.95\textwidth]{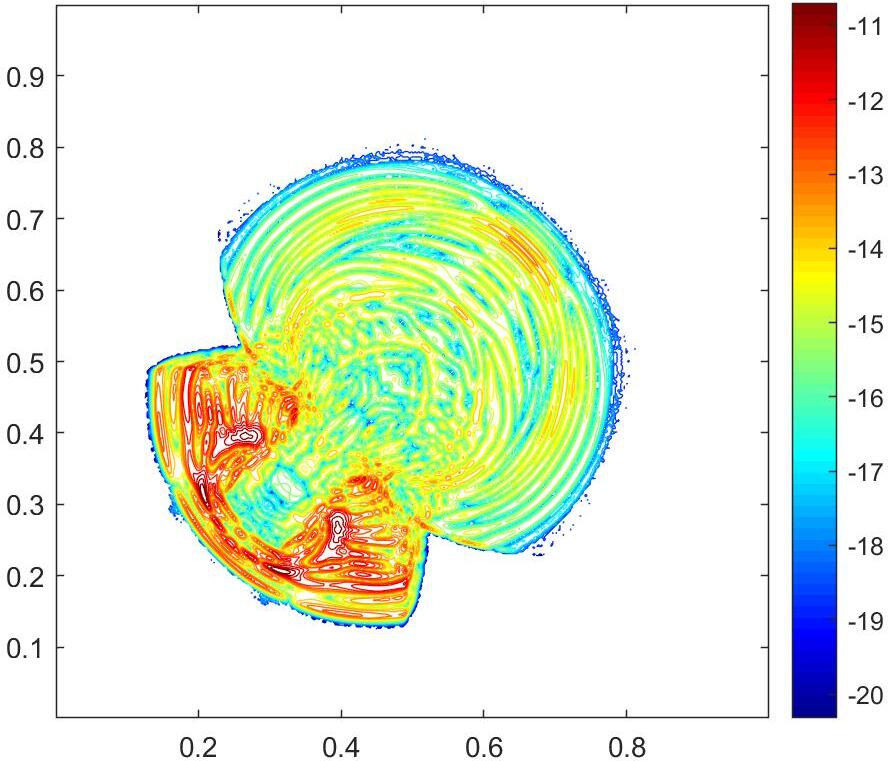}
		\caption{{Error logarithm  for first 2D Riemann problem.}}
	\end{subfigure}
	\begin{subfigure}[t]{.48\linewidth}
		\centering
		\includegraphics[width=0.95\textwidth]{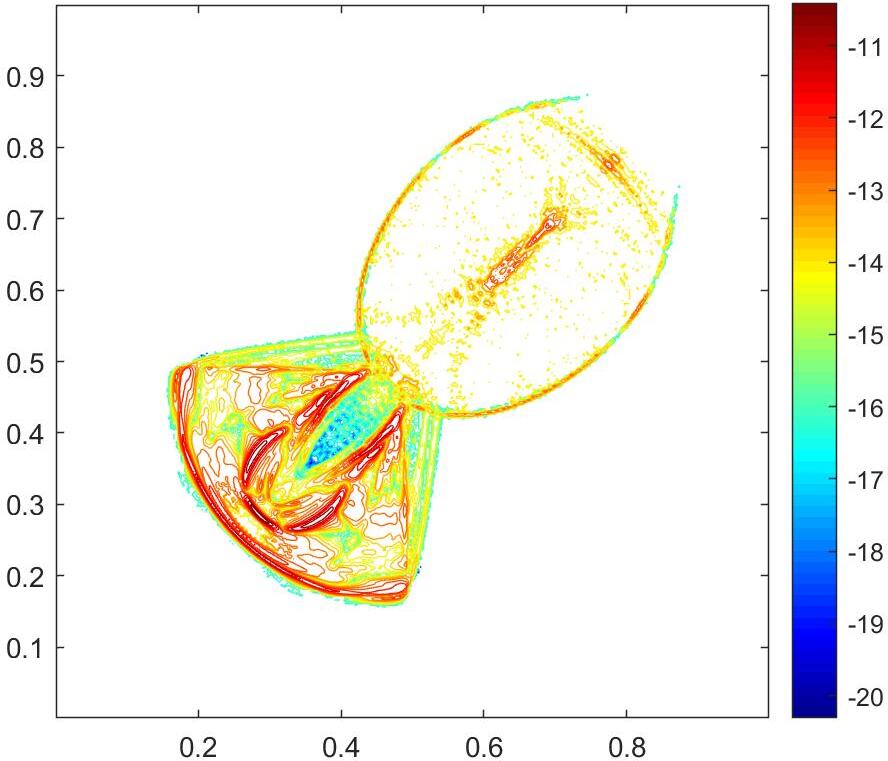}
		\caption{{Error logarithm  for second 2D Riemann problem.}}
	\end{subfigure}
	\caption{{Example \ref{Ex5.2.5}: Contours of the logarithm $\log(|\rho_{\text{\tt ex}}-\rho_{\text{\tt re}}|)$ at $t=0.4$ obtained by WENO5, where $\rho_{\text{\tt ex}}$ and $\rho_{\text{\tt re}}$ denote the density computed with the exact and relaxed flux limiters, respectively. 25 equally spaced contour levels are displayed.}}\label{Fig:2D_RP45_cmp}
\end{figure}

\begin{example}{(Shock-Vortex interaction problem)} \label{EX:SV}
This example examines the interaction between a shock wave and a vortex, a problem previously studied in \cite{duan2019high,chen2022physical}. 
The initial condition is given by
\begin{gather*}
    \mathbf{V}(x,y,0) = 
    \begin{cases}
        \mathbf{V}_{\rm L}, & x < -6, \\[5pt]
        \mathbf{V}_{\rm vortex}, & x \ge -6,
    \end{cases}
    \qquad{\rm with}\\
    \mathbf{V}_{\rm L}=\Big(4.891497310766981,-0.388882958251919,0,11.894863258311670\Big)^\top,\\
    \mathbf{V}_{\rm vortex}=\left(
        \left(\hat{\theta}(x,y)\right)^\frac{1}{\Gamma-1},
        \frac{w_1(x,y)-v_c}{1-v_c \, w_1(x,y)},
        \frac{\sqrt{1-v_c^2}\ w_2(x,y)}{1-v_c \, w_1(x,y)},
        \left(\hat{\theta}(x,y)\right)^\frac{\Gamma}{\Gamma-1}
        \right)^\top,
\end{gather*}
where the functions $\hat{\theta}$, $w_1$, and $w_2$ are given by
\begin{equation*}
    \hat{\theta}(x,y) = 1-\frac{{\varepsilon_{\rm vtx}^2}}{28\pi^2} \, \mathrm{e}^{1-\frac{x^2}{1-v_c^2}-y^2}, ~ 
    w_1(x,y) = -y \hat{f}(x,y), 
    ~ 
    w_2(x,y) = \frac{x}{\sqrt{1-v_c^2}}\hat{f}(x,y),
\end{equation*}
with
\begin{equation*}
    \hat{f}(x,y) = \sqrt{\frac{\beta}{1+\beta\left(\frac{x^2}{1-v_c^2}+y^2\right)}},\quad
    \beta = \frac{\frac{{\varepsilon_{\rm vtx}^2}}{10\pi^2}\ \mathrm{e}^{1-\frac{x^2}{1-v_c^2}-y^2}}{1.8-\frac{{\varepsilon_{\rm vtx}^2}}{20\pi^2}\ \mathrm{e}^{1-\frac{x^2}{1-v_c^2}-y^2}}.
\end{equation*}
This setup ensures that
\begin{equation*}
    \lim_{x\rightarrow -6^-} \mathbf{V}(x,y,0) = \mathbf{V}_{\rm L}, 
    \quad
    \lim_{x\rightarrow -6^+} \mathbf{V}(x,y,0) 
    \approx 
    \left(1,-0.9, 0, 1\right)^\top,
\end{equation*}
which represents a stationary shock at $x = -6$. On the right side of the shock, an initially centered vortex at $(0,0)$ moves leftward with a speed of $v_c=0.9$. To demonstrate the robustness of our flux limiter, we investigate a challenging case with a vortex strength set to ${\varepsilon_{\rm vtx}=10.0828}$.
The computational domain is $[-17,3] \times [-5,5]$ with reflective boundary conditions at $y = \pm5$, inflow at $x=3$, and outflow at $x=-17$. We solve this problem up to $t = 19$ using WENO5 with both versions of our flux limiter \eqref{16mat2D} and \eqref{relax2D} on a uniform grid with $\Delta x = \Delta y = 1/40$. 
The contour plots of $\log_{10}(1+|\nabla \rho|)$ are presented in Fig. \ref{Fig:2D_SV}. 
As the vortex interacts with the stationary shock wave, intricate wave structures emerge. {The WENO5 scheme, incorporating either our flux limiter \eqref{16mat2D} or \eqref{relax2D}}, effectively captures these wave patterns. 
{
The difference between the two implementations of our flux limiter, approximately $6.77\times10^{-4}$ in $l^{\infty}$ norm, is much smaller than the solution amplitude, demonstrating the accuracy of the relaxed limiter \eqref{relax2D}.
}

\begin{figure}[!htb]
	\centering
	\begin{subfigure}[t]{.48\textwidth}
		\centering
		\includegraphics[width=1\textwidth]{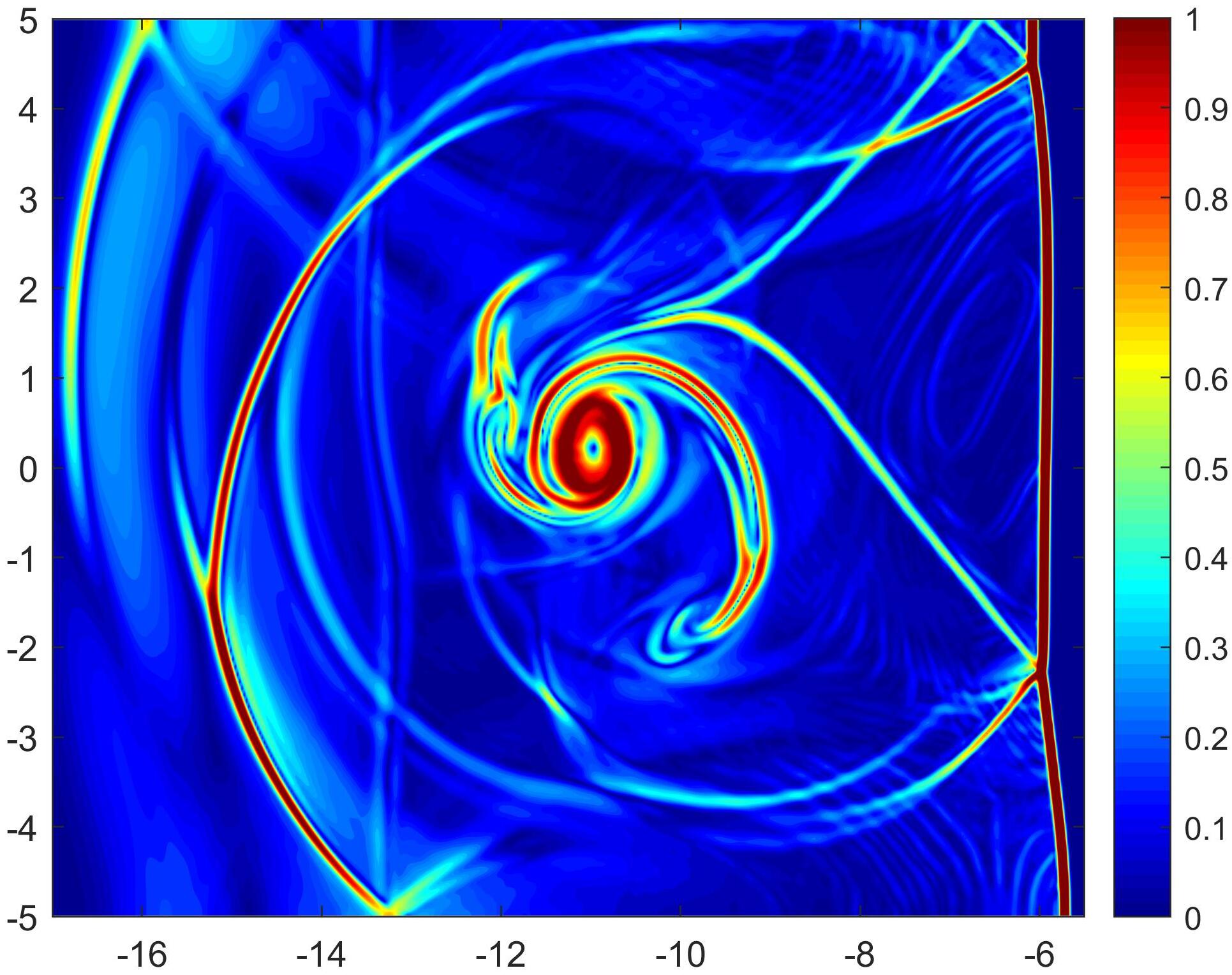}
		\caption{With flux limiter \eqref{16mat2D}.}
	\end{subfigure}
    \begin{subfigure}[t]{.48\textwidth}
		\centering
		\includegraphics[width=1\textwidth]{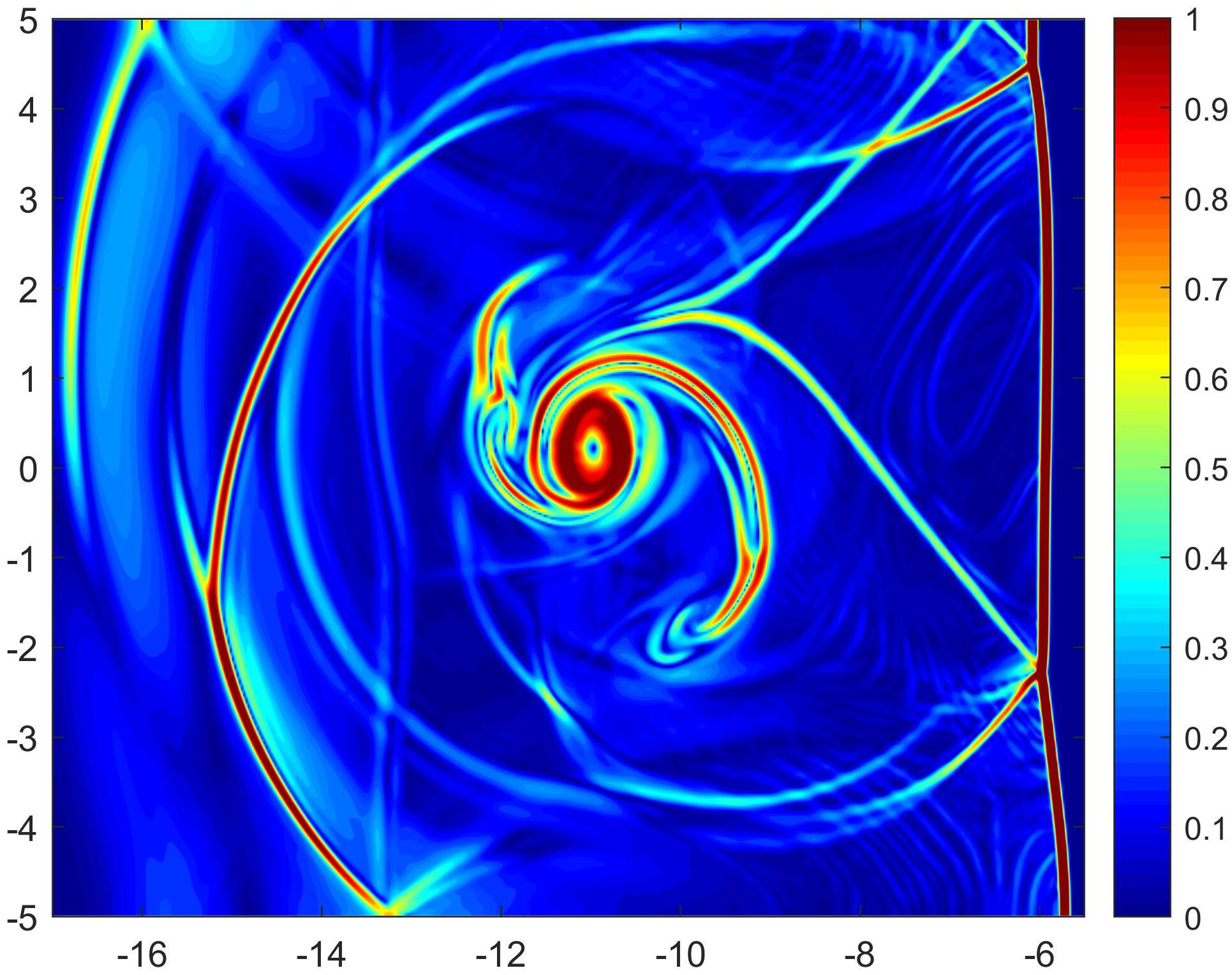}
		\caption{With flux limiter \eqref{relax2D}.}
	\end{subfigure}
	\caption{Example \ref{EX:SV}: Contour plots of ${\rm log}_{10}(1+|\nabla \rho|)$ at $t = 19$, obtained using WENO5.}\label{Fig:2D_SV}
\end{figure}

\end{example}

\begin{example}{(2D Relativistic jets)} \label{EX:jet}
	Relativistic jets are frequently observed in astrophysical phenomena such as active galactic nuclei. Due to the potential development of complex flow features, including strong relativistic shocks, shear layers, and interface instabilities, numerical simulations for relativistic jets are extremely challenging.
	
	We examine three configurations of the pressure-matched, highly supersonic jet model with increasingly extreme relativistic characteristics. The initial states of the ambient flow are defined as:
	\begin{align}
		{\bf V}_1(x,y,0) &= \left(1,0,0,2.3536240721718810\times10^{-5}\right)^\top,\label{j1}\\
		{\bf V}_2(x,y,0) &= \left(1,0,0,2.3966375079777595\times10^{-5}\right)^\top,\label{j2}\\
		{\bf V}_3(x,y,0) &= \left(1,0,0,2.3995344183272123\times10^{-7}\right)^\top.\label{j3}
	\end{align}
	A jet beam with rest-mass density $\rho^b=0.1$ is injected through the bottom boundary inlet $\{y=0,\ |x|\leq0.5\}$. The jet moves in the $y$-direction with velocity $v^b$ while its pressure matches the ambient value. The specific settings for the jet beams corresponding to the ambient flows \eqref{j1}, \eqref{j2}, and \eqref{j3} are:

        \vspace{3mm}
		(i) $v^b = 0.99$, $M_b = 50$ ($\gamma^b\approx7.09$, $M_r\approx354.37$);
  
        \vspace{3mm}
		(ii) $v^b = 0.999$, $M_b = 50$ ($\gamma^b\approx22.37$, $M_r\approx1118.09$);
  
        \vspace{3mm}
		(iii) $v^b = 0.9999$, $M_b = 500$ ($\gamma^b\approx70.71$, $M_r\approx35356.15$).
  
        \vspace{3mm}
        
	Here, the classical beam Mach number is $M_b=v^b/c_s$, and the relativistic Mach number is $M_r=M_b\gamma^b/\gamma^s$, where $\gamma^b$ and $\gamma^s$ are the Lorentz factors for the jet speed and local sound speed, respectively.
	The computational domain $[0,12]\times[0,25]$ is discretized into $240\times500$ uniform cells. Reflective boundary conditions are enforced at $x=0$, a fixed jet beam inflow is imposed on the nozzle $\{y=0,\ 0\leq x\leq0.5\}$, and outflow conditions are applied elsewhere.

    {
	Figs.~\ref{Fig:2D_Jet_lnrho} and~\ref{Fig:2D_Jet_lnpre} present schlieren images of the logarithmic rest-mass density and pressure for the three configurations, computed using both the exact flux limiter \eqref{16mat2D}  (top rows), and the pointwise logarithmic absolute difference between the exact \eqref{16mat2D} and relaxed \eqref{relax2D} limiters (bottom rows). The solutions obtained with the two limiters are visually indistinguishable, with the average pointwise density and pressure differences both below $10^{-5}$
    across all three configurations.} These regimes, particularly the third case with $\gamma^b \approx 70.71$, are notoriously prone to numerical breakdown due to the rapid generation of negative pressure or density during transient evolution. We observe that the standard finite difference WENO5 scheme fails almost immediately without the PCP limiting procedure. In contrast, our GQL-based limiter successfully stabilizes the simulation even at $M_b=500$. All computed solutions exhibit well-resolved flow structures, with the leading Mach shock and the Kelvin-Helmholtz instabilities along the shear layer captured with high resolution. These results confirm that the proposed scheme combines strict physical robustness with excellent shock-capturing capabilities in handling complex, ultra-relativistic jet dynamics.
\end{example}

\begin{figure}[!htb]
	\centering
	\begin{subfigure}[t]{.32\textwidth}
		\centering
		\includegraphics[width=1\textwidth]{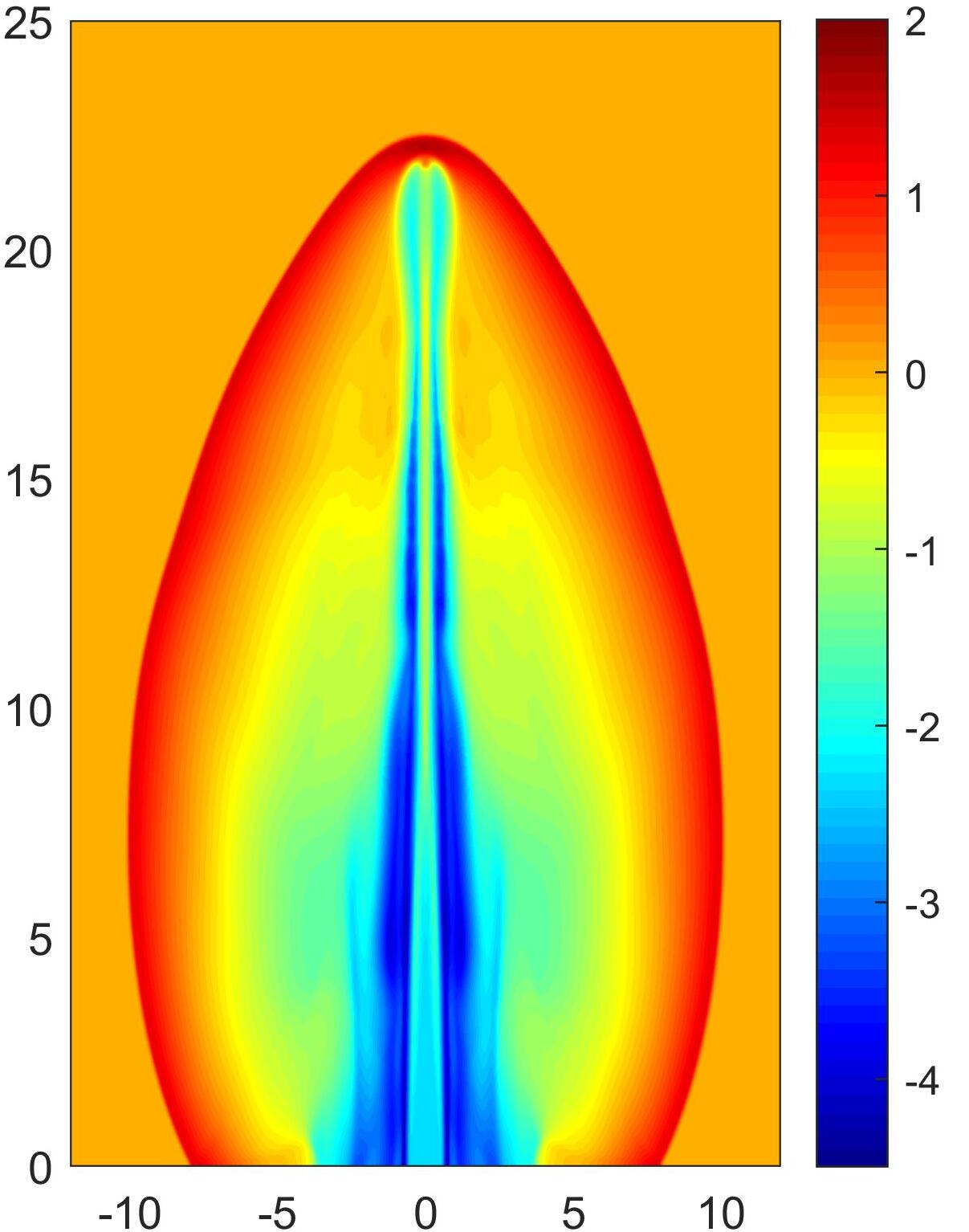}
	\end{subfigure}
    \begin{subfigure}[t]{.32\textwidth}
		\centering
		\includegraphics[width=1\textwidth]{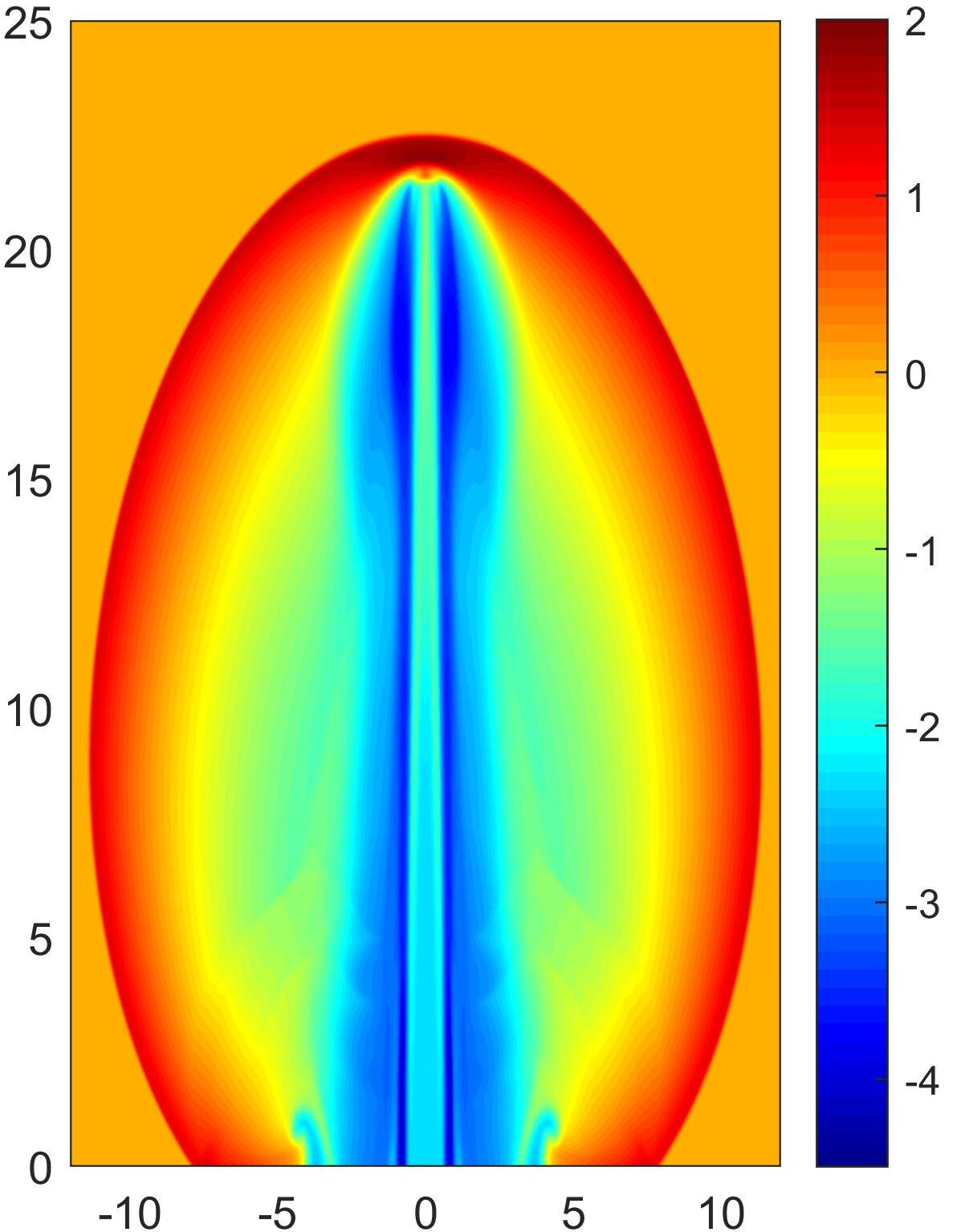}
	\end{subfigure}
 	\begin{subfigure}[t]{.32\textwidth}
		\centering
		\includegraphics[width=1\textwidth]{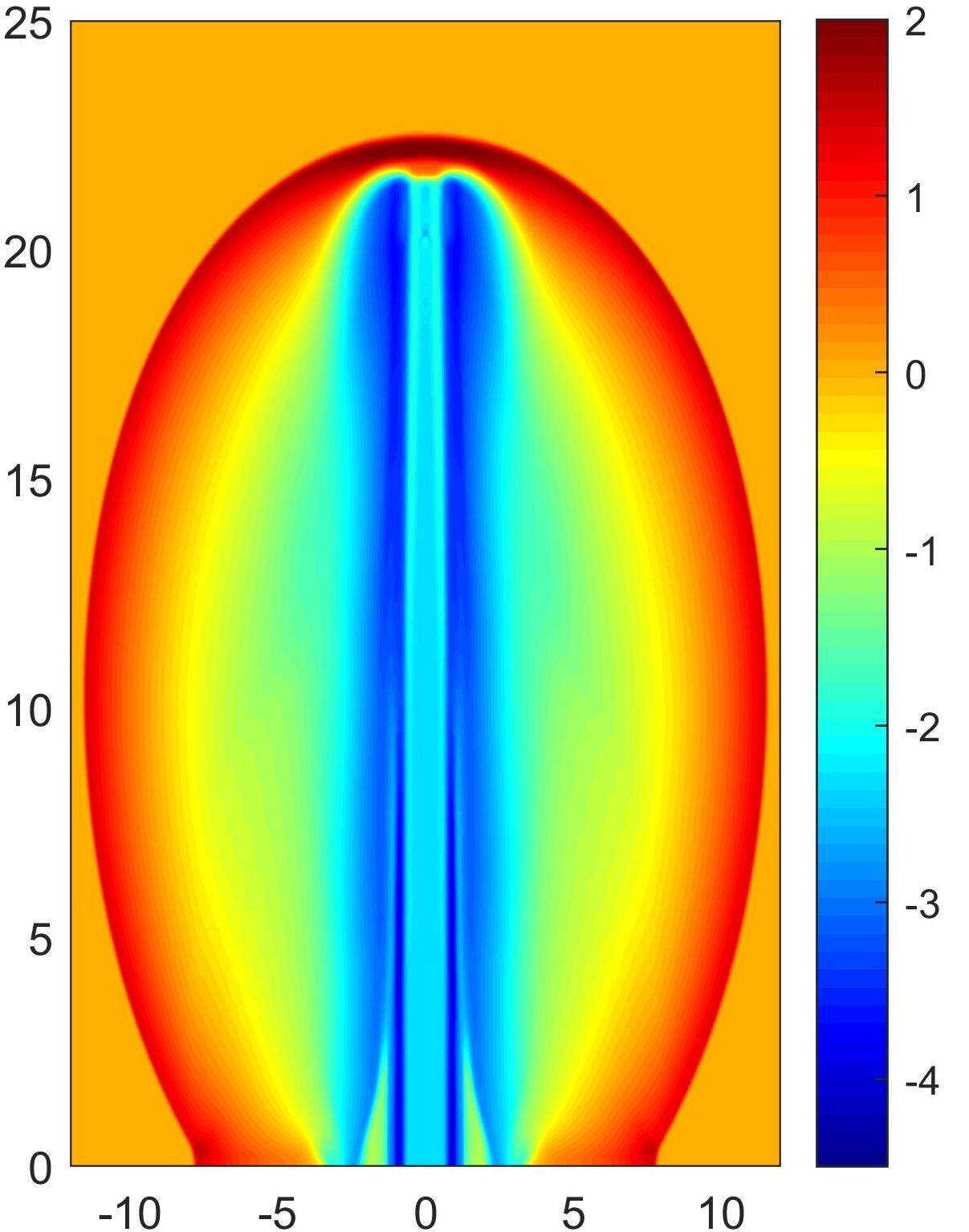}
	\end{subfigure}
    \begin{subfigure}[t]{.32\textwidth}
		\centering
		\includegraphics[width=1\textwidth]{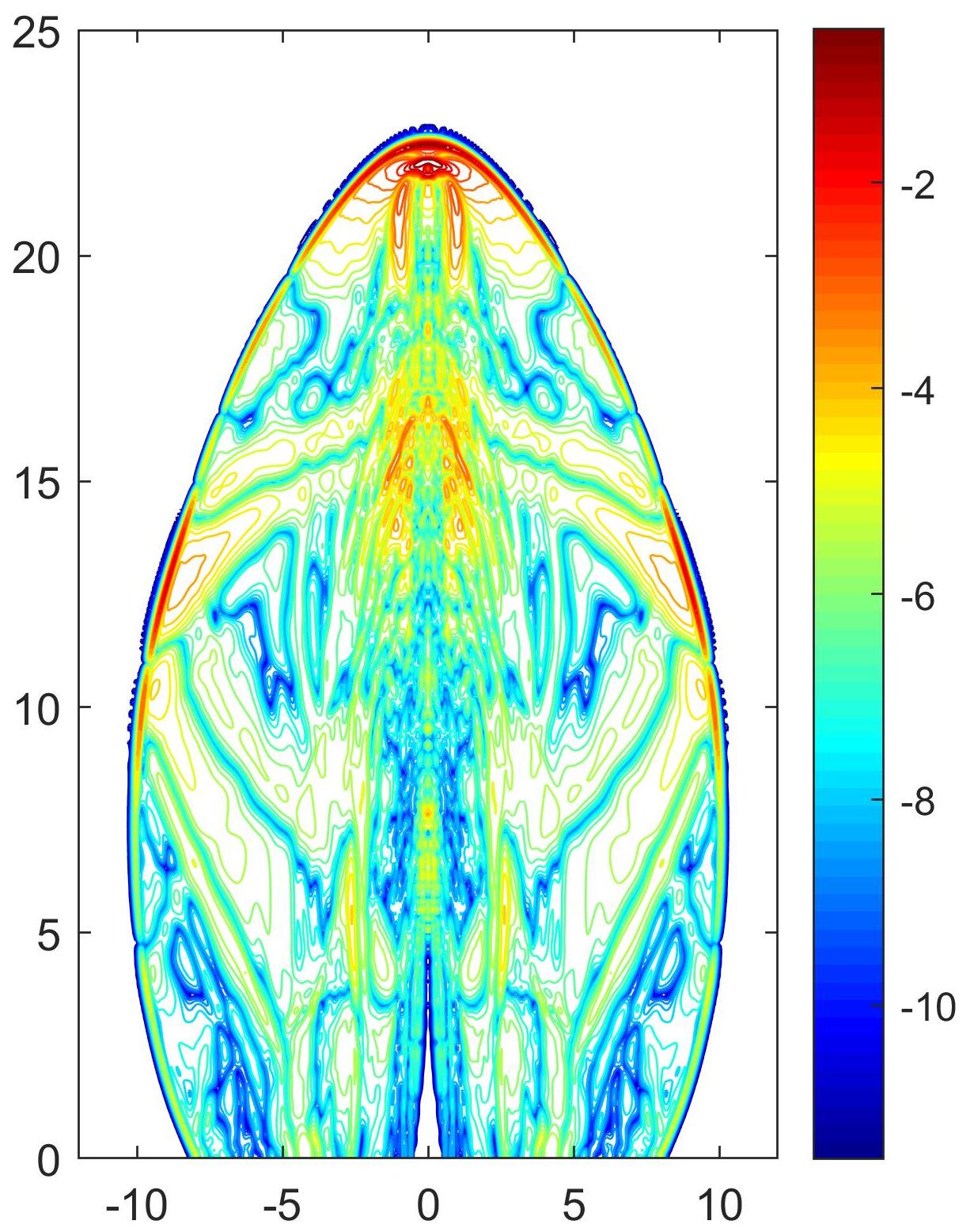}
	\end{subfigure}
	\begin{subfigure}[t]{.32\textwidth}
		\centering
		\includegraphics[width=1\textwidth]{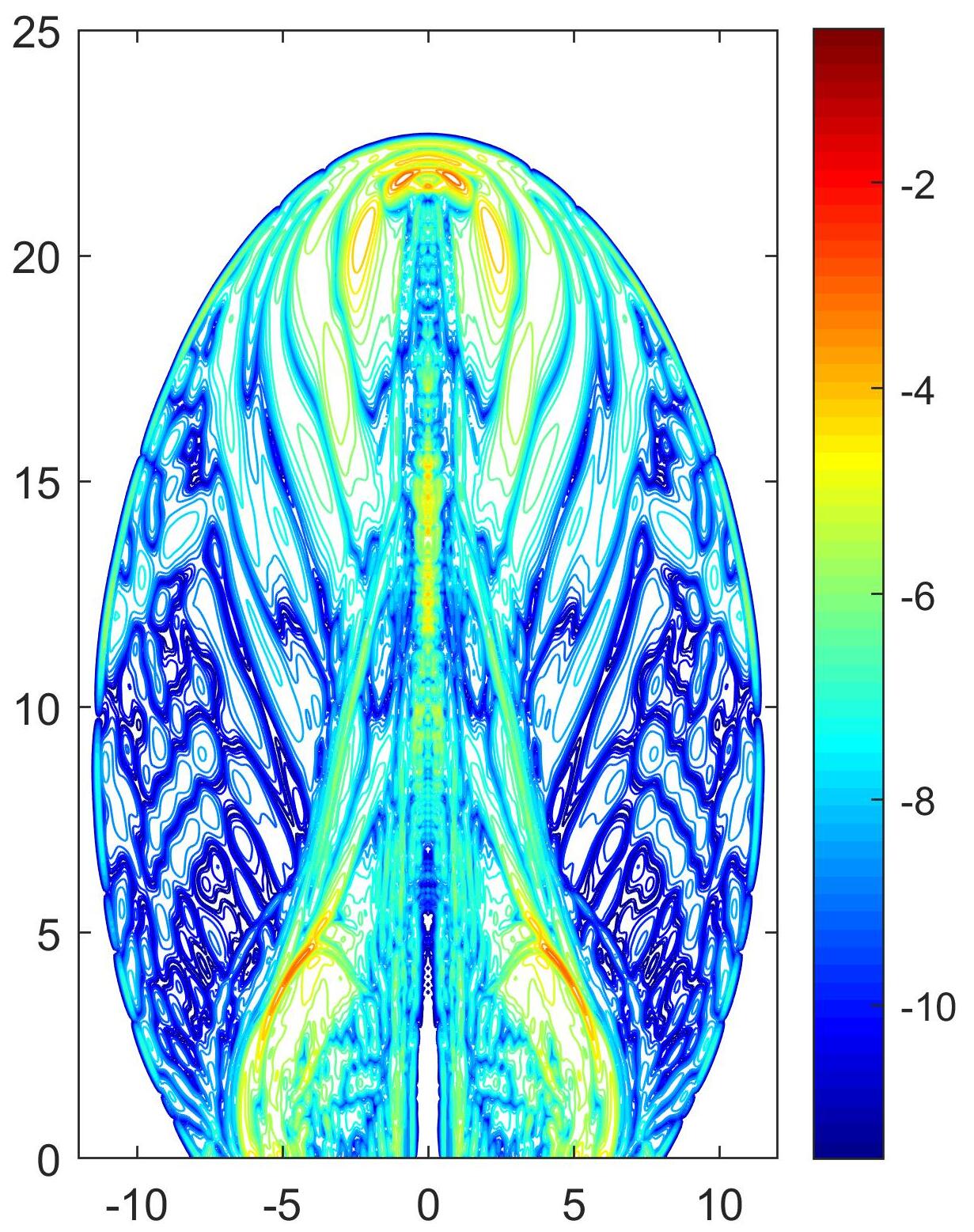}
	\end{subfigure}
    \begin{subfigure}[t]{.32\textwidth}
		\centering
		\includegraphics[width=1\textwidth]{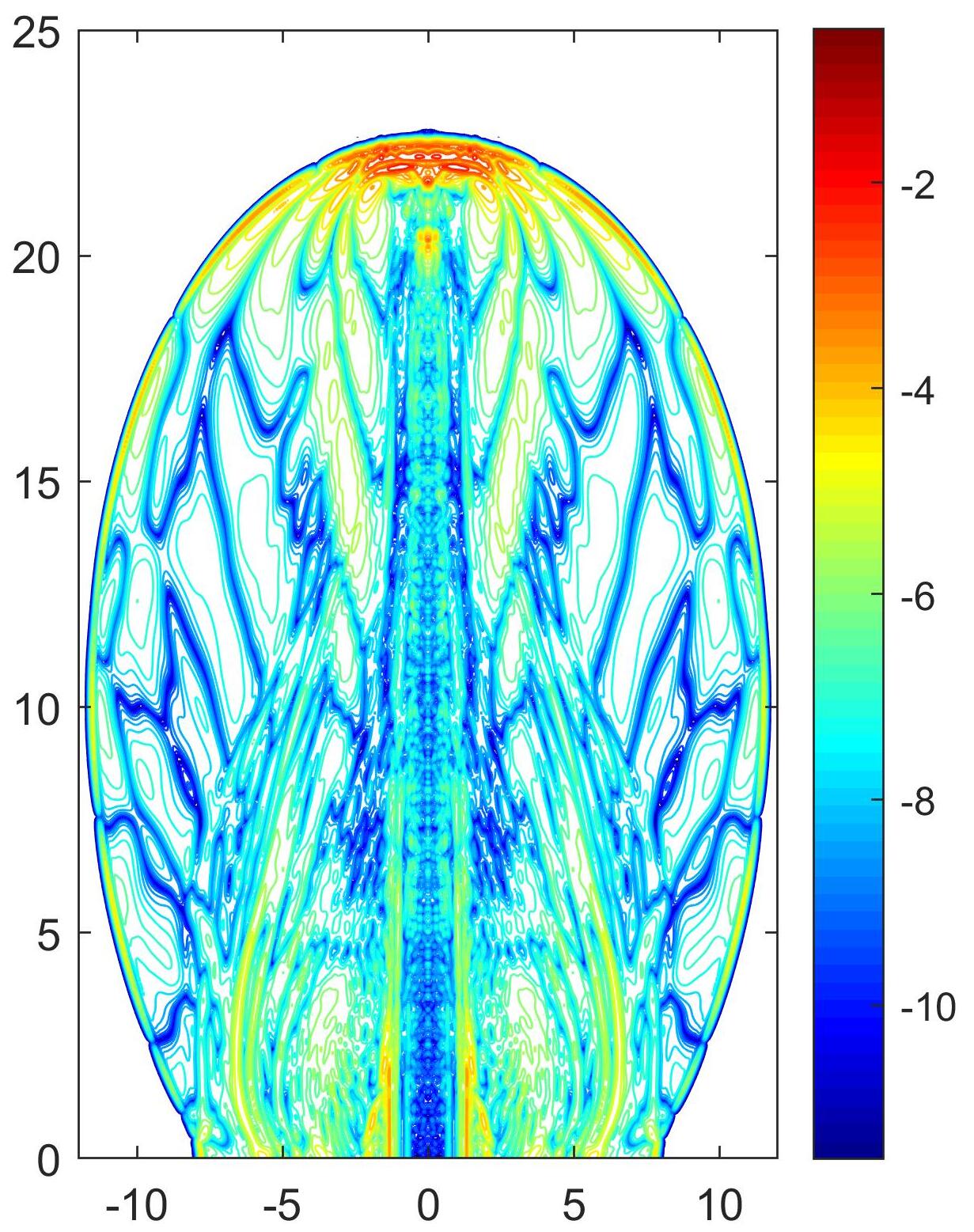}
	\end{subfigure}
	\caption{Example \ref{EX:jet}: Top: Schlieren images of $\ln\rho$ obtained using WENO5 with our flux limiter \eqref{16mat2D}; bottom: contours of the logarithm $\log(|\rho_{\text{\tt ex}}-\rho_{\text{\tt re}}|)$, where $\rho_{\text{\tt ex}}$ and $\rho_{\text{\tt re}}$ denote the density computed with the exact and relaxed flux limiters, respectively. 25 equally spaced contour levels from -11.5 to -0.5 are displayed. From left to right: configurations (i) at $t=30$, (ii) at $t=25$, and (iii) at $t=23$.}\label{Fig:2D_Jet_lnrho}
\end{figure}

\begin{figure}[!htb]
	\centering
	\begin{subfigure}[t]{.32\textwidth}
		\centering
		\includegraphics[width=1\textwidth]{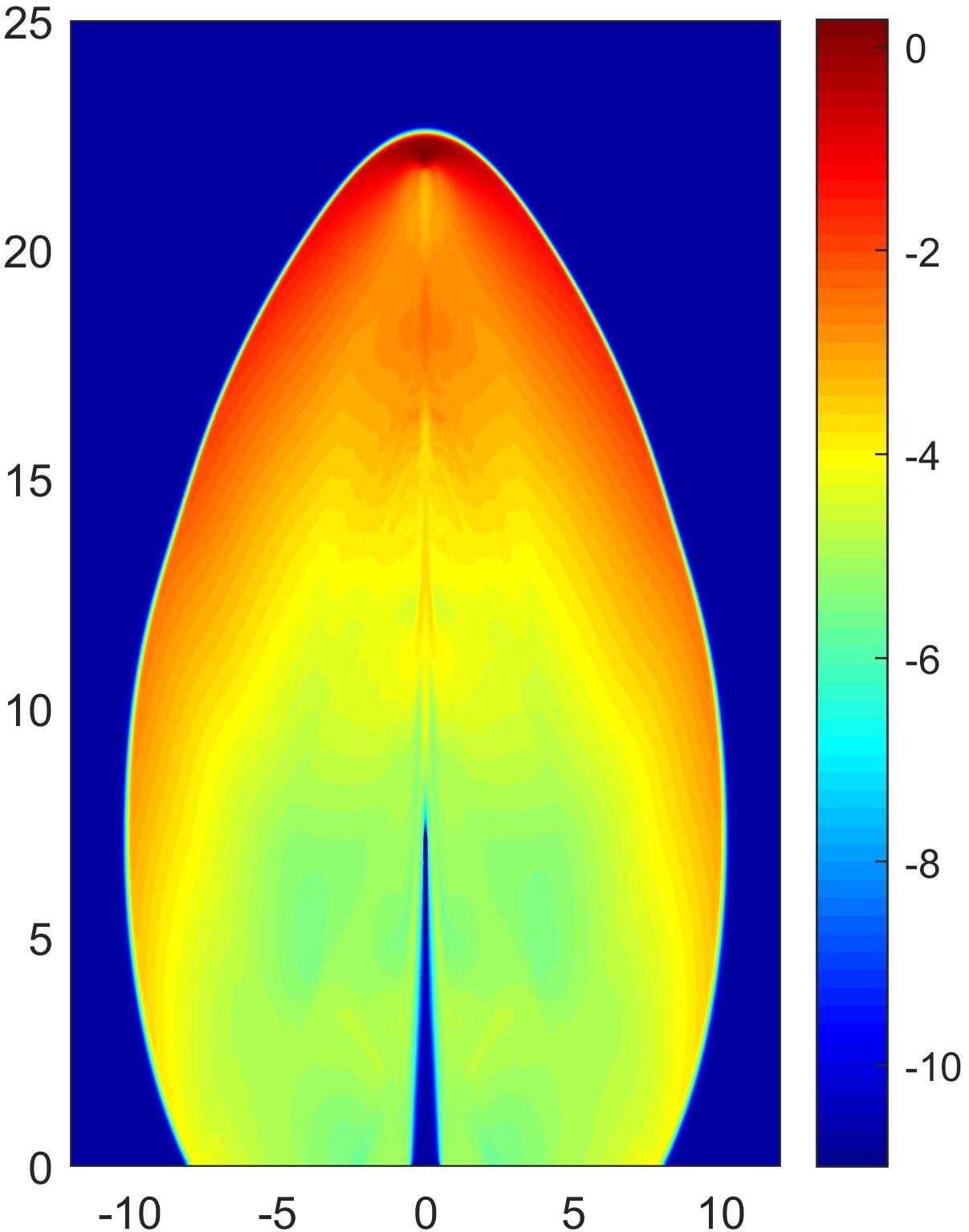}
	\end{subfigure}
    \begin{subfigure}[t]{.32\textwidth}
		\centering
		\includegraphics[width=1\textwidth]{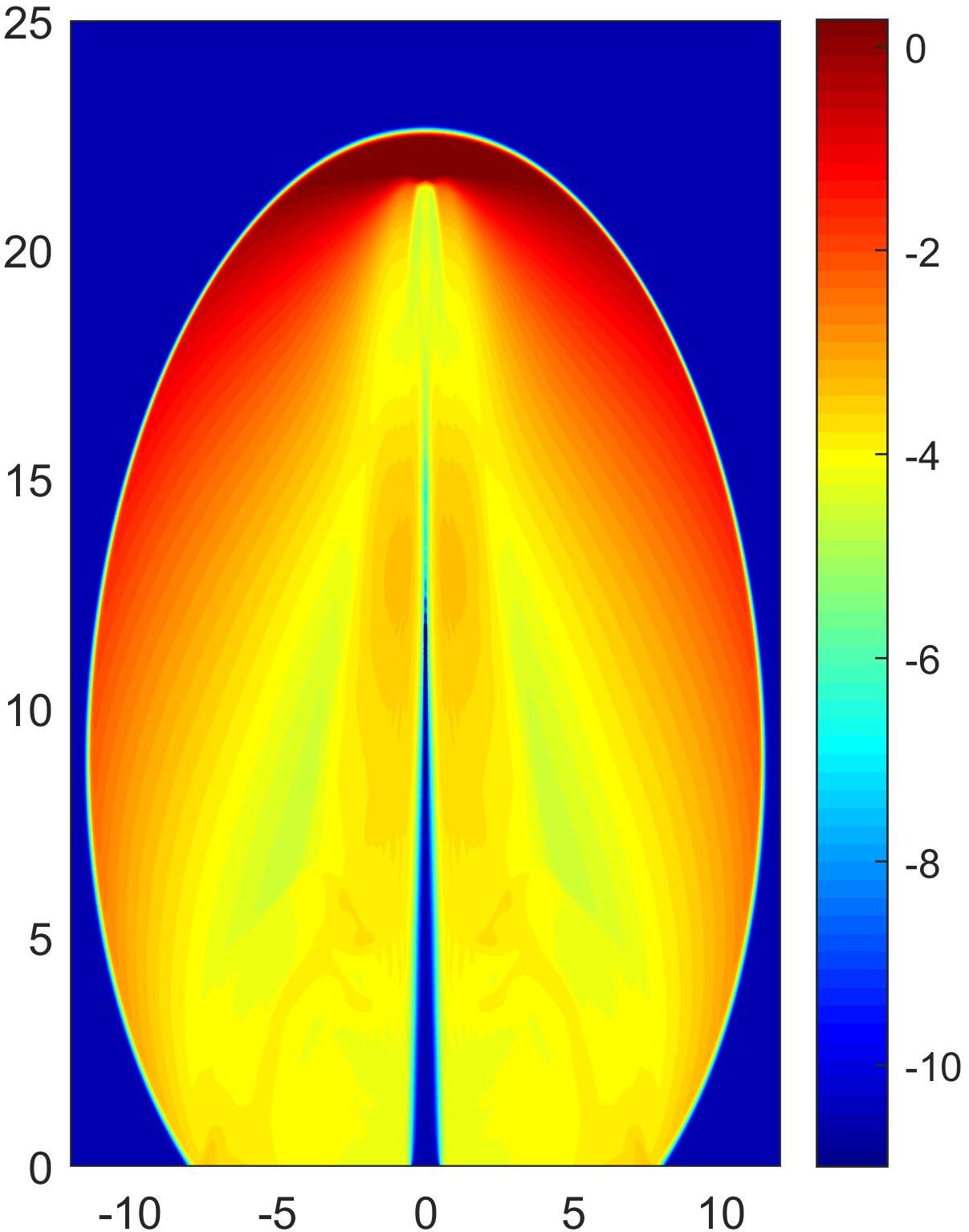}
	\end{subfigure}
 	\begin{subfigure}[t]{.32\textwidth}
		\centering
		\includegraphics[width=1\textwidth]{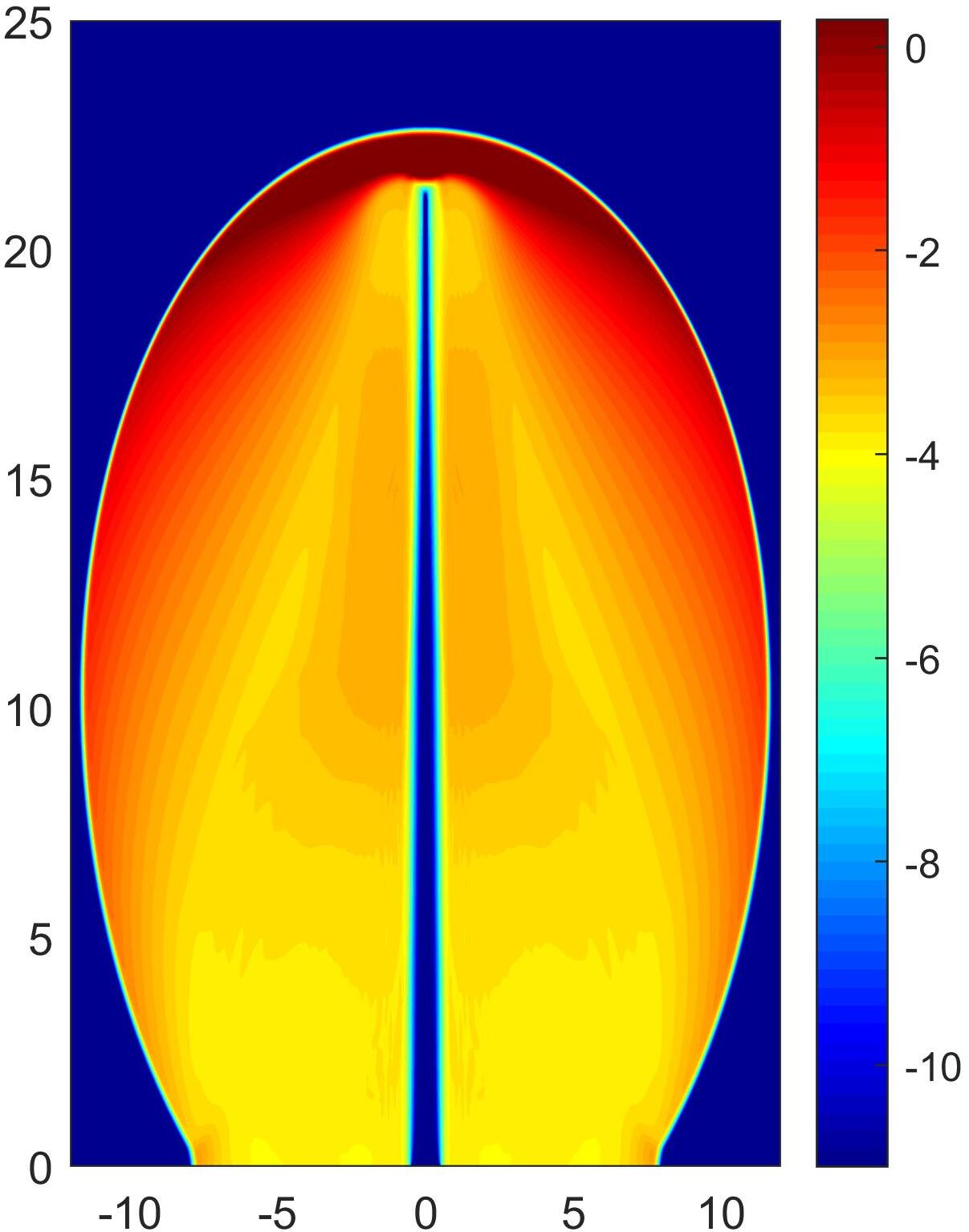}
	\end{subfigure}
    \begin{subfigure}[t]{.32\textwidth}
		\centering
		\includegraphics[width=1\textwidth]{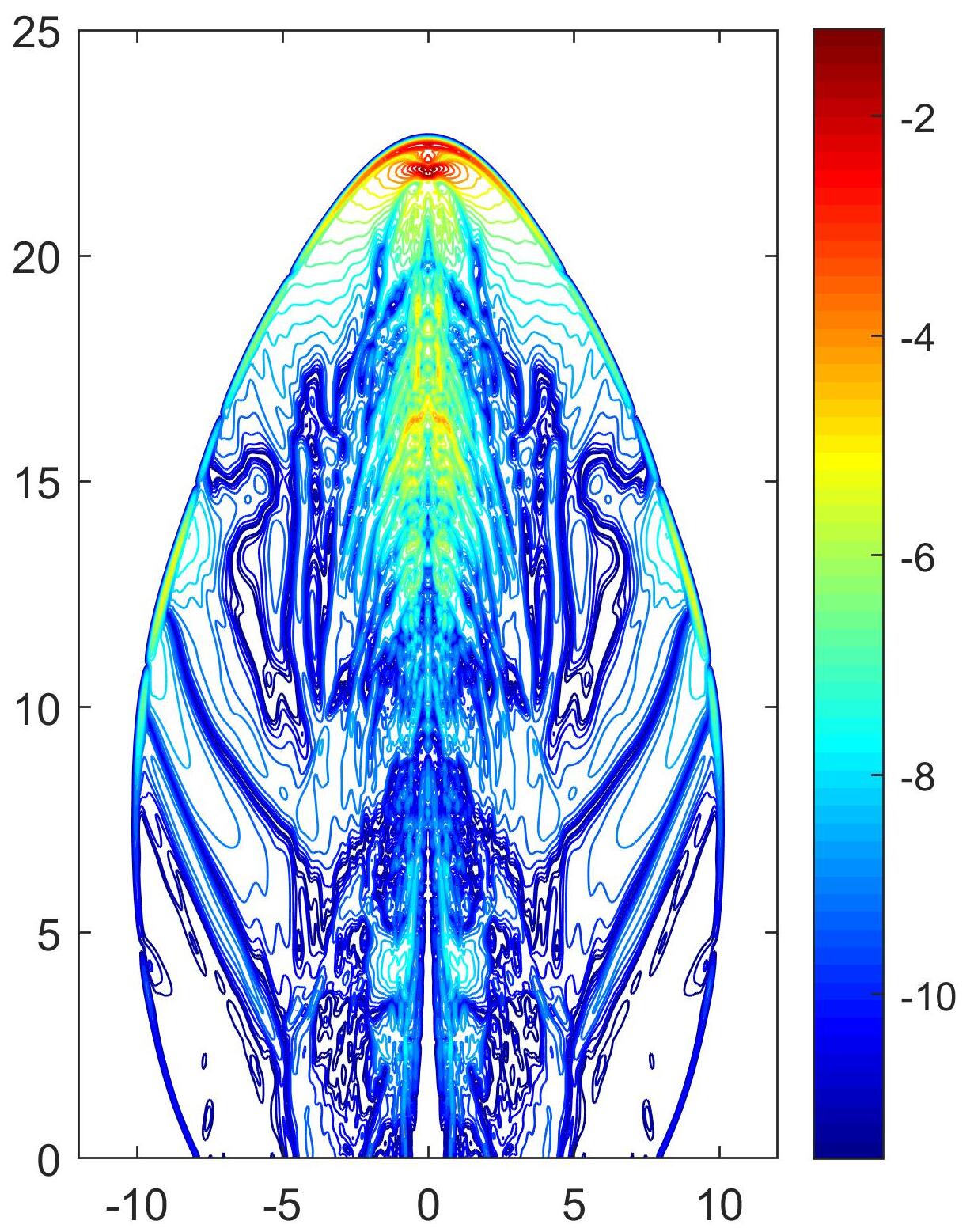}
	\end{subfigure}
	\begin{subfigure}[t]{.32\textwidth}
		\centering
		\includegraphics[width=1\textwidth]{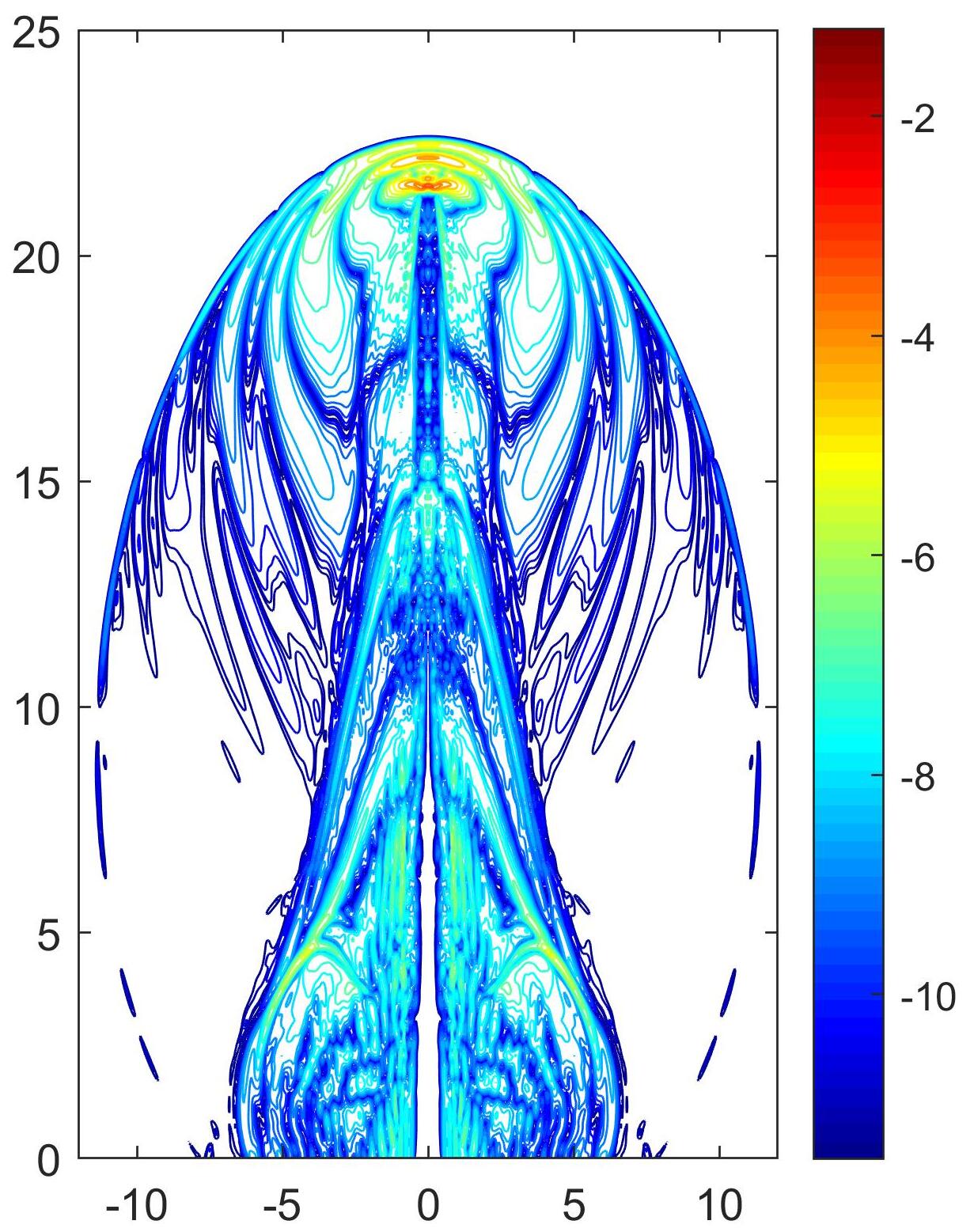}
	\end{subfigure}
    \begin{subfigure}[t]{.32\textwidth}
		\centering
		\includegraphics[width=1\textwidth]{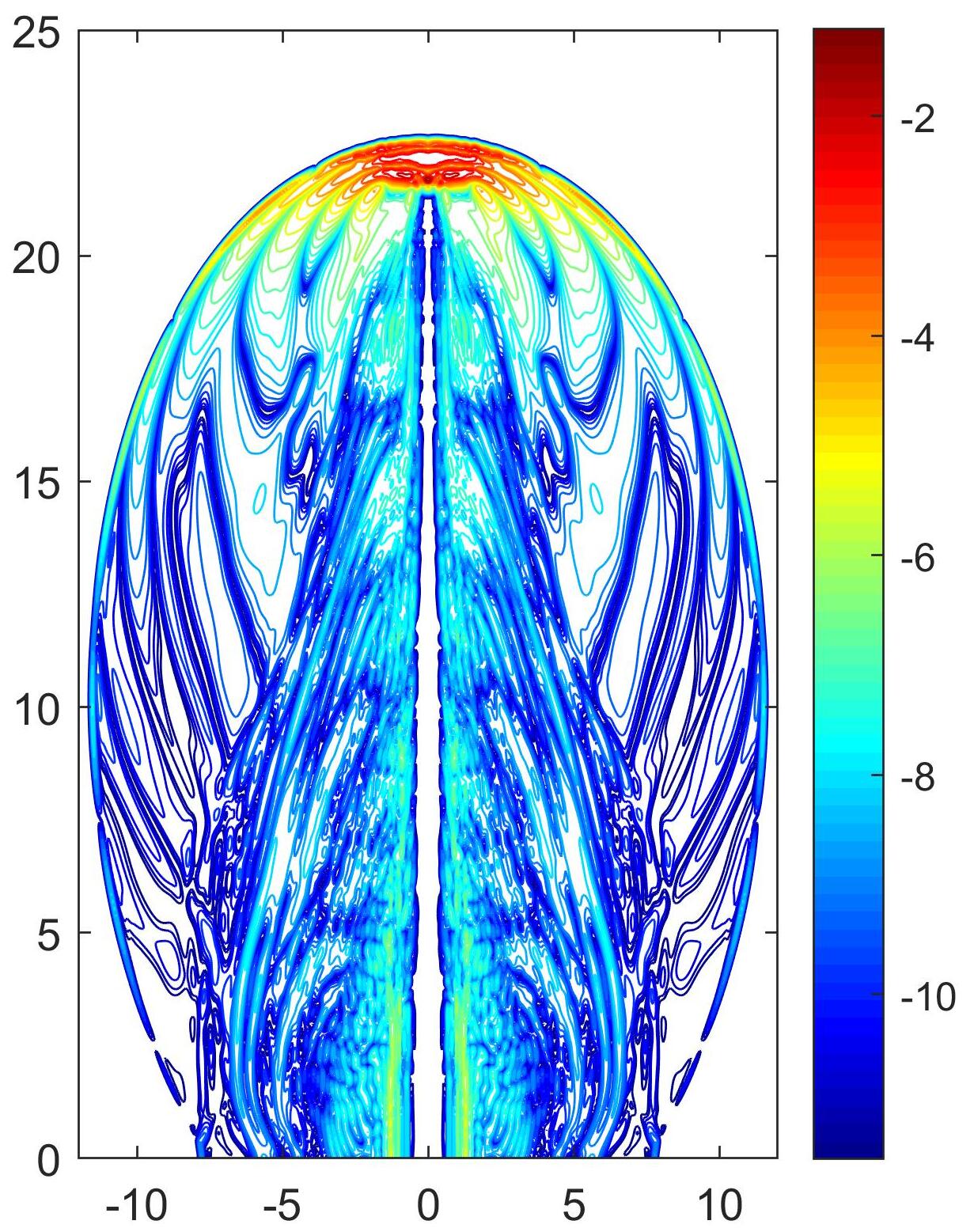}
	\end{subfigure}
	\caption{Example \ref{EX:jet}: Top: Schlieren images of $\ln p$ obtained using WENO5 with our flux limiter \eqref{16mat2D}; bottom: contours of the logarithm $\log(|p_{\text{\tt ex}}-p_{\text{\tt re}}|)$, where $p_{\text{\tt ex}}$ and $p_{\text{\tt re}}$ denote the pressure computed with the exact and relaxed flux limiters, respectively. 25 equally spaced contour levels from -11.5 to -1.2 are displayed. From left to right: configurations (i) at $t=30$, (ii) at $t=25$, and (iii) at $t=23$.}\label{Fig:2D_Jet_lnpre}
\end{figure}

\begin{example}{(2D Axisymmetric relativistic jets)} \label{EX:Axisjet}
This test considers an axisymmetric relativistic jet simulation in cylindrical coordinates $(r,z)$. Compared with the rectangular coordinates $(x,y)$, the RHD equations in $(r,z)$ coordinates
\begin{equation*}
    \frac{\partial{\bf U}}{\partial t}+\frac{\partial{\bf F}_1({\bf U})}{\partial r}+\frac{\partial {\bf F}_2({\bf U})}{\partial z} = {\bf S}({\bf U}, r),
\end{equation*}
are different by the source term
\begin{equation*}
    {\bf S}({\bf U}, r) = -\frac{1}{r}\left(Dv_1,m_1v_1,m_2v_1,m_1\right)^\top.
\end{equation*}
In the case of rectangular coordinates, we have ${\bf U}_{ij}+\Delta t\mathscr{L}({\bf U}_{ij})\in\mathcal{G}$ by enforcing our flux limiter. When transforming into cylindrical coordinates, we should update ${\bf U}_{ij}+\Delta t\left(\mathscr{L}({\bf U}_{ij})+{\bf S}({\bf U}_{ij}, r_i)\right)$ in each time step. To analyze the possible extra PCP condition caused by the source term, we rewrite
\begin{equation*}
    {\bf U}_{ij}+\Delta t\left(\mathscr{L}({\bf U}_{ij})+{\bf S}({\bf U}_{ij}, r_i)\right)
    = (1-\beta)\left({\bf U}_{ij}+\frac{\Delta t}{1-\beta}\mathscr{L}({\bf U}_{ij})\right)
    + \beta \left({\bf U}_{ij}+\frac{\Delta t}{\beta}{\bf S}({\bf U}_{ij}, r_i)\right),\quad \beta \in (0,1).
\end{equation*}
It is sufficient to ensure that ${\bf U}_{ij}+\frac{\Delta t}{1-\beta}\mathscr{L}({\bf U}_{ij})\in\mathcal{G}$ and ${\bf U}_{ij}+\frac{\Delta t}{\beta}{\bf S}({\bf U}_{ij}, r_i)\in\mathcal{G}$ when ${\bf U}_{ij}\in\mathcal{G}$ for all $i,j$. The first part ${\bf U}_{ij}+\frac{\Delta t}{1-\beta}\mathscr{L}({\bf U}_{ij})\in\mathcal{G}$ can be achieved by enforcing our flux limiter after replacing $\Delta t$ in ${Q}_{ij}^-:=\frac{\Delta x\Delta y}{\Delta t}$ by $\frac{\Delta t}{1-\beta}$. The second part is true if 
\begin{equation}\label{source:beta}
    \Delta t \leq \beta A_s,\quad
    A_s:=\min\limits_{\{i,j\}\in\mathcal{P}_v}\left\{\frac{i\Delta r\ q({\bf U}_{ij})}{\left(p({\bf U}_{ij})+q({\bf U}_{ij})\right)\left|v_1({\bf U}_{ij})\right|}\right\},
\end{equation}
where $\mathcal{P}_v=\{(i,j):\ i,j\in\mathbb{Z},\ v_1({\bf U}_{ij}) > 0\}$.

We adopt the pressure-matched, highly supersonic C2 jet model, also referred to as the cold jet. Initially, the computational domain $[0,15]\times[0,45]$ is filled with the ambient medium, whose state is set as
\begin{equation*}
    {\bf V}(r,z,0) = \left(1,0,0,1.70304823218172071\times10^{-4}\right)^\top.
\end{equation*}
The settings of boundary conditions are the same as those in Example \ref{EX:jet} except for two minor differences. The rest-mass density and the injected velocity of the jet beam are $\rho^b=0.01$ and $v^b=0.99$ in this cold jet model, and the nozzle on the bottom boundary $y=0$ is changed into $\{z=0,\ r\leq1\}$. The beam Mach number is $M_b=6$, corresponding to a relativistic Mach number $M_r\approx41.95$. 

{
The left panel of Fig. \ref{Fig:2D_AxisJet2} shows the schlieren images of the logarithmic rest-mass density obtained with exact flux limiter \eqref{16mat2D} within the symmetric domain $[-15,15]\times[0,45]$ at $t = 100$ on a uniform $360\times1080$ mesh. }
The simulation employs WENO5 coupled with both versions of our flux limiter \eqref{16mat2D} and \eqref{relax2D}, in which the source terms are specially treated by setting $\beta = 0.1$ in \eqref{source:beta} to maintain solution admissibility. The results clearly show the bow shock at the jet head, large-scale vortex structures generated by Kelvin--Helmholtz instabilities, and the resulting turbulent mixing region. These morphological features align well with results reported in \cite{2006Zhang, chen2022physical, WuTang2015}, which were obtained using adaptive mesh refinement RHD codes or high-order positivity-preserving WENO methods. 
{
The right panel of Fig. \ref{Fig:2D_AxisJet2} shows the pointwise logarithmic absolute difference between the exact \eqref{16mat2D} and relaxed \eqref{relax2D} limiters.
Quantitatively, the maximum absolute difference of rest-mass density between the exact and relaxed limiters reaches 1.727, occurring on the symmetry axis $x=0$, where the leading bow shock is located; the corresponding density values from the exact and relaxed schemes are 2.6583 and 4.3853, respectively. For pressure, the maximum absolute difference is 0.1284, also concentrated on the symmetry axis, with the exact and relaxed values being 0.1454 and 0.0169 at the peak point. This deviation pattern is physically reasonable: the relaxed limiter introduces a small amount of extra dissipation as a conservative guarantee of the PCP property, and this effect accumulates most prominently in the high-gradient bow shock region where the limiter is activated most frequently. The average absolute differences are approximately $7.861\times10^{-3}$ for density and $2.321\times10^{-4}$ for pressure, both orders of magnitude smaller than the local peak. Therefore, the minor accuracy loss is overall acceptable when both limiters strictly enforce the physical admissibility constraints across the entire domain.
}

{
Fig.~\ref{Fig:2D_AxisJet2_LimiterPercent} presents the limiter activation frequency maps at two representative time instants, $t=20$ and $t=100$. {For comparison, we adopt a higher-contrast colormap and unify the colorbar range for both subfigures.} The color scale represents the proportion of time steps during which the limiter was activated at each cell, i.e., the fraction of time steps where one of the four limiter coefficients $\theta_{i\pm1/2,j}$ and $\theta_{i,j\pm1/2}$ is strictly less than $1$. Temporally, the maximum activation frequency decreases from 27.5\% at the early stage to 19.95\% at $t=100$, with denser contours in the early phase. This arises as transient sharp shocks during initial jet penetration challenge WENO5 reconstruction, while mature turbulent structures are better resolved with fewer spurious oscillations. Spatially, {the activation pattern matches the flow structures well: the highest frequency concentrates on the symmetry axis where flow conditions are most demanding, and periodic stripes correspond to the Kelvin–Helmholtz vortices along the shear layer. This confirms that the limiter activates locally and adaptively with strong flow features.} 
}

\end{example}

\begin{figure}[!htb]
	\centering
 	\begin{subfigure}[t]{.48\textwidth}
		\centering
		\includegraphics[width=0.96\textwidth]{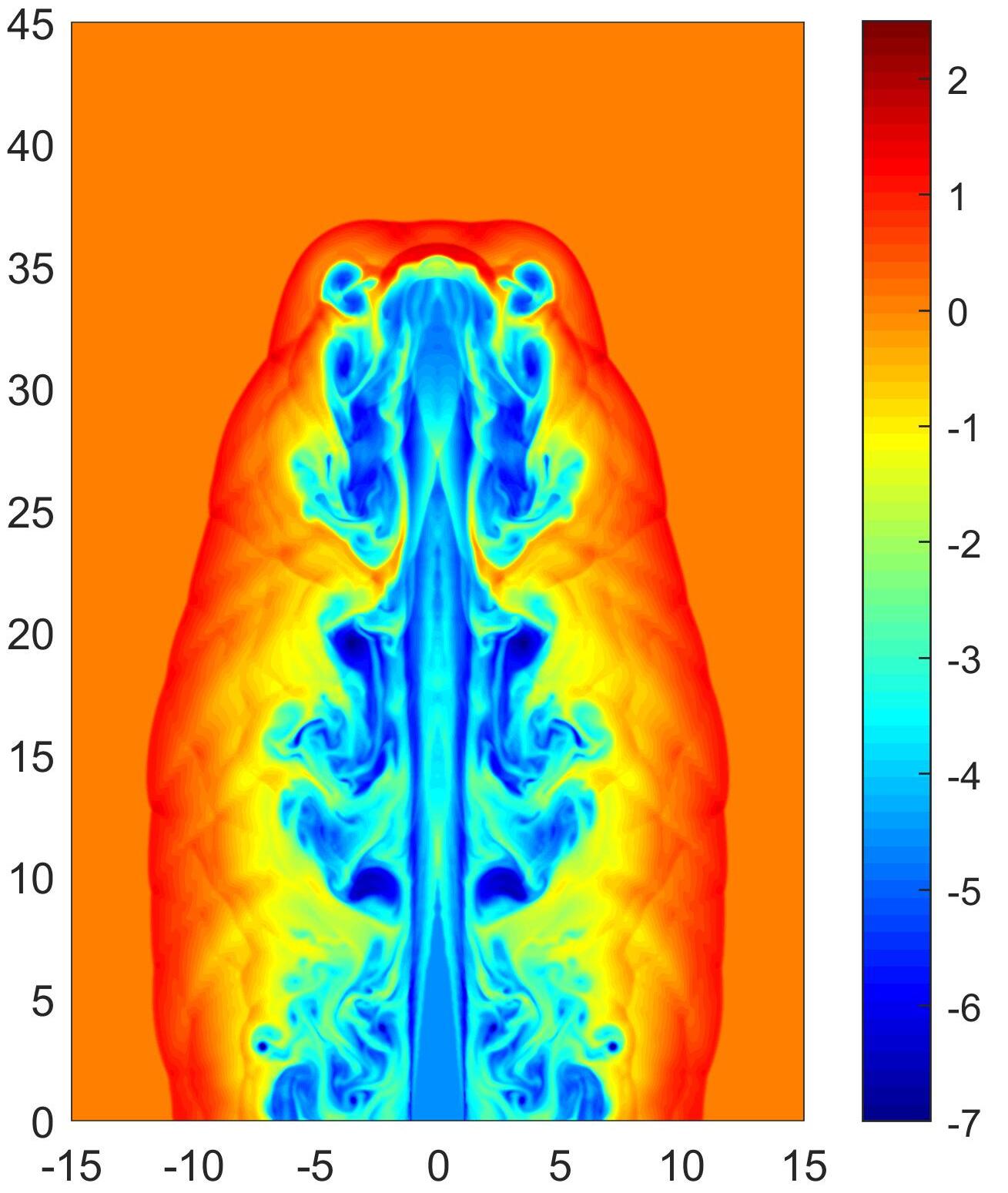}
	\end{subfigure}
    \begin{subfigure}[t]{.48\textwidth}
		\centering
		\includegraphics[width=0.96\textwidth]{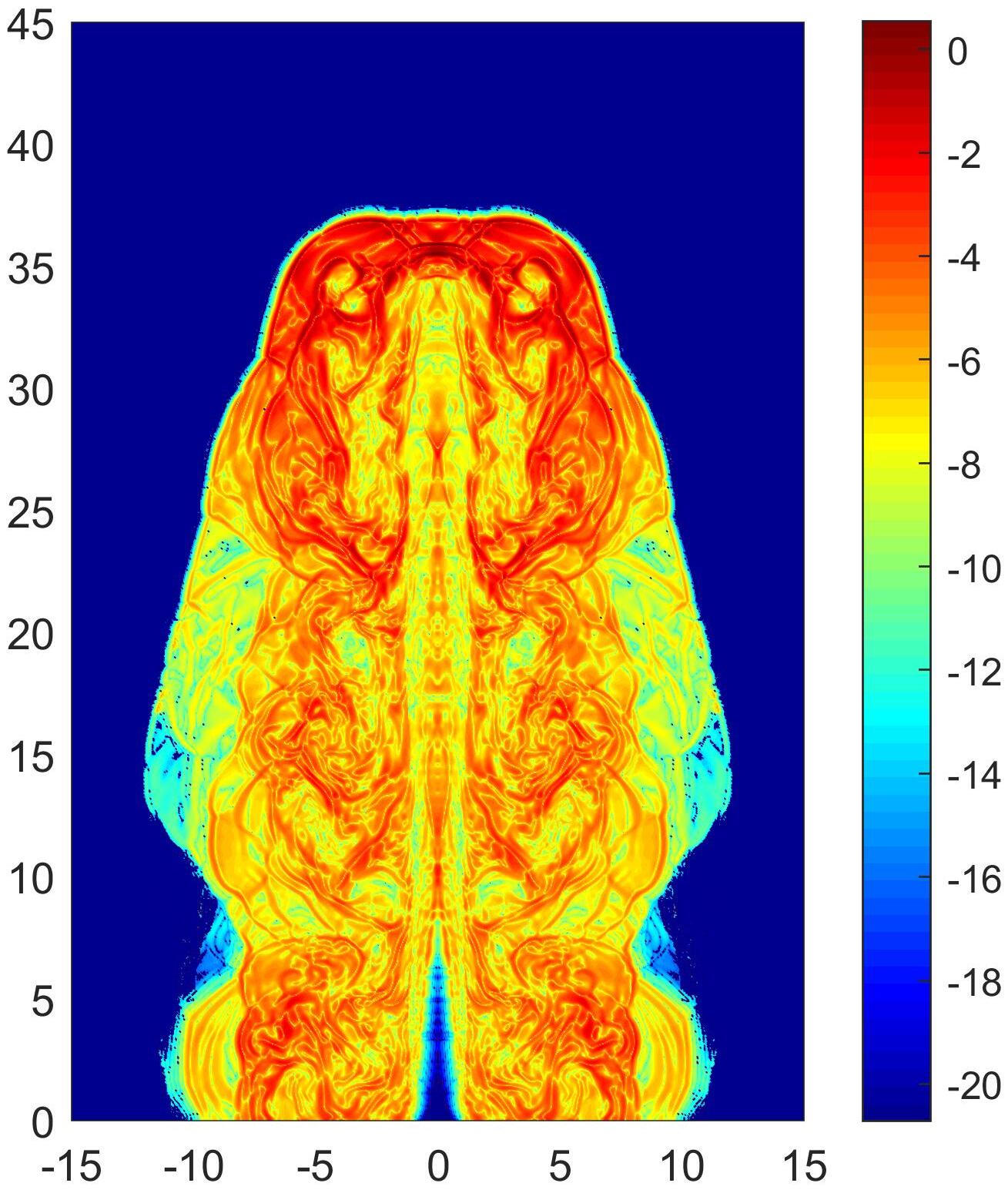}
	\end{subfigure}
	\caption{Example \ref{EX:Axisjet}: Left: Schlieren images of $\ln\rho$ obtained using WENO5 with our flux limiter \eqref{16mat2D} at $t = 100$; bottom: Schlieren images of the logarithm $\log(|\rho_{\text{\tt ex}}-\rho_{\text{\tt re}}|)$ at $t = 100$, where $\rho_{\text{\tt ex}}$ and $\rho_{\text{\tt re}}$ denote the density computed with the exact and relaxed flux limiters, respectively.}\label{Fig:2D_AxisJet2}
\end{figure}

\begin{figure}[!htb]
	\centering
 	\begin{subfigure}[t]{.48\textwidth}
		\centering
		\includegraphics[width=0.96\textwidth]{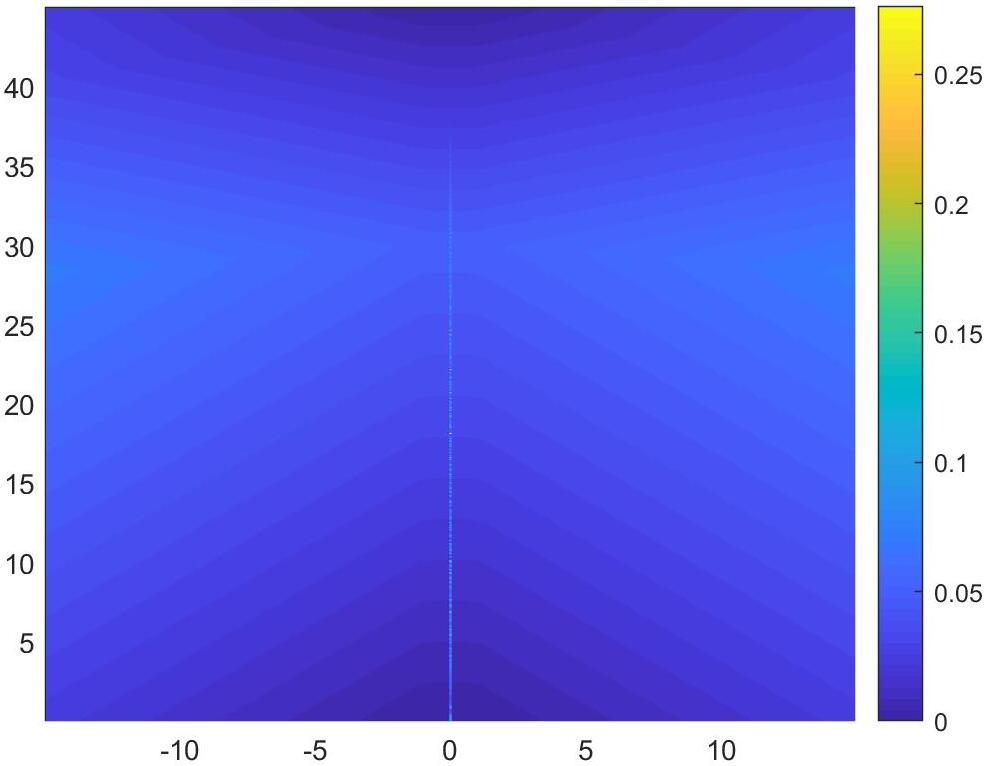}
	\end{subfigure}
    \begin{subfigure}[t]{.48\textwidth}
		\centering
		\includegraphics[width=0.96\textwidth]{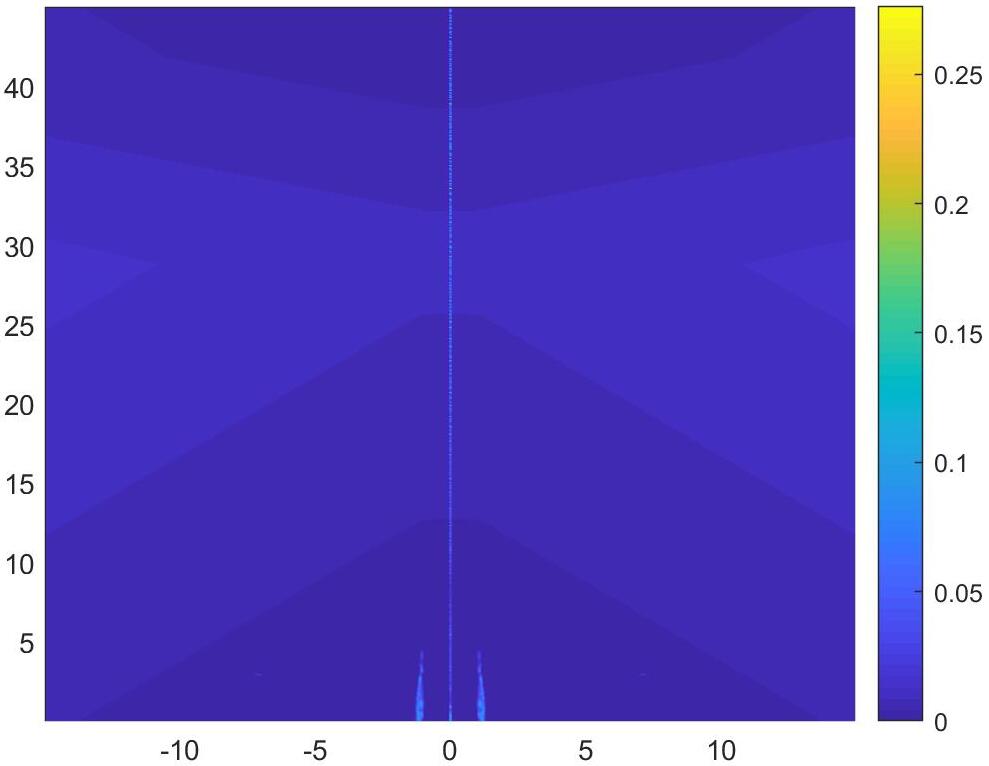}
	\end{subfigure}
	\caption{Example \ref{EX:Axisjet}: Limiter activation frequency maps. Left: $t=20$; right: $t=100$. Both subfigures share a unified color scale for cross-time comparison, and activation patterns align with the key flow structures. The color scale represents the proportion of time steps during which the limiter was activated at each cell (zero means that the limiter is not activated).}\label{Fig:2D_AxisJet2_LimiterPercent}
\end{figure}

\begin{example}[3D Shock-bubble interaction problems]\label{Ex:SB}
The next 3D example simulates the shock-bubble interaction problem within the computational domain $[0,325]\times[0,90]\times[0,90]$. Reflective boundary conditions are applied at $y=0$, $y=90$, $z=0$, and $z=90$, inflow boundary conditions at $x=325$, and outflow boundary conditions at $x=0$. 
The left and right states of the shock are initialized as follows:
\begin{equation*}
  \mathbf{V}(x,y,z,0) = 
    \begin{cases}
    {(1, 0, 0, 0, 0.05)^\top}&  x<265,\\
    {(1.865225080631180,-0.196781107378299, 0, 0, 0.15)^\top}& x>265. 
    \end{cases}
\end{equation*}
The state of the bubble differs from that in {the numerical RHD study} \cite{duan2021entropy}: 
\begin{equation*}
    \textbf{V}(x,y,z,0) = {(0.01358, 0, 0, 0, 0.05)^\top},\quad\quad\quad\sqrt{(x-215)^2+(y-45)^2+(z-45)^2}\leq 25.
\end{equation*}
The density of the bubble is specifically adjusted to construct a demanding scenario, which imposes strict requirements on the numerical scheme, making the adoption of our proposed flux limiter indispensable for achieving stable and accurate simulations. 

\begin{figure}[!htb]
	\centering
    \begin{subfigure}[t]{.48\textwidth}
    	\includegraphics[width=1\textwidth]{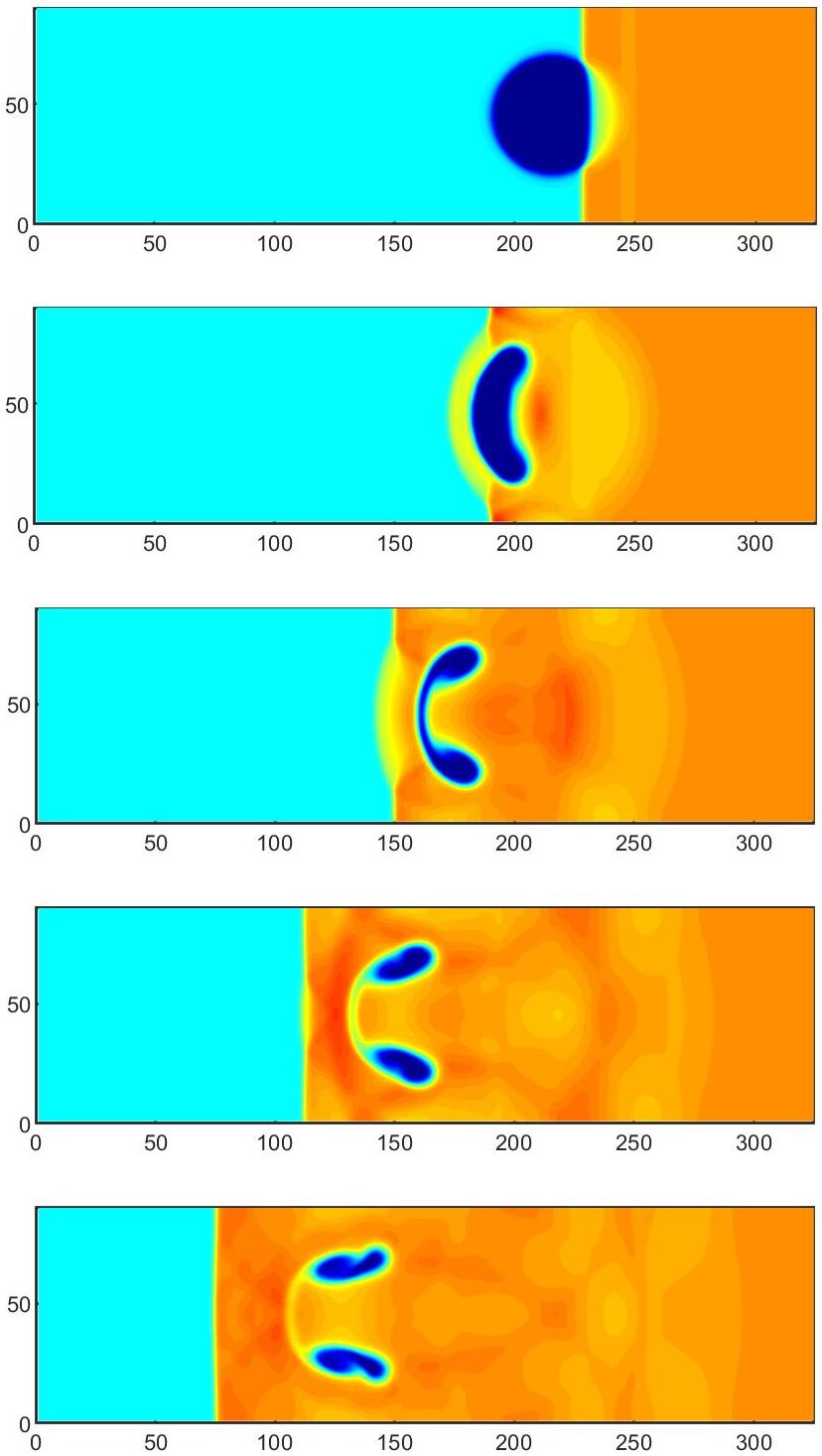}
     \end{subfigure}
     \begin{subfigure}[t]{.48\textwidth}
    	\includegraphics[width=1\textwidth]{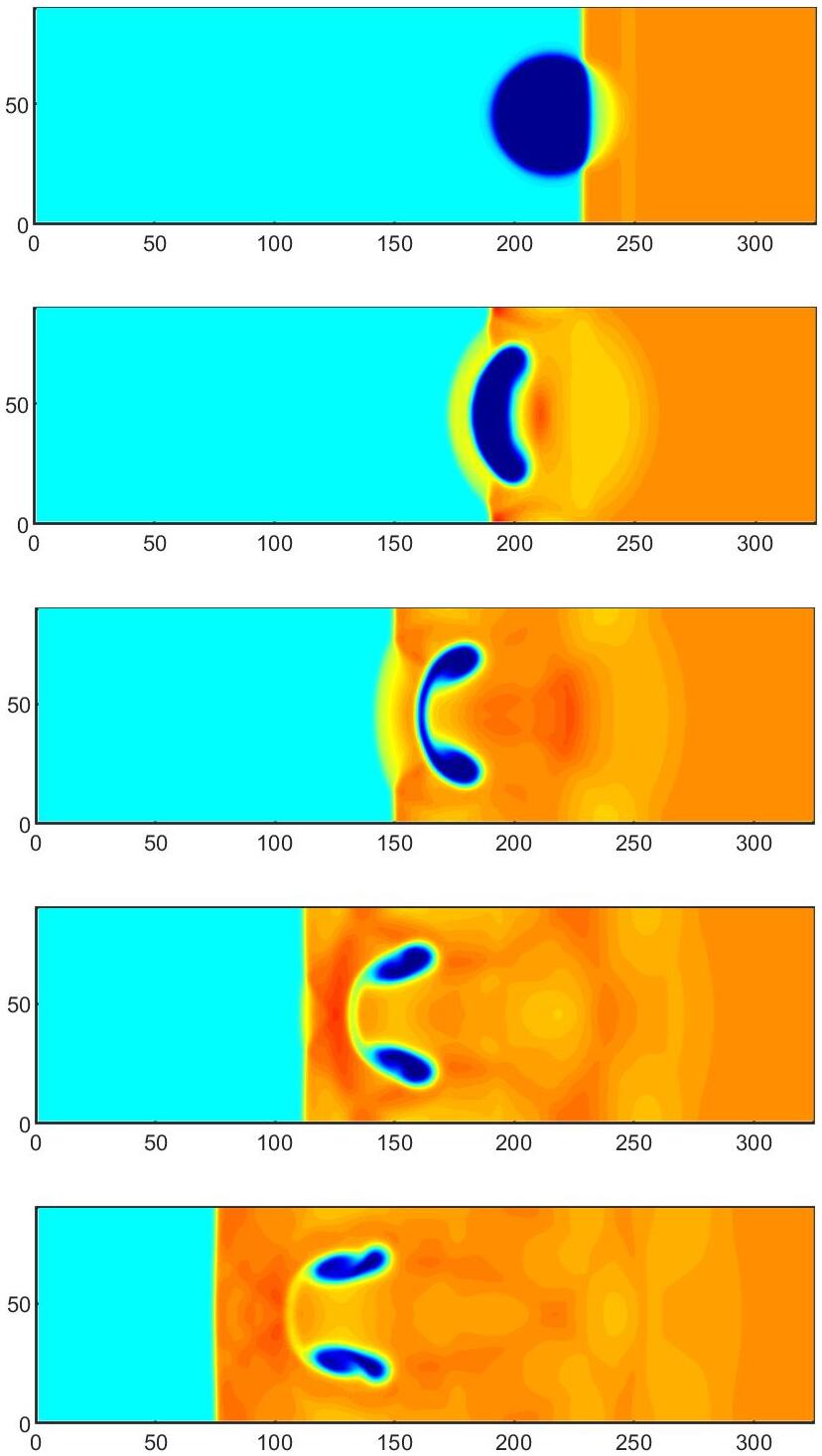}
     \end{subfigure}
	\caption{Example \ref{Ex:SB}: Schlieren images of $\rho$ at $t=90,180,270,360,450$ (from top to bottom). Left: on the slice {$z=45$}; right: on the slice {$y=45$}. The two orthogonal slices jointly demonstrate the rotational symmetry of the 3D flow field.}\label{Fig:2D_shockbb_2Dslice}
\end{figure}

\begin{figure}[!htb]
	\centering
	\begin{subfigure}[t]{.32\textwidth}
		\centering
		\includegraphics[width=1\textwidth]{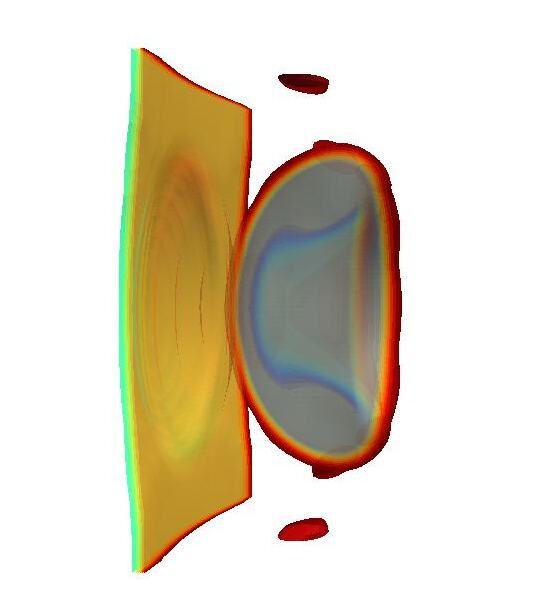}
		\caption{$t = 270$.}
	\end{subfigure}
    \begin{subfigure}[t]{.32\textwidth}
		\centering
		\includegraphics[width=1\textwidth]{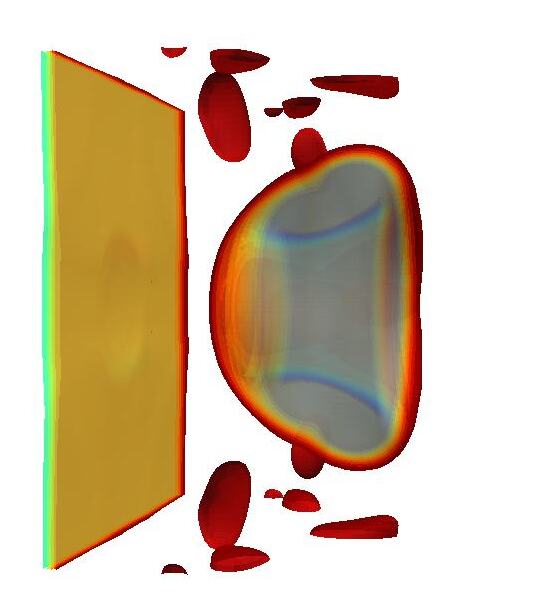}
		\caption{$t = 360$.}
	\end{subfigure}
 	\begin{subfigure}[t]{.32\textwidth}
		\centering
		\includegraphics[width=1\textwidth]{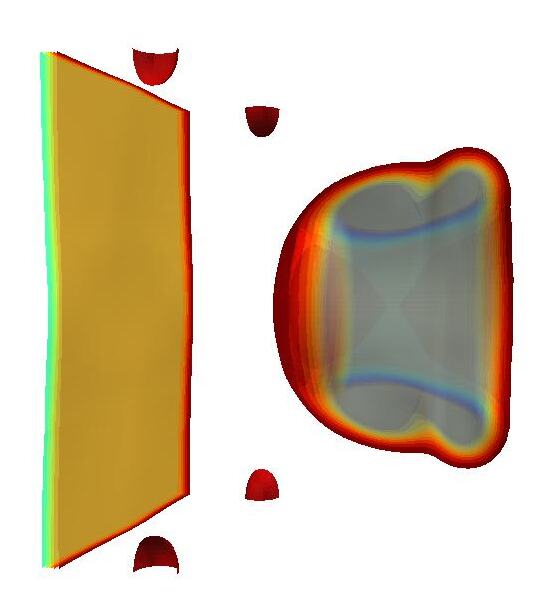}
		\caption{$t = 450$.}
	\end{subfigure}
	\caption{Example \ref{Ex:SB}: 10 iso-surfaces of $\rho$ equally spaced from 0.55 to 1.75, obtained using WENO5 with the relaxed PCP flux limiter \eqref{relax3D}.}\label{Fig:3D_SB_rho}
\end{figure}

To comprehensively visualize the interaction between the shock and the bubble, Fig. \ref{Fig:2D_shockbb_2Dslice} presents the schlieren images of the rest-mass density $\rho$ at $t=90,180,270,360,450$ on the mid-plane slices {$y=45$ and $z=45$}. 
{
The two orthogonal mid-plane slices complement each other: they capture the shock-bubble interaction dynamics from different viewing angles and provide intuitive evidence that the numerical scheme well preserves the rotational symmetry of the initial configuration.
}
Meanwhile, Fig. \ref{Fig:3D_SB_rho} displays the 3D iso-surfaces of $\rho$ at three later time instants ($t = 270,\ 360,\ 450$), consisting of 10 equally spaced iso-values ranging from 0.55 to 1.75. The 3D visualization intuitively depicts the complete three-dimensional dynamic process of the shock-bubble interaction, capturing the spatial morphology of the bubble deformation, shock propagation, and mutual interaction with high fidelity. 
All numerical solutions are obtained using WENO5 with our flux limiter on $325\times90\times90$ uniform grids. As demonstrated by both the 2D schlieren images and 3D iso-surface plots, our scheme effectively captures the dynamics of the interaction between the left-moving shock and the bubble.
\end{example}

\begin{example}{(3D Axisymmetric relativistic jets)}\label{EX:Axisjet3D}
This 3D example is based on the pressure-matched hot A1 model in the 2D case. The adiabatic index $\Gamma = 4/3$. The computational domain is $[0,7]\times[0,7]\times[0,50]$, which is initially filled with uniform static gas with the state
\begin{equation*}
    {\bf V}(x,y,z,0) = \left(1,0,0,0,0.40611878453038897\right)^\top.
\end{equation*}
Through the bottom boundary circular region $\{z=0,\ \sqrt{x^2+y^2}\leq1\}$, a light jet beam with the state $$(\rho^b,v_1^b,v_2^b,v_3^b,p^b) = (0.01,0,0,0.99,0.40611878453038897)$$ is injected parallel to the $z$-axis. The boundary conditions are specified in a manner very similar to the 2D case. The fixed jet beam inflow is imposed on $\{z=0,\ \sqrt{x^2+y^2}\leq1\}$, reflective conditions are applied on planes $x = 0$ and $y = 0$, and outflow conditions are specified on the remaining boundaries.

{Note that the computational domain for this simulation is $[0,7]\times[0,7]\times[0,50]$, which covers one quadrant of the full axisymmetric domain. For the visualization in Fig.~\ref{Fig:3D_AxisJet1_FL_Relax_yhalf}, we mirror the computed solution exactly across the $x=0$ symmetry plane to display the symmetric half-domain $[-7,7]\times[0,7]\times[0,50]$; this post-processing does not alter the numerical solution.} {Fig. \ref{Fig:3D_AxisJet1_FL_Relax_yhalf_density_T20}, \ref{Fig:3D_AxisJet1_FL_Relax_yhalf_density_T40}, and \ref{Fig:3D_AxisJet1_FL_Relax_yhalf_density_T60} present} 15 iso-surfaces of $\ln\rho$ in the symmetric region $[-7,7]\times[0,7]\times[0,50]$ at $t = 20,\ 40$, and 60, using a uniform $70\times70\times500$ mesh. {The corresponding velocity fields are visualized in Fig. \ref{Fig:3D_AxisJet1_FL_Relax_yhalf_vel_T20}, \ref{Fig:3D_AxisJet1_FL_Relax_yhalf_vel_T40}, and \ref{Fig:3D_AxisJet1_FL_Relax_yhalf_vel_T60},} where the arrow direction indicates velocity direction and the arrow scale is proportional to velocity magnitude. The results demonstrate the accurate capture of the time evolution of a light, relativistic jet with high internal energy, and confirm the PCP capability of the proposed flux limiter.
\end{example}

\begin{figure}[!htb]
	\centering
    \begin{subfigure}[t]{.16\textwidth}
		\centering
		\includegraphics[width=1\textwidth]{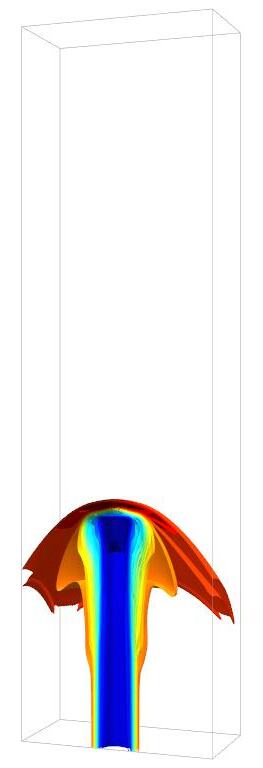}
		\caption{$t=20$.}\label{Fig:3D_AxisJet1_FL_Relax_yhalf_density_T20}
	\end{subfigure}
    \begin{subfigure}[t]{.16\textwidth}
		\centering
		\includegraphics[width=1\textwidth]{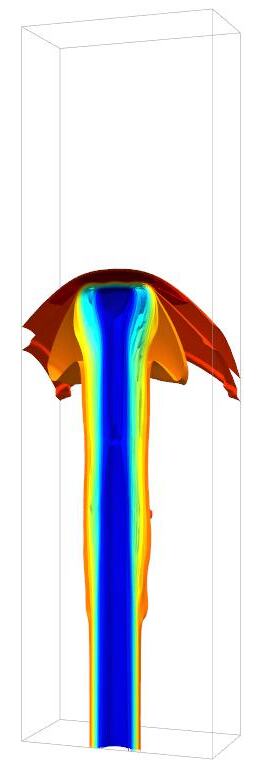}
		\caption{$t=40$.}\label{Fig:3D_AxisJet1_FL_Relax_yhalf_density_T40}
	\end{subfigure}
    \begin{subfigure}[t]{.16\textwidth}
		\centering
		\includegraphics[width=1\textwidth]{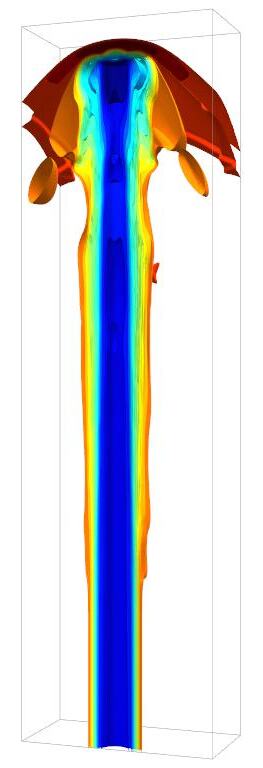}
		\caption{$t=60$.}\label{Fig:3D_AxisJet1_FL_Relax_yhalf_density_T60}
	\end{subfigure}
    \begin{subfigure}[t]{.16\textwidth}
		\centering
		\includegraphics[width=1\textwidth]{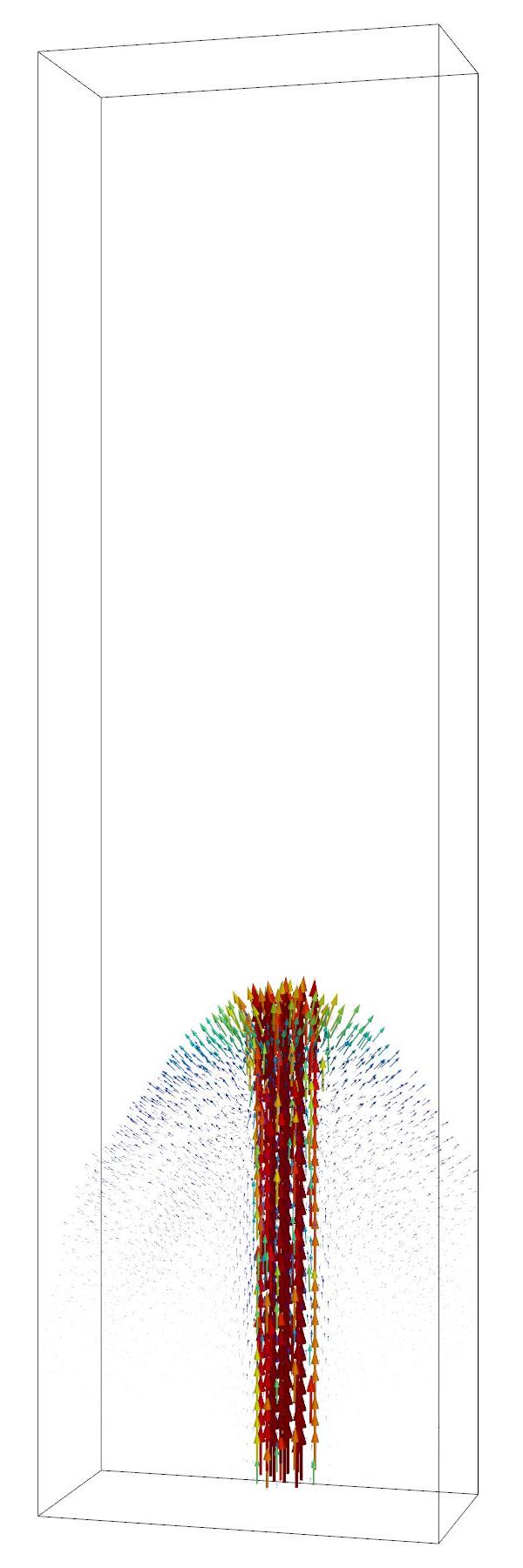}
		\caption{$t=20$.}\label{Fig:3D_AxisJet1_FL_Relax_yhalf_vel_T20}
	\end{subfigure}
    \begin{subfigure}[t]{.16\textwidth}
		\centering
		\includegraphics[width=1\textwidth]{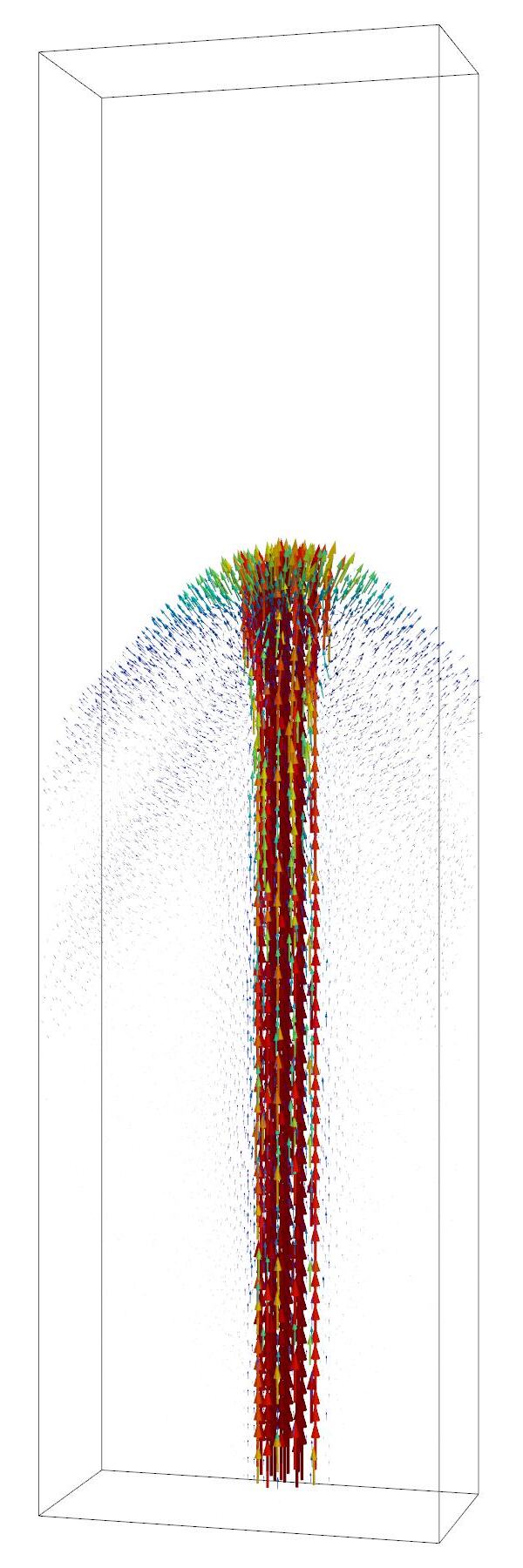}
		\caption{$t=40$.}\label{Fig:3D_AxisJet1_FL_Relax_yhalf_vel_T40}
	\end{subfigure}
    \begin{subfigure}[t]{.16\textwidth}
		\centering
		\includegraphics[width=1\textwidth]{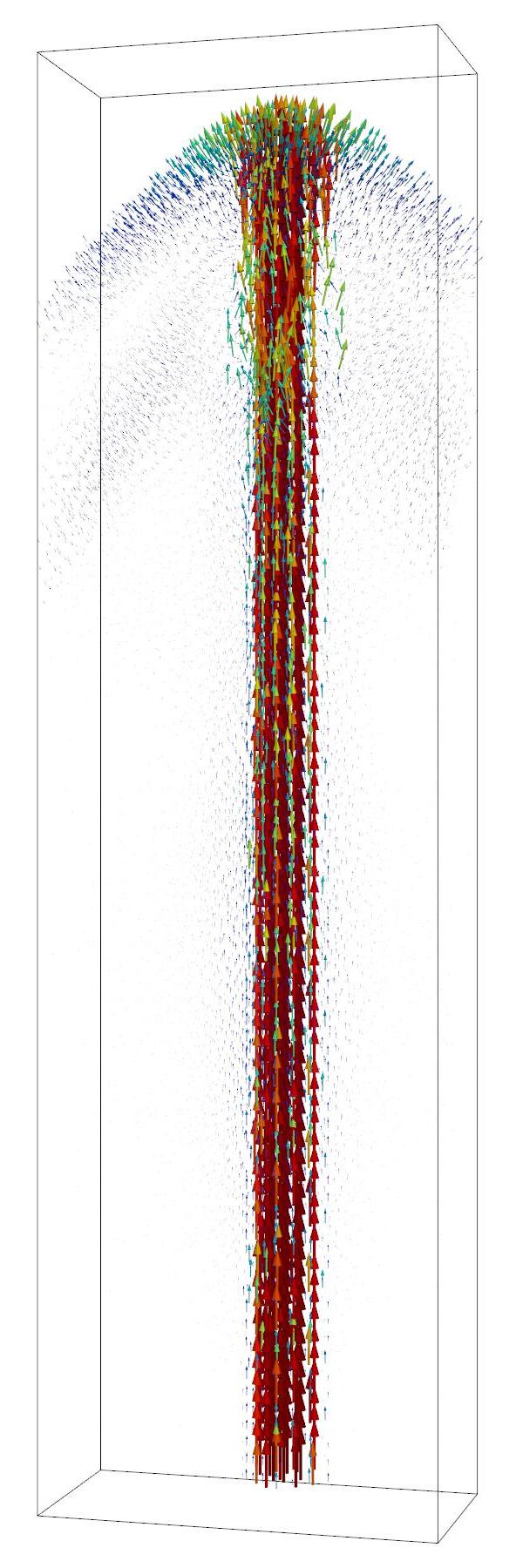}
		\caption{$t=60$.}\label{Fig:3D_AxisJet1_FL_Relax_yhalf_vel_T60}
	\end{subfigure}
	\caption{Example \ref{EX:Axisjet3D}: Numerical results obtained using WENO5 with the relaxed PCP flux limiter \eqref{relax3D}. Left three subfigures: 15 iso-surfaces of $\ln\rho$ equally spaced from -5 to 1; right three subfigures: Slope field visualization of the velocity ${\bm v}$.}\label{Fig:3D_AxisJet1_FL_Relax_yhalf}
\end{figure}


\section{Conclusions}\label{sec:conclusion}
We have developed a robust and efficient physical-constraint-preserving (PCP) flux-limiting framework for high-order schemes, exemplified by finite-difference WENO methods, for special relativistic hydrodynamics. The admissible state set in conservative variables is strictly enforced through the equivalent conditions $D>0$ and $q({\bf{U}})=E-\sqrt{D^2+|{\bm m}|^2}>0$, corresponding to positive rest-mass density, positive pressure, and subluminal velocity.

The central idea is to leverage the geometric quasilinearization (GQL) characterization, which represents the nonlinear constraint $q({\bf{U}})>0$ exactly as the intersection of linear half-space inequalities ${\bf{U}}\cdot{\bm{n}}_*({\bm{v}}_*)>0$ for all auxiliary variables ${\bm{v}}_*\in \mathbb{B}_1({\bf 0})$. This formulation allows us to embed the RHD constraints into a scalar-style Zalesak-type FCT limiter: by projecting the flux-corrected update onto the relevant normals, the constraint enforcement reduces to lower-bound limiting of scalar quantities while the limiter operates directly on conservative variables.

A critical innovation is the {explicit, non-iterative determination} of the limiting parameters via a rational stereographic parameterization of the GQL normal vector. This technique transforms the required worst-case minimization over auxiliary variables into a generalized Rayleigh-quotient formulation, allowing the optimal parameters to be obtained by solving small symmetric eigenvalue problems ($2\times2$ in 1D; $(d+1)\times(d+1)$ in $d$ dimensions). In multidimensions, relaxed variants further reduce the number of eigenvalue evaluations while preserving a rigorous PCP guarantee. Consequently, the proposed parameter estimator avoids expensive iterative searches over ${\bm{v}}_*$ and replaces them with {closed-form} eigenvalue evaluations, resulting in a lightweight limiter suitable for large-scale multidimensional simulations.

Extensive numerical experiments in one to three space dimensions confirm that the resulting schemes robustly enforce physical admissibility, retain fifth-order accuracy for smooth solutions, and sharply resolve strong discontinuities. Future work will focus on less restrictive (yet still explicit) estimators for the $q$-limiting factor, extensions to curvilinear or unstructured meshes and more general equations of state, and applications of the GQL--FCT strategy to other hyperbolic systems with nonlinear invariant domains.


\section*{Data availability}
The data and code that support the findings of this study are available from the corresponding author upon reasonable request.

\section*{Declaration of competing interest}
The authors declare that they have no known competing financial interests or personal relationships that could have appeared to influence the work reported in this paper.


\section*{Acknowledgement}
This work was partially supported by Science Challenge Project (No.~TZ2025007) and the Shenzhen Science and Technology Program (Grant Nos.~JCYJ20250604144300001 and RCJC20221008092757098).


{
\appendix
\section{Complete Pseudocode for the proposed PCP flux Limiter}\label{appendix:code}

In this appendix, we provide complete pseudocode for our physical-constraint-preserving flux-limiting framework. All notation matches the main text.

\subsection{Notation Summary}
\begin{itemize}
    \item ${\bf U}^n$: Conservative state at time $t^n$
    \item $\widehat{\mathcal{F}}^H$: High-order WENO5 numerical flux
    \item $\widehat{\mathcal{F}}^L$: Low-order global Lax-Friedrichs (Rusanov) flux
    \item $\widehat{\mathcal{F}}^A = \widehat{\mathcal{F}}^H - \widehat{\mathcal{F}}^L$: Anti-diffusive flux
    \item $\Delta x, \Delta y, \Delta z, \Delta t$: Mesh size and time step
    \item $\alpha$: Global maximum wave speed
\end{itemize}

\subsection{Main Algorithm: PCP flux Limiter}

\noindent\textbf{Algorithm A.1} \textit{GQL-Based PCP Flux-Limiting Framework (1D/2D/3D)}

\medskip
\noindent\textbf{Input:} Conservative state ${\bf U}^n$, high-order flux $\widehat{\mathcal{F}}^H$, low-order flux $\widehat{\mathcal{F}}^L$, mesh parameters $\Delta x, \Delta y, \Delta z$, time step $\Delta t$, CFL number

\noindent\textbf{Output:} Updated state ${\bf U}^{n+1}$ satisfying ${\bf U}^{n+1} \in \mathcal{G}_\epsilon$

\medskip
\begin{enumerate}
    \item \textbf{Step 1: Compute low-order update (guaranteed to be admissible)}
    
    Compute low-order state ${\bf U}^L$ using first-order LF flux:
    \[
    {\bf U}^L = {\bf U}^n - \frac{\Delta t}{\Delta x} \left( \widehat{\mathcal{F}}^L_{i+1/2,j,k} - \widehat{\mathcal{F}}^L_{i-1/2,j,k} \right) - \frac{\Delta t}{\Delta y} \left( \widehat{\mathcal{G}}^L_{i,j+1/2,k} - \widehat{\mathcal{G}}^L_{i,j-1/2,k} \right) - \frac{\Delta t}{\Delta z} \left( \widehat{\mathcal{H}}^L_{i,j,k+1/2} - \widehat{\mathcal{H}}^L_{i,j,k-1/2} \right)
    \]
    (For 1D, omit $y,z$ terms; for 2D, omit $z$ term.)

    \item \textbf{Step 2: Compute anti-diffusive flux}
    
    $\widehat{\mathcal{F}}^A_{i+1/2} \leftarrow \widehat{\mathcal{F}}^H_{i+1/2} - \widehat{\mathcal{F}}^L_{i+1/2}$
    
    (Similarly for $\widehat{\mathcal{G}}^A, \widehat{\mathcal{H}}^A$ in 2D/3D.)

    \item \textbf{Step 3: Compute density limiter $\theta^D$ (enforces $D \geq \epsilon$)}
    
    \begin{enumerate}
        \item For each cell $i$ (and $j,k$ in 2D/3D):
        \begin{itemize}
            \item $\epsilon_D \leftarrow \min(10^{-13}, D({\bf U}^L_i))$
            \item Compute $P^-_i, Q^-_i, R^-_i$ using Zalesak's notation with ${\bm n}_D = (1, 0, \dots, 0)^\top$
        \end{itemize}
        
        \item For each interface $i+1/2$:
        \begin{itemize}
            \item If $\widehat{\mathcal{F}}^A_{i+1/2} \cdot {\bm n}_D \geq 0$: $\theta^D_{i+1/2} \leftarrow R^-_i$
            \item Else: $\theta^D_{i+1/2} \leftarrow R^-_{i+1}$
        \end{itemize}
    \end{enumerate}
    (Similarly for $y,z$ directions in 2D/3D.)

    \item \textbf{Step 4: Compute q-limiter $\theta^q$ (enforces $q(U) \geq \epsilon$)}
    
    \begin{itemize}
        \item If dimension == 1: $\theta^q \leftarrow \texttt{Compute1DLimiter}({\bf U}^L, \widehat{\mathcal{F}}^A, \Delta x, \Delta t)$
        \item Else: $\theta^q \leftarrow \texttt{ComputeRelaxedLimiter}({\bf U}^L, \widehat{\mathcal{F}}^A, \widehat{\mathcal{G}}^A, \widehat{\mathcal{H}}^A, \Delta x, \Delta y, \Delta z, \Delta t)$
    \end{itemize}

    \item \textbf{Step 5: Apply final limiter and update state}
    
    \begin{itemize}
        \item For each interface: $\theta_{i+1/2} \leftarrow \min(\theta^D_{i+1/2}, \theta^q_{i+1/2})$
        \item Compute flux-corrected update:
        \[
        {\bf U}^{n+1} = {\bf U}^L - \frac{\Delta t}{\Delta x} \left( \theta_{i+1/2} \widehat{\mathcal{F}}^A_{i+1/2} - \theta_{i-1/2} \widehat{\mathcal{F}}^A_{i-1/2} \right)
        \]
    \end{itemize}
\end{enumerate}

\medskip
\noindent\textbf{Return:} ${\bf U}^{n+1}$

\medskip
\subsection{Subroutine: Exact 1D q-Limiter}

\noindent\textbf{Algorithm A.2} \textit{Compute1DLimiter(${\bf U}^L, \widehat{\mathcal{F}}^A, \Delta x, \Delta t$)}

\noindent\textbf{Input:} Low-order state ${\bf U}^L$, anti-diffusive flux $\widehat{\mathcal{F}}^A$, mesh parameters

\noindent\textbf{Output:} q-limiter $\theta^q$

\medskip
\begin{enumerate}
    \item For each cell $i$:
    \begin{enumerate}
        \item \textbf{Step 1: Extract coefficients }
        
        ${\bf U}^L_i = (D^L, m^L, E^L)^\top$
        
        $\epsilon_q \leftarrow \min(10^{-13}, q({\bf U}^L_i))$
        
        $a \leftarrow E^L + D^L - \epsilon_q$, $b \leftarrow 2m^L$, $c \leftarrow E^L - D^L - \epsilon_q$
        
        $d^\pm \leftarrow \widehat{\mathcal{F}}^A_{i\pm1/2} \cdot (1, 0, 1)^\top$, $e^\pm \leftarrow 2\widehat{\mathcal{F}}^A_{i\pm1/2} \cdot (0, 1, 0)^\top$, $f^\pm \leftarrow \widehat{\mathcal{F}}^A_{i\pm1/2} \cdot (-1, 0, 1)^\top$

        \item \textbf{Step 2: Rational stereographic parameterization and Rayleigh quotient}
        
        Define symmetric matrices:
        \[
        B = \begin{pmatrix} a & -b/2 \\ -b/2 & c \end{pmatrix}, \quad {A^\star={\tt sign}(\star)\begin{pmatrix} d^\star & -e^\star/2 \\ -e^\star/2 & f^\star \end{pmatrix},\quad \star\in\{+,-\}}
        \]
        {Here ${\tt sign}(+)=1$ and ${\tt sign}(-)=-1$.}
        
        Compute Cholesky decomposition $B = R^\top R$
        
        ${\widetilde{A}^\star \leftarrow R^{-\top} A^\star R^{-1},\quad \star\in\{+,-\}}$

        \item \textbf{Step 3: Compute maximal eigenvalues}
        
        $\lambda^+_{\text{max}} \leftarrow \texttt{MaxEigenvalue}(\widetilde{A}^+)$
        
        $\lambda^-_{\text{max}} \leftarrow \texttt{MaxEigenvalue}(\widetilde{A}^-)$
        
        $\lambda^*_{\text{max}} \leftarrow \texttt{MaxEigenvalue}(\widetilde{A}^+ + \widetilde{A}^-)$

        \item \textbf{Step 4: Check feasibility of eigenvectors ($|u_*| < 1$)}
        
        For each eigenvalue $\lambda$:
        \begin{itemize}
            \item Compute the corresponding eigenvector $z$
            \item ${\bm w} \leftarrow R^{-1} {\bm z}$
            \item {If the last component of ${\bm w}$ is nonzero, set $u_* \leftarrow w_1/w_2$; otherwise mark this eigenvector as infeasible}
            \item If $|u_*| \geq 1$: Reject this eigenvalue (infeasible)
        \end{itemize}

        \item \textbf{Step 5: Compute cell-wise limiting factor}
        
        {$M_i \leftarrow \max(0, \lambda^+_{\text{max, feasible}}, \lambda^-_{\text{max, feasible}}, \lambda^*_{\text{max, feasible}},\text{ boundary candidates})$}
        
        {$\mathcal{L}_i \leftarrow 1$ if $M_i=0$; otherwise $\mathcal{L}_i\leftarrow \min\left\{1,\dfrac{\Delta x}{\Delta t\,M_i}\right\}$.}
    \end{enumerate}

    \item \textbf{Step 6: Compute interface limiting factors}
    
    For each interface $i+1/2$: $\theta^q_{i+1/2} \leftarrow \min(\mathcal{L}_i, \mathcal{L}_{i+1})$
\end{enumerate}

\medskip
\noindent\textbf{Return:} $\theta^q$

\medskip
\subsection{Subroutine: Relaxed 2D/3D q-Limiter}

\noindent\textbf{Algorithm A.3} \textit{ComputeRelaxedLimiter(${\bf U}^L, \widehat{\mathcal{F}}^A, \widehat{\mathcal{G}}^A, \widehat{\mathcal{H}}^A, \Delta x, \Delta y, \Delta z, \Delta t$)}

\medskip
\noindent\textbf{Input:} Low-order state ${\bf U}^L$, anti-diffusive fluxes in all directions, mesh parameters

\noindent\textbf{Output:} Relaxed q-limiter $\theta^q$

{\noindent\textbf{Note:} \texttt{ComputeDirectionalDenominator} returns the weighted maximal denominator estimate in \eqref{relax2D} or \eqref{relax3D}; the conversion to the actual limiter coefficient is performed only after summing the directional denominator estimates.}

\medskip
\begin{enumerate}
    \item For each cell $i,j,k$:
    
        {\textbf{Step 1: Extract coefficients for each spatial direction}}
        
        Extract $(a, {\bm b}, c)$ from ${\bf U}^L_{ijk}$ (vector ${\bm b}$ for momentum)
        
        Extract $(d^\pm_x, {\bm e}^\pm_x, f^\pm_x)$ from $\widehat{\mathcal{F}}^A_{i\pm1/2,j,k}$
        
        Extract $(d^\pm_y, {\bm e}^\pm_y, f^\pm_y)$ from $\widehat{\mathcal{G}}^A_{i,j\pm1/2,k}$
        
        If dimension == 3:
        \begin{itemize}
            \item Extract $(d^\pm_z, {\bm e}^\pm_z, f^\pm_z)$ from $\widehat{\mathcal{H}}^A_{i,j,k\pm1/2}$
            \item {$M^x_{ijk} \leftarrow \texttt{ComputeDirectionalDenominator}(a, {\bm b}, c, d^\pm_x, {\bm e}^\pm_x, f^\pm_x, \Delta y \Delta z)$}
            \item {$M^y_{ijk} \leftarrow \texttt{ComputeDirectionalDenominator}(a, {\bm b}, c, d^\pm_y, {\bm e}^\pm_y, f^\pm_y, \Delta x \Delta z)$}
            \item {$M^z_{ijk} \leftarrow \texttt{ComputeDirectionalDenominator}(a, {\bm b}, c, d^\pm_z, {\bm e}^\pm_z, f^\pm_z, \Delta x \Delta y)$}
            \item {$M_{ijk} \leftarrow M^x_{ijk} + M^y_{ijk} + M^z_{ijk}$}
            \item {$\mathcal{L}_{ijk}\leftarrow 1$ if $M_{ijk}=0$; otherwise $\mathcal{L}_{ijk}\leftarrow\min\left\{1,\dfrac{\Delta x\Delta y\Delta z}{\Delta t\,M_{ijk}}\right\}$}
        \end{itemize}
        Else:
        \begin{itemize}
            \item {$M^x_{ij} \leftarrow \texttt{ComputeDirectionalDenominator}(a, {\bm b}, c, d^\pm_x, {\bm e}^\pm_x, f^\pm_x, \Delta y)$}
            \item {$M^y_{ij} \leftarrow \texttt{ComputeDirectionalDenominator}(a, {\bm b}, c, d^\pm_y, {\bm e}^\pm_y, f^\pm_y, \Delta x)$}
            \item {$M_{ij} \leftarrow M^x_{ij} + M^y_{ij}$}
            \item {$\mathcal{L}_{ij}\leftarrow 1$ if $M_{ij}=0$; otherwise $\mathcal{L}_{ij}\leftarrow\min\left\{1,\dfrac{\Delta x\Delta y}{\Delta t\,M_{ij}}\right\}$}
        \end{itemize}

    \item \textbf{Step 2: Compute interface limiting factors}
    
    For each interface: $\theta^q \leftarrow \min(\mathcal{L}_{\text{left}}, \mathcal{L}_{\text{right}})$
\end{enumerate}

\medskip
\noindent\textbf{Return:} $\theta^q$

\medskip
\subsection{Auxiliary Function: Maximal Eigenvalue}

\noindent\textbf{Algorithm A.4} \textit{MaxEigenvalue($A$)}

\medskip
\noindent\textbf{Input:} Symmetric matrix $A$ (size 2x2, 3x3, or 4x4)

\noindent\textbf{Output:} Maximal eigenvalue $\lambda_{\text{max}}$

\medskip
\begin{itemize}
    \item Compute eigenvalues using the explicit formulae from the main text
    \item $\lambda_{\text{max}} \leftarrow \text{maximum of all eigenvalues}$
\end{itemize}

\medskip
\noindent\textbf{Return:} $\lambda_{\text{max}}$
}


\bibliographystyle{abbrv}
\bibliography{references}

@article{wang2022affine,
  title={{Affine-invariant WENO weights and operator}},
  author={Wang, Bao-Shan and Don, Wai Sun},
  journal={Applied Numerical Mathematics},
  volume={181},
  pages={630--646},
  year={2022},
  publisher={Elsevier}
}

@article{peng2025oedg,
  title        = {{OEDG}: Oscillation-eliminating discontinuous {Galerkin} method for hyperbolic conservation laws},
  author       = {Peng, Manting and Sun, Zheng and Wu, Kailiang},
  journal      = {Mathematics of Computation},
  volume       = {94},
  number       = {353},
  pages        = {1147--1198},
  year         = {2025},
  doi          = {10.1090/mcom/3998},
  eprint       = {2310.04807},
  archivePrefix = {arXiv},
  primaryClass = {math.NA}
}

@article{borges2008improved,
  title={An improved weighted essentially non-oscillatory scheme for hyperbolic conservation laws},
  author={Borges, Rafael and Carmona, Monique and Costa, Bruno and Don, Wai Sun},
  journal={Journal of computational physics},
  volume={227},
  number={6},
  pages={3191--3211},
  year={2008},
  publisher={Elsevier}
}

@article{wu2017design,
  title={Design of provably physical-constraint-preserving methods for general relativistic hydrodynamics},
  author={Wu, Kailiang},
  journal={Phys. Rev. D},
  volume={95},
  number={10},
  pages={103001},
  year={2017},
  publisher={APS}
}

@article{abgrall2024bound,
  title={{Bound preserving Point-Average-Moment PolynomiAl-interpreted (PAMPA)} scheme: one-dimensional case},
  author={Abgrall, R{\'e}mi and Jiao, Miaosen and Liu, Yongle and Wu, Kailiang},
  journal={Commun. Comput. Phys.},
  volume={39},
  number={5},
  pages={29--58},
  year={2026},
  publisher={Cambridge University Press}
}

@article{abgrall2025bound,
  title={{Bound preserving Point-Average-Moment PolynomiAl-interpreted (PAMPA) on polygonal meshes}},
  author={Abgrall, R{\'e}mi and Liu, Yongle and Boscheri, Walter},
  journal={},
  year={2025},
  note={\href{https://arxiv.org/abs/2502.10069}{https://arxiv.org/abs/2502.10069}}
}

@article{abgrall2025novel,
  title={A novel and simple invariant-domain-preserving framework for {PAMPA scheme: 1D} case},
  author={Abgrall, R{\'e}mi and Jiao, Miaosen and Liu, Yongle and Wu, Kailiang},
  journal={SIAM J. Sci. Comput.},
  volume={47},
  number={6},
  pages={A3536--A3565},
  year={2025},
  publisher={SIAM}
}

@article{cui2025local,
  title={On Local Minimum Entropy Principle of High-Order Schemes for Relativistic {E}uler Equations},
  author={Cui, Shumo and Wu, Kailiang and Xu, Linfeng},
  journal={Math. Comp. },
  volume={in press},
  year={2025},
  note={\href{https://doi.org/10.1090/mcom/4139}{https://doi.org/10.1090/mcom/4139}} 
}

@article{liu1994weighted,
  title={Weighted essentially non-oscillatory schemes},
  author={Liu, Xu-Dong and Osher, Stanley and Chan, Tony},
  journal={J. Comput. Phys.},
  volume={115},
  number={1},
  pages={200--212},
  year={1994},
  publisher={Elsevier}
}

@article{jiang1996efficient,
  title={Efficient implementation of weighted {ENO} schemes},
  author={Jiang, Guang-Shan and Shu, Chi-Wang},
  journal={J. Comput. Phys.},
  volume={126},
  number={1},
  pages={202--228},
  year={1996},
  publisher={Elsevier}
}

@article{zalesak1979fully,
  title={Fully multidimensional flux-corrected transport algorithms for fluids},
  author={Zalesak, Steven T},
  journal={J. Comput. Phys.},
  volume={31},
  number={3},
  pages={335--362},
  year={1979},
  publisher={Elsevier}
}

@book{zalesak2012design,
  title={The design of {Flux-Corrected Transport} ({FCT}) algorithms for structured grids},
  author={Zalesak, Steven T},
  year={2012},
  publisher={Springer}
}

@book{kuzmin2012flux,
  title={Flux-corrected transport: principles, algorithms, and applications},
  author={Kuzmin, Dmitri and L{\"o}hner, Rainald and Turek, Stefan},
  year={2012},
  publisher={Springer Science \& Business Media}
}

@article{wilson1972numerical,
  title={{Numerical study of fluid flow in a Kerr space}},
  author={Wilson, James R},
  journal={Astrophys. J.},
  volume={173},
  pages={431},
  year={1972}
}

@article{may1966hydrodynamic,
  title={Hydrodynamic calculations of general-relativistic collapse},
  author={May, Michael M and White, Richard H},
  journal={Phys. Rev.},
  volume={141},
  number={4},
  pages={1232},
  year={1966},
  publisher={APS}
}

@article{chen2021second,
	title={Second-order accurate {BGK} schemes for the special relativistic hydrodynamics with the {S}ynge equation of state},
	author={Chen, Yaping and Kuang, Yangyu and Tang, Huazhong},
	journal={J. Comput. Phys.},
	volume = {442},
	pages={110438},
	year={2021},
	publisher={Elsevier}
}

@article{radice2012thc,
	title={{THC: a new high-order finite-difference high-resolution shock-capturing code for special-relativistic hydrodynamics}},
	author={Radice, David and Rezzolla, Luciano},
	journal={Astron. Astrophys.},
	volume={547},
	pages={A26},
	year={2012},
	publisher={EDP Sciences}
}

@article{2006Zhang,
  title={{RAM: A relativistic adaptive mesh refinement hydrodynamics code}},
  author={Zhang, Weiqun and MacFadyen, Andrew I},
  journal={Astrophys. J. Suppl. S.},
  volume={164},
  number={1},
  pages={255},
  year={2006},
  publisher={IOP Publishing}
}

@article{hughes2002three,
  title={Three-dimensional hydrodynamic simulations of relativistic extragalactic jets},
  author={Hughes, Philip A and Miller, Mark A and Duncan, G Comer},
  journal={Astrophys. J.},
  volume={572},
  number={2},
  pages={713},
  year={2002},
  publisher={IOP Publishing}
}

@article{zhang2012positivity,
  title={Positivity-preserving high order finite difference {WENO} schemes for compressible {E}uler equations},
  author={Zhang, Xiangxiong and Shu, Chi-Wang},
  journal={J. Comput. Phys.},
  volume={231},
  number={5},
  pages={2245--2258},
  year={2012},
  publisher={Elsevier}
}

@article{jiang2013parametrized,
  title={Parametrized maximum principle preserving limiter for finite difference {WENO} schemes solving convection-dominated diffusion equations},
  author={Jiang, Yi and Xu, Zhengfu},
  journal={SIAM J. Sci. Comput.},
  volume={35},
  number={6},
  pages={A2524--A2553},
  year={2013},
  publisher={SIAM}
}

@article{boris1973flux,
  title={Flux-corrected transport. {I. SHASTA}, a fluid transport algorithm that works},
  author={Boris, Jay P and Book, David L},
  journal={J. Comput. Phys.},
  volume={11},
  number={1},
  pages={38--69},
  year={1973},
  publisher={Elsevier}
}

@article{book1975flux,
  title={Flux-corrected transport {II: G}eneralizations of the method},
  author={Book, David L and Boris, Jay P and Hain, K},
  journal={J. Comput. Phys.},
  volume={18},
  number={3},
  pages={248--283},
  year={1975},
  publisher={Elsevier}
}

@article{boris1976flux,
  title={Flux-corrected transport. {III. Minimal-error FCT} algorithms},
  author={Boris, Jay P and Book, David L},
  journal={J. Comput. Phys.},
  volume={20},
  number={4},
  pages={397--431},
  year={1976},
  publisher={Elsevier}
}

@article{kuzmin2001positive,
  title={Positive finite element schemes based on the flux-corrected transport procedure},
  author={Kuzmin, D},
  journal={Computational Fluid and Solid Mechanics, Elsevier},
  pages={887--888},
  year={2001}
}

@article{kuzmin2004high,
  title={High-resolution {FEM--FCT} schemes for multidimensional conservation laws},
  author={Kuzmin, Dimitri and M{\"o}ller, Matthias and Turek, Stefan},
  journal={Comput. Methods Appl. Mech. Engrg.},
  volume={193},
  number={45-47},
  pages={4915--4946},
  year={2004},
  publisher={Elsevier}
}

@article{kuzmin2020monolithic,
  title={Monolithic convex limiting for continuous finite element discretizations of hyperbolic conservation laws},
  author={Kuzmin, Dmitri},
  journal={Comput. Methods Appl. Mech. Engrg.},
  volume={361},
  pages={112804},
  year={2020},
  publisher={Elsevier}
}

@article{wu2025high,
  title={High Order Numerical Methods Preserving Invariant Domain for Hyperbolic and Related Systems},
  author={Wu, Kailiang and Zhang, Xiangxiong and Shu, Chi-Wang},
  journal={SIAM Rev.},
  volume={in press},
  year={2025},
  note={\href{https://doi.org/10.48550/arXiv.2512.09116}{https://doi.org/10.48550/arXiv.2512.09116}}
}

@article{duan2021entropy,
  title={Entropy stable adaptive moving mesh schemes for {2D and 3D} special relativistic hydrodynamics},
  author={Duan, Junming and Tang, Huazhong},
  journal={J. Comput. Phys.},
  volume={426},
  pages={109949},
  year={2021},
  publisher={Elsevier}
}

@article{cao2025robust,
  title={Robust discontinuous {G}alerkin methods maintaining physical constraints for general relativistic hydrodynamics},
  author={Cao, Huihui and Peng, Manting and Wu, Kailiang},
  journal={J. Comput. Phys.},
  volume={526},
  pages={113770},
  year={2025},
  publisher={Elsevier}
}

@article{cai2024provably,
  title={Provably convergent Newton--Raphson methods for recovering primitive variables with applications to physical-constraint-preserving Hermite WENO schemes for relativistic hydrodynamics},
  author={Cai, Chaoyi and Qiu, Jianxian and Wu, Kailiang},
  journal={J. Comput. Phys.},
  volume={498},
  pages={112669},
  year={2024},
  publisher={Elsevier}
}

@article{suresh1997accurate,
  title={Accurate monotonicity-preserving schemes with {Runge--Kutta} time stepping},
  author={Suresh, Ambady and Huynh, Hung T},
  journal={J. Comput. Phys.},
  volume={136},
  number={1},
  pages={83--99},
  year={1997},
  publisher={Elsevier}
}

@article{he2012adaptive1,
	title={An adaptive moving mesh method for two-dimensional relativistic hydrodynamics},
	author={He, Peng and Tang, Huazhong},
	journal={Commun. Comput. Phys.},
	volume={11},
	number={1},
	pages={114--146},
	year={2012},
	publisher={Cambridge University Press}
}

@article{chen2022physical,
  title={{A physical-constraint-preserving finite volume WENO method for special relativistic hydrodynamics on unstructured meshes}},
  author={Chen, Yaping and Wu, Kailiang},
  journal={J. Comput. Phys.},
  volume={466},
  pages={111398},
  year={2022},
  publisher={Elsevier}
}

@article{xu2024high,
  title={High-Order Accurate Entropy Stable Schemes for Relativistic Hydrodynamics with General Synge-Type Equation of State},
  author={Xu, Linfeng and Ding, Shengrong and Wu, Kailiang},
  journal={J. Sci. Comput.},
  volume={98},
  number={2},
  pages={43},
  year={2024},
  publisher={Springer}
}

@article{wu2023geometric,
  title={Geometric quasilinearization framework for analysis and design of bound-preserving schemes},
  author={Wu, Kailiang and Shu, Chi-Wang},
  journal={SIAM Rev.},
  volume={65},
  number={4},
  pages={1031--1073},
  year={2023},
  publisher={SIAM}
}

@article{WuMEP2021,
	title={Minimum principle on specific entropy and high-order accurate invariant region preserving numerical methods for relativistic hydrodynamics},
	author={Wu, Kailiang},
	journal={SIAM J. Sci. Comput.},
	volume={43},
	number={6},
	pages={B1164--B1197},
	year={2021},
	publisher={SIAM}
}

@article{marquina2019capturing,
  title={Capturing Composite Waves in Non-convex Special Relativistic Hydrodynamics},
  author={Marquina, Antonio and Serna, Susana and Ib{\'a}{\~n}ez, Jos{\'e} M},
  journal={J. Sci. Comput.},
  volume={81},
  number={3},
  pages={2132--2161},
  year={2019},
  publisher={Springer}
}

@article{duan2019high,
  title={High-order accurate entropy stable finite difference schemes for one-and two-dimensional special relativistic hydrodynamics},
  author={Duan, Junming and Tang, Huazhong},
  journal = {Adv. Appl. Math. Mech.},
year = {2020},
volume = {12},
number = {1},
pages = {1--29}
}

@article{guermond2017invariant,
  title={Invariant domains and second-order continuous finite element approximation for scalar conservation equations},
  author={Guermond, Jean-Luc and Popov, Bojan},
  journal={SIAM J. Numer. Anal.},
  volume={55},
  number={6},
  pages={3120--3146},
  year={2017},
  publisher={SIAM}
}

@article{guermond2018second,
  title={Second-order invariant domain preserving approximation of the {Euler} equations using convex limiting},
  author={Guermond, Jean-Luc and Nazarov, Murtazo and Popov, Bojan and Tomas, Ignacio},
  journal={SIAM J. Sci. Comput.},
  volume={40},
  number={5},
  pages={A3211--A3239},
  year={2018},
  publisher={SIAM}
}

@article{guermond2019invariant,
  title={Invariant domain preserving discretization-independent schemes and convex limiting for hyperbolic systems},
  author={Guermond, Jean-Luc and Popov, Bojan and Tomas, Ignacio},
  journal={Comput. Methods Appl. Mech. Eng.},
  volume={347},
  pages={143--175},
  year={2019},
  publisher={Elsevier}
}

@article{Hu2013,
  title={Positivity-preserving method for high-order conservative schemes solving compressible {Euler} equations},
  author={Hu, Xiangyu Y and Adams, Nikolaus A and Shu, Chi-Wang},
  journal={J. Comput. Phys.},
  volume={242},
  pages={169--180},
  year={2013},
  publisher={Elsevier}
}

@article{Liang2014,
  title={Parametrized maximum principle preserving flux limiters for high order schemes solving multi-dimensional scalar hyperbolic conservation laws},
  author={Liang, Chao and Xu, Zhengfu},
  journal={J. Sci. Comput.},
  volume={58},
  number={1},
  pages={41--60},
  year={2014},
  publisher={Springer}
}

@article{Marti2015,
  title={Grid-based methods in relativistic hydrodynamics and magnetohydrodynamics},
  author={Mart{\'\i}, Jos{\'e} Mar{\'\i}a and M{\"u}ller, Ewald},
  journal={Living Rev. Comput. Astrophys.},
  volume={1},
  number={1},
  pages={3},
  year={2015},
  publisher={Springer}
}

@article{QinShu2016,
title = "Bound-preserving discontinuous {Galerkin} methods for relativistic hydrodynamics",
journal = "J. Comput. Phys.",
volume = "315",
pages = "323--347",
year = "2016",
author = "Tong Qin and Chi-Wang Shu and Yang Yang"
}

@book{rezzolla2013relativistic,
  title={Relativistic Hydrodynamics},
  author={Rezzolla, Luciano and Zanotti, Olindo},
  year={2013},
  publisher={Oxford University Press}
}

@article{Wu2017a,
  author =	 {K. Wu},
  title =	 {Positivity-preserving analysis of numerical schemes for ideal magnetohydrodynamics},
    journal = {SIAM J. Numer. Anal.},
	volume = {56},
number = {4},
pages = {2124--2147},
	year = {2018}
 }

@article{WuShu2018,
  author =	 {K. Wu and C.-W. Shu},
  title =	 {A provably positive discontinuous {Galerkin} method for multidimensional ideal magnetohydrodynamics},
  journal = {SIAM J. Sci. Comput.},
  volume={40},
  number={5},
  pages={B1302--B1329},
  year={2018},
  publisher={SIAM}
}

@Article{WuShu2019,
author="Wu, Kailiang and Shu, Chi-Wang",
title="Provably positive high-order schemes for ideal magnetohydrodynamics: analysis on general meshes",
journal="Numer. Math.",
year="2019",
month="Aug",
day="01",
volume="142",
number="4",
pages="995--1047"
}

@article{WuShu2019SISC,
  title={Entropy symmetrization and high-order accurate entropy stable numerical schemes for relativistic {MHD} equations},
  author={Wu, Kailiang and Shu, Chi-Wang},
  journal={SIAM J. Sci. Comput.},
volume = {42},
number = {4},
pages = {A2230--A2261},
year = {2020}
}

@article{WuShu2020NumMath,
  title={Provably Physical-Constraint-Preserving Discontinuous {Galerkin} Methods for Multidimensional Relativistic {MHD} Equations},
  author={Wu, Kailiang and Shu, Chi-Wang},
  journal={Numer. Math.},
  year={2021},
  doi={10.1007/s00211-021-01209-4}
}

@article{WuTang2015,
  title={High-order accurate physical-constraints-preserving finite difference {WENO} schemes for special relativistic hydrodynamics},
  author={Wu, Kailiang and Tang, Huazhong},
  journal={J. Comput. Phys.},
  volume={298},
  pages={539--564},
  year={2015},
  publisher={Elsevier}
}

@article{WuTangM3AS,
  title={Admissible states and physical-constraints-preserving schemes for relativistic magnetohydrodynamic equations},
  author={Wu, Kailiang and Tang, Huazhong},
  journal={Math. Models Methods Appl. Sci.},
  volume={27},
  number={10},
  pages={1871--1928},
  year={2017},
  publisher={World Scientific}
}

@article{WuTang2017ApJS,
  title={Physical-constraint-preserving central discontinuous {Galerkin} methods for special relativistic hydrodynamics with a general equation of state},
  author={Wu, Kailiang and Tang, Huazhong},
  journal={Astrophys. J. Suppl. Ser.},
  volume={228},
  number={1},
  eid={3},
  year={2017},
  publisher={IOP Publishing}
}

@article{Xiong2016,
  title={Parametrized positivity preserving flux limiters for the high order finite difference {WENO} scheme solving compressible {Euler} equations},
  author={Xiong, Tao and Qiu, Jing-Mei and Xu, Zhengfu},
  journal={J. Sci. Comput.},
  volume={67},
  number={3},
  pages={1066--1088},
  year={2016},
  publisher={Springer}
}

@article{Xu2014,
  title={Parametrized maximum principle preserving flux limiters for high order schemes solving hyperbolic conservation laws: one-dimensional scalar problem},
  author={Xu, Zhengfu},
  journal={Math. Comp. },
  volume={83},
  number={289},
  pages={2213--2238},
  year={2014}
}

@article{zhang2010,
  title={On maximum-principle-satisfying high order schemes for scalar conservation laws},
  author={Zhang, Xiangxiong and Shu, Chi-Wang},
  journal={J. Comput. Phys.},
  volume={229},
  number={9},
  pages={3091--3120},
  year={2010},
  publisher={Elsevier}
}

@article{zhang2010b,
  title={On positivity-preserving high order discontinuous {Galerkin} schemes for compressible {Euler} equations on rectangular meshes},
  author={Zhang, Xiangxiong and Shu, Chi-Wang},
  journal={J. Comput. Phys.},
  volume={229},
  number={23},
  pages={8918--8934},
  year={2010},
  publisher={Elsevier}
}

@article{ZHANG2017301,
title = "On positivity-preserving high order discontinuous {Galerkin} schemes for compressible {Navier-Stokes} equations",
journal = "J. Comput. Phys.",
volume = "328",
pages = "301--343",
year = "2017",
author = "Xiangxiong Zhang"
}

@article{radice2011discontinuous,
  title={Discontinuous {Galerkin} methods for general-relativistic hydrodynamics: formulation and application to spherically symmetric spacetimes},
  author={Radice, David and Rezzolla, Luciano},
  journal={Phys. Rev. D},
  volume={84},
  number={2},
  pages={024010},
  year={2011},
  publisher={APS}
}

@article{zhao2013runge,
  title={{Runge--Kutta} discontinuous {Galerkin} methods with {WENO} limiter for the special relativistic hydrodynamics},
  author={Zhao, Jian and Tang, Huazhong},
  journal={J. Comput. Phys.},
  volume={242},
  pages={138--168},
  year={2013},
  publisher={Elsevier}
}

@article{marti2003numerical,
  title={Numerical hydrodynamics in special relativity},
  author={Mart{\'\i}, Jos{\'e} Maria and M{\"u}ller, Ewald},
  journal={Living Rev. Relativ.},
  volume={6},
  number={1},
  pages={7},
  year={2003},
  publisher={Springer}
}

@Misc{amsmath,
  author =	 {{American Mathematical Society}},
  title =	 {User's Guide for the \texttt{amsmath} Package
                  (Version 2.0)},
  url =		 {ftp://ftp.ams.org/pub/tex/doc/amsmath/amsldoc.pdf},
  urldate =	 {2015-07-30},
  year =	 2002}

@Misc{pgfplots,
  author =	 {Christian Feuers\"anger},
  title =	 {Manual for Package \texttt{PGFPLOTS}},
  month =	 may,
  year =	 2015,
  url =		 {http://sourceforge.net/projects/pgfplots}
}
 
\end{document}